\documentclass[11pt,a4paper,leqno]{amsart}

\usepackage{amsmath}
\usepackage{amssymb,mathrsfs}
\usepackage{bm,bbm} %add characteristic 
\usepackage{amsfonts}
\usepackage{enumitem}
\usepackage{amsthm}
\newtheorem{definition}{Definition}[section]
\newtheorem{theorem}[definition]{Theorem}
\newtheorem{proposition}[definition]{Proposition}
\newtheorem{remark}[definition]{Remark}
\newtheorem{corollary}[definition]{Corollary}
\newtheorem{lemma}[definition]{Lemma}

\usepackage{indentfirst}
\usepackage{multirow}
\numberwithin{equation}{section}

\usepackage[left=3.5cm,right=3.5cm,top=3cm,bottom=3cm]{geometry}

\title{Global Well-posedness of a 2D fluid-Structure interaction problem with free surface}

\author{Thomas Alazard}
\address{CNRS - Centre de Mathématiques Laurent Schwartz, {\'E}cole Polytechnique, 
Institut Polytechnique de Paris}
\email{thomas.alazard@polytechnique.edu}

\author{Chengyang Shao}
\address{Department of Mathematics, University of Chicago}
\email{shaoc@uchicago.edu}

\author{Haocheng Yang}
\address{Centre Borelli, {\'E}cole normale supérieure Paris-Saclay}
\email{haocheng.yang@universite-paris-saclay.fr}

\usepackage[citestyle=numeric,maxbibnames=9, maxcitenames=9,doi=false,isbn=false,url=false,giveninits=true,bibstyle=numeric]{biblatex}
\renewbibmacro{in:}{}
\AtEveryBibitem{%
\clearlist{language}%
}
\usepackage[colorlinks,linkcolor=blue,citecolor=red,urlcolor=black]{hyperref}

\usepackage{verbatim} %add comments 

\DeclareMathOperator{\RE}{Re}

\DeclareMathOperator{\IM}{Im}
\DeclareMathOperator{\supp}{supp}

\DeclareMathOperator{\diff}{d}
\DeclareMathOperator{\sgn}{sign}

\begin{document}

\pagenumbering{arabic}

\newcommand{\BMO}{\mathrm{BMO}}
\newcommand{\Op}[2][]{\operatorname{Op}^{#1}\left(#2\right)}
\newcommand{\Supp}[1]{\operatorname{Supp} #1}
\newcommand{\Real}{\operatorname{Re}}
\newcommand{\Imag}{\operatorname{Im}}
\newcommand{\Id}{\mathrm{Id}}
\newcommand{\G}{\mathcal{G}}
\newcommand{\B}{\mathcal{B}}
\newcommand{\Hi}{\mathcal{H}}
\newcommand{\V}{\mathcal{V}}
\newcommand{\diver}{\operatorname{div}}
\newcommand{\curl}{\operatorname{curl}}
\newcommand{\tr}[1]{\operatorname{tr}\left(#1\right)}

\newcommand{\R}{\mathbb{R}}
\newcommand{\C}{\mathbb{C}}
\newcommand{\N}{\mathbb{N}}
\newcommand{\T}{\mathbb{T}}
\newcommand{\Z}{\mathbb{Z}}

\renewcommand{\le}{\leqslant}
\renewcommand{\ge}{\geqslant}

\def\dtau{\diff \! \tau}
\def\dt{\diff \! t}
\def\dx{\diff \! x}
\def\dz{\diff \! z}
\def\dzeta{\diff \! \zeta}
\def\dX{\diff \! X}
\def\dxdt{\dx \dt}
\def\dxi{\diff \! \xi}
\def\dy{\diff \! y}
\def\dydt{\diff \! y \diff \! t}
\def\dydx{\diff \! y \diff \! x}
\def\dydxdt{\diff \! y \diff \! x \diff \! t}
\def\fract{\frac{\diff}{\dt}}
\def\fraceps{\frac{\diff}{\diff\!\eps}}	

\def\Dx{\la D_x\ra}

\def\FU{\mathcal{F}(U)}

\numberwithin{equation}{section}
\pagestyle{plain}

\def\defn{\mathrel{:=}}
\def\Deltax{\Delta}
\def\Deltayx{\Delta_{x,y}}
\def\Deltazx{\Delta_{x,z}}
\def\eps{\varepsilon}
\def\Hs{\mathcal{H}_\sigma}
\def\Hsi{\mathcal{H}_{\sigma+\mez}}
\def\la{\left\lvert}
\def\lA{\left\lVert}
\def\le{\leq}
\def\ge{\geq}
\def\L#1{\langle{#1}\rangle}
\def\mez{\frac{1}{2}}
\def\partialx{\nabla}
\def\partialyx{\nabla_{x,y}}
\def\pxz{\nabla_{x,z}}
\def\ra{\right\rvert}
\def\rA{\right\rVert}
\def\s{\sigma}
\def\slam{a}
\def\Slam{A}
\def\tdm{\frac{3}{2}}
\def\tq{\frac{3}{4}}
\def\uq{1/4}
\def\xC{\mathbf{C}}
\def\xN{\mathbf{N}}

\maketitle

\begin{abstract}
This paper is devoted to the analysis of the incompressible Euler equation in a time-dependent fluid domain, whose interface evolution is governed by the law of linear elasticity. Our main result asserts that the Cauchy problem is globally well-posed in time for any irrotational initial data in the energy space, without any smallness assumption. We also prove continuity with respect to the initial data and the propagation of regularity. The main novelty is that no dissipative effect is assumed in the system. In the absence of parabolic regularization, the key observation is that the system can be transformed into a nonlinear Schr\"odinger-type equation, to which dispersive estimates are applied. This allows us to construct solutions that are very rough from the point of view of fluid dynamics—the initial fluid velocity has merely one-half derivative in $L^2$. The main difficulty is that the problem is critical in the energy space with respect to several key inequalities from harmonic analysis. The proof incorporates new estimates for the Dirichlet-to-Neumann operator in the low-regularity regime, including refinements of paralinearization formulas and shape derivative formulas, which played a key role in the analysis of water waves.
\end{abstract}

\section{Introduction}\label{sect:Intro}
\subsection{Description of the problem}
Consider a one-dimensional free surface, that is, a time-dependent curve \(\Gamma\), given as the graph of some function \(\eta\), so that at time \(t \geq 0\),  
\begin{equation*}
\Gamma(t) := \{ (x,\eta(t,x)): x\in\mathbb{R} \}.
\end{equation*}  
We are interested in specific free boundary problems arising from fluid-structure interactions, whose charm lies in the fact that \(\eta\) satisfies {\em two} evolution equations, both of which are classical equations from mathematical physics.

The first one is related to fluid dynamics, postulating that $\Gamma(t)$ is transported by a fluid occupying the region $\Omega(t)\subset \R^2$ located beneath $\Gamma(t)$, that is
\begin{equation*}
\Omega(t) := \{ (x,y) \in \R \times \R : y < \eta(t,x) \}.
\end{equation*}  
More precisely, we assume that $\Omega$ is occupied by a two-dimensional incompressible, irrotational ideal fluid and that the normal velocity of the fluid must coincide with the normal component of the velocity of the boundary material. This is equivalent to the so-called \emph{kinematic boundary condition}:
\begin{equation}\label{eq-intro-fsi:KinBdyCond}
	\partial_t\eta 
	= \sqrt{1+\eta_x^2} \, \bm{v}|_{y=\eta} \cdot \bm{n}_\eta,
\end{equation}
where \(\bm{v}\colon \Omega\to\R^2\) is the velocity field,  $\eta_x=\partial_x\eta$ and 
$$
\bm{n}_\eta = \frac{1}{\sqrt{1+\eta_x^2}}\begin{pmatrix}
	-\eta_x \\
	1
	\end{pmatrix}
$$
is the outward unit normal vector field to $\Gamma(t)$.

The second evolution equation for \(\eta\) expresses a balance of stress across the free surface. 
In our case, we assume that the boundary is covered by an elastic one-dimensional material that 
obeys the law of \emph{linear elasticity}, so that the stress balance equation along the boundary is  
\begin{equation}\label{eq-intro-fsi:Plate}
\rho_{\text{m}}\partial_t^2 \eta + \sigma\partial_x^4\eta=P|_{y=\eta},
\end{equation}  
where $P\colon \Omega\to\R$ represents the pressure distribution within the fluid domain (while the exterior pressure is normalized to be zero); $\rho_{\text{m}}$ is the density of the material and $\sigma$ is the flexural rigidity, both assumed to be constant. Therefore, the stress tensor of the material obeys the similar assumptions as in classical \emph{Euler-Bernoulli beam theory}.	

The main result of this paper is a global well-posedness result for the case where~\(\bm{v}\) is related to the pressure \(P\) through the incompressible Euler equations:
\begin{equation}\label{eq-intro-fsi:Euler}
\left\{\begin{aligned}
    &\partial_t \bm{v} + (\bm{v}\cdot\nabla_{x,y})\bm{v} + \frac{1}{\rho_{\text{f}}}\nabla_{x,y}P 
    = \begin{pmatrix}
	0 \\
	-g_0
	\end{pmatrix} 
	&& \text{in }\Omega(t), \\[0.5ex]
    &\diver_{x,y}\bm{v} = 0,\ \ \curl_{x,y}\bm{v} = 0&& \text{in }\Omega(t).
\end{aligned}\right.
\end{equation}  
Here $\rho_{\text{f}}$ is the density of the fluid (which is a constant) and $g_0$ is the gravitational acceleration. One difficulty is that we will work with solutions for which the velocity field is merely $L^4$ in time with values in $L^\infty_{x,y}$.

The Cauchy problem has already attracted significant attention in cases where the velocity $\bm{v}$ and pressure $P$ satisfy the Navier-Stokes equations, where damping, dissipation, or viscosity plays a crucial role. Among these works, the main differences arise from the choice of the elastic structure model. Three frequently considered models are:  
(1) the linear elastic model,  
(2) the Willmore model, and  
(3) the Koiter shell model.  
As a result, most previous works have considered systems that are either parabolic or parabolic-hyperbolic with additional damping.

The first model, which is also adopted in the present paper (see \eqref{eq-intro-fsi:Plate}), is the simplest one and accurately captures the interface behavior near the reference manifold (e.g., plane, cylinder, etc.). The study of the Cauchy problem traces back to the pioneering work~\cite{MR2166981}, in which the authors examined the existence (but not the uniqueness) of a global weak solution for an unsteady fluid-structure interaction problem governed by the Navier–Stokes equations, interacting with a flexible elastic plate governed by the linear beam equation with damping. See also~\cite{MR3466847} for the 2D case, where the authors focused on the problem with periodic boundary conditions and constructed a unique global strong solution. 

In~\cite{GHL}, the analysis of the local well-posedness theory was extended to more general linear models. Another key aspect of the problem is addressing the case where the structure is governed by a weakly damped wave equation, as originally considered by Lions in~\cite{lions1969livre} as an example of a parabolic-hyperbolic system. This problem was investigated in~\cite{MR3604365}, which establishes global well-posedness for small data in the three-dimensional case with damping. When the fluid domain is nearly cylindrical, the existence of a local-in-time weak solution was then established for both 2D and 3D cases in~\cite{MC1,MC2}.

The second model, which incorporates the curvature of the interface, originates from differential geometry and plays a crucial role in the study of biomembranes (see, for instance,~\cite{ouyang1990membranes}). The study of this problem was initiated in~\cite{MR2349865}, using a stationary equation derived from Willmore energy to describe the state of the structure, with the fluid governed by the 3D Navier-Stokes equations in a bounded domain, and local well-posedness was established. The non-stationary case was later analyzed in~\cite{MR4039524} for the 3D periodic setting, where global well-posedness for small data was proven.  

\smallskip

The third model, originating from engineering, accounts for both curvature and membrane strain. In the case of a general bounded fluid domain, local well-posedness was proved in~\cite{MR2644917}. A similar scenario was examined in~\cite{L,LR}, where the authors considered a general stress tensor and established the existence of weak solutions, which were either global in time or exhibited self-intersection in finite time. Another particularly interesting setting involves fluid contained inside a cylindrical container, with the interface remaining close to a planar configuration. This case was recently studied in~\cite{MR4450291,schwarzacher-su2023regularity}, where the authors proved weak-strong uniqueness in 3D, along with the local existence and uniqueness of strong solutions for sufficiently regular data in 2D. Cylindrical fluid domains were also analyzed in~\cite{MRR,MR4540753}, where global well-posedness in 3D was established under the condition that no self-intersection occurs.  

\smallskip

Beyond these three models, more nonlinear formulations have attracted significant mathematical interest. For instance, in~\cite{BKS2023}, the authors examined a 3D fluid contained within a cylindrical domain and proved the existence of weak solutions. The problem was formulated using the first Piola-Kirchhoff stress tensor, which describes a broader class of nonlinear elasticity compared to the three models mentioned above. Overall, the study of fluid-structure interactions is a vast and evolving field, and the discussion here provides only a limited overview. For a more comprehensive account of recent advancements, we refer readers to~\cite{zbMATH05875759,MR4188809,MR2922368,PT2011}.  

\smallskip

Let us emphasize that the main difficulty in our study lies in the fact that we consider a system (see \eqref{eq-intro-fsi:Plate}-\eqref{eq-intro-fsi:Euler}) that includes neither damping nor dissipative (viscosity) effects. The corresponding Cauchy problem was previously investigated in~\cite{AKT2024}, where the existence (but not uniqueness) of a global weak solution was established in two space dimensions for a fluid periodic in the horizontal variable $x$, in the general case without damping, for any initial data with finite energy.

\smallskip

In this paper, we go much further by proving that the Cauchy problem is globally well-posed in time in the setting of the whole real line instead of torus. The lack of damping or dissipative effects makes arguments for parabolic PDEs inappropriate, significantly complicating the analysis. While the main challenge in~\cite{AKT2024} was to give meaning to the equations for solutions with finite energy, we extend this by transforming the equations into a Schrödinger-type equation. 

\smallskip

More precisely, to reveal the central feature of the phenomenon under consideration, we perform a reduction to the boundary. To explain this, we set $\psi$ to be the trace of the velocity potential of $\bm{v}$ on the boundary, and introduce the complex-valued unknown
$$
U:=q(D_x)\eta + i (\psi+\partial_t\eta).
$$
We will prove that there exist two (nontrivial) Fourier multipliers, $p(D_x)$ and $q(D_x)$, which behave as $\partial_x^2$ at high frequencies, so that
$$
		\partial_t U + i p(D_x)U = \mathcal{F}(U),
$$
where $\mathcal{F}(U)$ denotes a nonlinear, nonlocal operator acting on $U$. 

These transformations are particularly delicate to implement at the regularity level corresponding to initial data with finite energy, without any smallness assumptions. Indeed, the problem is critical in the energy space with respect to several key inequalities from harmonic analysis, which play a central role in the proof. In particular, we will establish several results of independent interest related to the study of the Dirichlet-to-Neumann operator.

\subsection{Main result} 
We will establish three properties of the system (\ref{eq-intro-fsi:KinBdyCond})–(\ref{eq-intro-fsi:Plate}):
\begin{enumerate}
\item The system (\ref{eq-intro-fsi:KinBdyCond})–(\ref{eq-intro-fsi:Plate}) is equivalent to a Schr\"odinger-type nonlinear evolution equation; this reduction is done by paradifferential calculus.
\item The Cauchy problem for (\ref{eq-intro-fsi:KinBdyCond})–(\ref{eq-intro-fsi:Plate}) is globally well-posed in time for any initial data with finite energy.
\item The system (\ref{eq-intro-fsi:KinBdyCond})–(\ref{eq-intro-fsi:Plate}) allows for the propagation of the regularity of the initial data.  
\end{enumerate}
To achieve these goals, following~\cite{AKT2024}, we adopt a formulation that reduces the problem to a set of equations defined on the free surface. This formulation involves the displacement function $\eta$ and the trace $\psi$ of the potential flow on the free surface. The approach is deeply rooted in a rich historical context, with seminal contributions from Zakharov~\cite{Zakharov1968}, Craig and Sulem~\cite{CrSu}, and Lannes~\cite{lannes2005water}, who pioneered the analysis of water wave equations through the framework of the Dirichlet-to-Neumann operator. As in~\cite{AKT2024}, we adopt this perspective and employ the \emph{Craig-Sulem-Zakharov formulation} for (\ref{eq-intro-fsi:KinBdyCond})-(\ref{eq-intro-fsi:Plate}). 

We begin by recalling the derivation of these equations. At this stage, we perform some formal computations, which will be rigorously justified later. We emphasize that the paper is self-contained, as our analysis requires a thorough re-examination of the Dirichlet-to-Neumann operator. Specifically, we will revisit its definition, its differential with respect to $\eta$, and its microlocal properties.

It is well known that the incompressible, irrotational 
fluid occupying \(\Omega(t)\) can be completely described by the \emph{velocity potential function} \(\Psi(t,x,y)\), such that  
\begin{equation}\label{eq-intro-csz:HarExt}
\bm{v} = \nabla_{x,y}\Psi,\ \ \Delta_{x,y}\Psi = 0.
\end{equation}  
The velocity potential \(\Psi(t,x,y)\) is a harmonic function and is unique up to an addendum that depends only on \(t\). By normalizing \(\nabla_{x,y}\Psi(t,x,y) \to 0\) as \(y \to -\infty\), we find that \(\Psi(t,x,y)\) is uniquely determined by its boundary value  
\begin{equation}\label{eq-intro-csz:psi}
\psi(t,x) = \Psi(t,x,y)|_{y=\eta(t,x)}
= \Psi(t,x,\eta(t,x)).
\end{equation}  
Let us now introduce the \emph{Dirichlet-to-Neumann operator}: if the boundary value \(\psi(t,x) = \Psi(t,x,y)|_{y=\eta}\) is given, then $\Psi$ is uniquely determined, so the following function is well-defined:
$$
\mathcal{G}(\eta)\psi := \sqrt{1+\eta_x^2}\,  \frac{\partial\Psi}{\partial \bm{n}_\eta}\Big|_{y=\eta},
$$  
where $\partial\Psi/\partial \bm{n}_\eta=\bm{n}_\eta\cdot\nabla_{x,y}\Psi$ is the outward-pointing normal derivative of \(\Psi\) on the boundary. 
This is the formal definition (the rigorous one will be studied later on) 
of the Dirichlet-to-Neumann operator $\G(\eta)\psi$, which is a linear operator acting on $\psi$ that depends nonlinearly on the free surface elevation~$\eta$. 
Since the flow inside \(\Omega(t)\) is a potential one, we may integrate the velocity field \(v(t,x,y)\) along any path to obtain the Bernoulli equation:
$$
\partial_t\Psi+\frac{1}{2}|\nabla_{x,y}\Psi|^2+g_0y+\frac{1}{\rho_{\text{f}}}P
=0.
$$
Evaluating at the boundary $y=\eta(t,x)$, we find that (\ref{eq-intro-fsi:KinBdyCond})-(\ref{eq-intro-fsi:Plate}) becomes
\begin{equation}\label{eq-intro-csz:CSZ}
	\left\{\begin{aligned}
		&\partial_t \eta = \G(\eta)\psi, \\ 
		&\partial_t \psi + \left(g_0\eta + \frac{\rho_{\text{m}}}{\rho_{\text{f}}}\partial_t^2\eta + \frac{\sigma}{\rho_{\text{f}}}\partial_x^4 \eta\right)
		+ \frac{1}{2}(\partial_x\psi)^2 -\frac{1}{2}\frac{\left(\partial_x\eta\cdot\partial_x\psi + \G(\eta)\psi\right)^2}{1+(\partial_x\eta)^2} = 0, \\
		&(\eta,\psi)|_{t=0} = (\eta_0,\psi_0).
	\end{aligned}\right.
\end{equation}		
From now on, for simplicity in notation, we shall normalize $\rho_{\text{f}}=1$, $\rho_{\text{m}}=1$ and $\sigma=1$, $g_0=1$.

Compared to~\cite{AKT2024}, a technical difference is that we shall work with the unknown
\begin{equation}\label{eq-intro-csz:u-psi}
	u := \psi + \mathcal{G}(\eta)\psi = \psi + \partial_t\eta.
\end{equation}
We immediately rewrite (\ref{eq-intro-csz:CSZ}) into the following equivalent form:
\begin{equation}\label{eq-intro-csz:MainSys}
	\left\{\begin{array}{l}
		\partial_t \eta = \mathcal{G}(\eta)(\Id + \mathcal{G}(\eta))^{-1}u, \\ [0.5ex]
		\partial_t u + \partial_x^4 \eta + \eta + N(\eta,u) = 0, \\ [0.5ex]
		(\eta,u)|_{t=0} = (\eta_0,u_0),
	\end{array}\right.
\end{equation}
where the nonlinear term $N$ equals
\begin{equation}\label{eq-intro-csz:N}
    N(\eta,u) = \frac{1}{2}(\partial_x\psi)^2 -\frac{1}{2}\frac{\left(\partial_x\eta\cdot\partial_x\psi + \G(\eta)\psi\right)^2}{1+(\partial_x\eta)^2} ,
    \quad  
    \psi=(\Id + \mathcal{G}(\eta))^{-1}u.
\end{equation}
Such formulation is convenient for two reasons. The most important one is that \eqref{eq-intro-csz:MainSys} has an explicit Hamiltonian structure: the total energy
\begin{equation}\label{E}
\frac{1}{2} 
\int_\R u\cdot\G(\eta)\big(\Id+\G(\eta)\big)^{-1}u \dx 
+ \frac{1}{2}\int_\R \big((\partial_x^2\eta)^2 + \eta^2\big) \dx
\end{equation}
serves as the Hamiltonian function, so that $\eta$ and $u$ are conjugate variables. On the other hand, in the low regularity regime, the time derivative $\partial_t^2\eta$ might be ill-defined pointwise, while $\partial_tu=\partial_t(\psi+\partial_t\eta)$ as a whole still makes sense.

We are at the place to state the main result of this paper. 
\begin{theorem}\label{thm:main}
For all initial data $(\eta_0,u_0)\in H^{2}\times L^2$, the Cauchy problem \eqref{eq-intro-csz:MainSys} has a unique, global in time solution $(\eta,u)$ in the space 
$$
C^0_{\mathrm{loc}}(\R;H^2\times L^2)
\cap L^4_{\mathrm{loc}}(\R;W^{2,\infty}\times L^{\infty}),
$$
and the solution map $(\eta_0,u_0)\mapsto (\eta,u)$ is 
    %Lipschitz continuous from $ H^{1}\times H^{-1}$ (\textcolor{red}{which space we should put here ? $ H^{1}\times H^{-1}$ or $ H^{2}\times L^2$}) to $C^0_{\mathrm{loc}}(\R;H^{1}\times H^{-1})$ and is 
    continuous from $ H^{2}\times L^2$ to $C^0_{\mathrm{loc}}(\R;H^2\times L^2)$. 
Moreover:
\begin{itemize}
   % \item The solution map $(\eta_0,u_0)\mapsto (\eta,u)$ is 
    %Lipschitz continuous from $ H^{1}\times H^{-1}$ (\textcolor{red}{which space we should put here ? $ H^{1}\times H^{-1}$ or $ H^{2}\times L^2$}) to $C^0_{\mathrm{loc}}(\R;H^{1}\times H^{-1})$ and is 
   % continuous from $ H^{2}\times L^2$ to $C^0_{\mathrm{loc}}(\R;H^2\times L^2)$.
    \item If the initial data $(\eta_0,u_0)$ belongs to $H^{s+2}\times H^s$ with $s>0$, then the solution $(\eta,u)$ of is of class
    $$
    C^0_{\mathrm{loc}}(\R;H^{s+2}\times H^s)
    \cap C^1_{\mathrm{loc}}(\R;H^{s}\times H^{s-2}),
    $$
    and the solution map $(\eta_0,u_0)\mapsto (\eta,u)$ is continuous from $H^{s+2}\times H^s$ to $C^0_{\mathrm{loc}}(\R;H^{s+2}\times H^s)$.
    \item If in addition $s\geq2$, then
    $$
    \begin{aligned}
    \eta&\in C^0_{\mathrm{loc}}(\R;H^{s+2})\cap C^1_{\mathrm{loc}}(\R;H^{s})\cap C^2_{\mathrm{loc}}(\R;H^{s-2}),\\
    \psi&=(\Id+\G(\eta))^{-1}u\in C^0_{\mathrm{loc}}(\R;H^{s+1})\cap C^1_{\mathrm{loc}}(\R;H^{s-2})
    \end{aligned}
    $$
    do give rise to the unique classical solution of the Cauchy problem (\ref{eq-intro-csz:CSZ}) in the corresponding spaces, and the total energy~\eqref{E} is conserved.
\end{itemize}
\end{theorem}

\subsection{Relation to nonlinear Schr\"{o}dinger equation}
We would like to point out an interesting \emph{dispersive} feature of \eqref{eq-intro-csz:MainSys} that closely relates to the \emph{nonlinear Schr\"{o}dinger equation} (NLS), and explain the crucial ingredients of the proof of Theorem \ref{thm:main}. It turns out that our proof makes essential expoitation of this dispersive feature. 

The linearization of~\eqref{eq-intro-csz:MainSys} around the null solution $(\eta,u)=(0,0)$ reads
\begin{equation}\label{eq-intro-stri:LinSys}
	\left\{\begin{array}{l}
		\partial_t \eta = \Dx(\Id+\Dx)^{-1}u, \\ [0.5ex]
		\partial_t u + \Dx^4 \eta + \eta = 0.
	\end{array}\right.
\end{equation}
Since the operator $\Dx(\Id+\Dx)^{-1}$ is not invertible, it is convenient to introduce a frequency cut-off. Specifically, 
consider the Fourier multiplier $\theta(D_x)$ whose 
symbol $\theta(\xi)\geq1/4$ is a smooth, even function that equals $1/4$ for $|\xi|\leq1/2$ and $|\xi|/(1+|\xi|)$ for $|\xi|>1$. We absorb the low frequency part $\big(\Dx(\Id+\Dx)^{-1}-\theta(D_x)\big)u$ into the source term. Upon this absorption, a natural idea is to reformulate the original system \eqref{eq-intro-csz:MainSys} as the summation of linear parts and source terms, namely
\begin{equation*}
	\left\{\begin{array}{l}
		\partial_t \eta - \theta(D_x)u = \text{source terms}, \\ [0.5ex]
		\partial_t u + (\Dx^4+1)\eta = \text{source terms}, \\ [0.5ex]
		(\eta,u)|_{t=0} = (\eta_0,u_0).
	\end{array}\right.
\end{equation*}

The next step is to introduce a new complex-valued unknown $U$ defined by
$$
U:= q(D_x)\eta + i u,\ \ q(D_x)=\theta(D_x)^{-1/2}\big(\Id+\Dx^4\big)^{1/2}.
$$
Since $q(D_x)$ is an invertible elliptic operator, one can easily recover $(\eta,u)$ from $U$ by separating the real and imaginary parts. Consequently, system (\ref{eq-intro-csz:MainSys}) is transformed into
\begin{equation}\label{eq-intro-stri:MainEq}
	\left\{\begin{array}{l}
		\partial_t U + i p(D_x)U = \mathcal{F}(U), \\ [0.5ex]
		U|_{t=0} = U_0 := q(D_x)\eta_0 + i u_0.
	\end{array}\right.
\end{equation}
Here $\mathcal{F}(U)$ denotes a nonlinear, nonlocal operator acting on $U$ (see \eqref{eq-para-schr:Sour} below), and the symbol $p$ is given by
$$	
p(\xi) =  \theta(\xi)^{1/2}\left(1+|\xi|^4\right)^{1/2}.
$$
Notice that $p(D_x)$ almost equals $\partial_x^2$ for high frequencies. Therefore, (\ref{eq-intro-stri:MainEq}) can be considered as a realization of the \emph{nonlinear Schr{\"o}dinger equation} (NLS), a well-invetigated model problem in the study for dispersive PDEs 
(see \cite{BaVe} for a presentation of many 
recent results concerning the Cauchy problem in the 1d case considered in this paper). The most thouroughly studied example is the power NLS
$$
\partial_tf+i\Delta_xf=\pm i|f|^\alpha f.
$$
Despite the mathematical parallelism, we observe a notable difference between the derivation of (\ref{eq-intro-stri:MainEq}) and the power NLS. The latter is usually considered as the formal profile approximation of nonlocal nonlinear waves; see Chapter 20 of \cite{TM2005}.  However, for the fluid-structure interaction considered in this paper, the Schr{\"o}dinger type equation (\ref{eq-intro-stri:MainEq}) is a genuine reformulation of the physical laws governing the system.

Of course, all the difficulty lies in the precise analysis of the nonlinearity $\mathcal{F}(U)$. Let us then describe the strategies implemented to analyze $\mathcal{F}(U)$ in this paper.

The first set of ingredients includes various estimates for the Dirichlet-Neumann operator in the \emph{low regularity regime}. In Theorem \ref{thm:main}, only spatial derivatives of $\eta$ up to order two are considered. 
Consequently, it is essential to analyze $\G(\eta)$ under 
these restrictive conditions. Our approach is based on several harmonic 
analysis techniques for indices that are critical in a sense that will be clarified later. The outcomes can be summarized as follows:
\begin{enumerate}[label=\textbf{(\arabic*)}]
\item\label{11} The operator $\mathcal{R}(\eta) = \G(\eta) - |D_x|$ is bounded and self-adjoint on $L^2$, \emph{even when $\eta$ has only $H^2 \cap W^{2,\infty}$ regularity}. This result is sharp in the sense that the proof does not appear to remain valid if one assumes only that $\eta$ belongs to $W^{2-\varepsilon,\infty}$. A formal explanation is that the remainder $\mathcal{R}(\eta)$ necessarily involves information of the \emph{curvature} of the interface, the pointwise bound of which is necessary for the $L^2$ bound of $\mathcal{R}(\eta)$. Therefore, $\eta\in W^{2,\infty}$ is indeed the threshold regularity assumption. The proof follows the approach of paralinearizing the Laplace equation into an elliptic evolution problem, a technique previously used in the study of the water-wave problem (see~\cite{AM2009,ABZ2014}). It relies on a new cancellation that allows considering shapes $\eta$ with mere $W^{2,\infty}$ regularity instead of $W^{2+\varepsilon,\infty}$. A precise statement is given in Proposition \ref{prop-reg-dn:Main}.

\item\label{12} The \emph{shape derivative for $\G(\eta)$}, namely, the linearization of the mapping $\eta\mapsto\G(\eta)$ along an increment $\delta\eta$, can be bounded using only the $H^1$ regularity of $\delta\eta$, provided that $\eta\in H^2\cap W^{2,\infty}$. The proof relies on the $L^2$ boundedness of $\mathcal{R}(\eta)$, and also the \emph{end point case} of estimates for the commutator $[|D_x|,A]f$ with $A\in H^1$, $f\in L^\infty$. To prove this end point commutator estimate, standard facts about Carleson measure are needed. Precise statements are given in Lemma \ref{Commu_Esti}, Proposition \ref{thm-reg-shpder:Main} and \ref{prop-reg-shpder:sharp}. 
\end{enumerate}

Another crucial fact is the following observation: even if \( \mathcal{R}(\eta) = \mathcal{G}(\eta) - \Dx \) is a sub-principal term (meaning that \( \mathcal{R}(\eta) \)  
is of order \( 0 \), while \( \mathcal{G}(\eta) \) and \( \Dx \) are of order \( 1 \)),  
the difference between  
\( \mathcal{G}(\eta)(\Id+\mathcal{G}(\eta))^{-1} \) and \( \Dx(\Id+\Dx)^{-1} \) is a sub-sub-principal term, meaning an operator  
of order \( -2 \). This follows from a direct computation for linear operators:
\begin{align*}
\mathcal{G}(\eta)(\Id+\mathcal{G}(\eta))^{-1}  &= 
		\Dx(\Id+\Dx)^{-1}
		+\mathcal{Q}(\eta), \qquad \text{where}\\
	    \mathcal{Q}(\eta) &= (\Id+\Dx)^{-1}\mathcal{R}(\eta)(\Id+\mathcal{G}(\eta))^{-1}.
	    \end{align*}
     as can be easily verified by letting \( \Id+\mathcal{G}(\eta) \) act in the right, and \( \Id+\Dx \) on the left. 
As a result, by combining \ref{11} and \ref{12}, we conclude that the nonlinearity $\mathcal{F}(U)$ in \eqref{eq-intro-stri:MainEq} is well-defined when $U$ has $L^2 \cap L^\infty$ regularity in the spatial variable, thereby setting the stage for a Banach fixed-point argument.

The main estimates on $\mathcal{F}(U)$ read:
\begin{align}
	\| \mathcal{F}(U) \|_{L^2} \le& K\big(\|U\|_{L^2}\big) \big(\| U\|_{Y^0} + 1 \big), \label{eq-intro-stri:BdFU}\\
	\| \mathcal{F}(U_1) - \mathcal{F}(U_2) \|_{H^{-1}} \le& K\big( \|(U_1,U_2)\|_{L^2} \big) \label{eq-intro-stri:LipFU}\\
	&\times\big( \|U_1\|_{Y^0} + \|U_2\|_{Y^0} + 1 \big) \|U_1 - U_2\|_{H^{-1}},\nonumber
\end{align}
where $\|U\|_{Y^0} := |U|_{L^\infty} + |\mathcal{H} U|_{L^\infty}$. At first, one may not expect to have to consider the $Y^0$ norm, but it naturally arises in various places, particularly when estimating the vertical component of the fluid velocity on the free surface; see Proposition~\ref{bdr_vel_bound}. The proof relies on estimates for $\mathcal{R}(\eta)$ and will be given in Propositions \ref{prop-para-schr:SourEsti}-\ref{prop-para-schr:SourLip}.

On the other hand, the $L^2$ estimate of $\mathcal{F}(U)$ can be extended to higher regularity cases: for $s>0$, when $\langle D_x \rangle^s U \in L^2 \cap Y^0$, the function $\mathcal{F}(U)$ belongs to $H^s$ for all $s \ge 0$. More precisely, in Proposition~\ref{prop-para-schr:SourEsti}, we will show that  
\begin{equation}\label{eq-intro-stri:BdFUHighReg}
    \| \mathcal{F}(U) \|_{H^s} \leq K\big(\|U\|_{L^2}\big) \left( \|U\|_{Y^0} + 1 \right) \Big[ (\|U\|_{Y^0}+1) \|U\|_{H^s} + \|U\|_{Y^s} \Big],
\end{equation}
where $\|U\|_{Y^s} := \|\langle D_x \rangle^s U\|_{Y^0}$. We emphasize that this estimate is \emph{tame}, meaning that the right-hand side is linear in the 
strongest norms $\|U\|_{H^s}$ and $\|U\|_{Y^s}$, with coefficients depending only on the weakest norms $\|U\|_{L^2}$ and $\|U\|_{Y^0}$. Once the $L^2 \cap Y^0$ (spatial) regularity of the solution $U$ is established, the \emph{tame estimate} above enables us to deduce the \emph{propagation of regularity}—that is, if the initial data satisfies $U_0 \in H^s \cap Y^s$, then $U(t) \in H^s \cap Y^s$ for all time $t$.

The second set of ingredients consists of \emph{Strichartz estimates}, which reflect the dispersive nature of the system. 
Namely, by rewriting (\ref{eq-intro-stri:MainEq}) in its integral equation form:
\begin{equation}\label{Strong_Intro}
U(t,x)=e^{-itp(D_x)}U_0(x)+\int_0^t e^{i(\tau-t)p(D_x)}\mathcal{F}(U(\tau,x))\dtau,
\end{equation}
we see that $U(t,x)$ enjoys better integrability in $x$ compared to $U_0(x)$. Such phenomenon was originally observed by Strichartz in the context of Fourier restriction \cite{Strichartz1977}, and soon implemented by Ginibre-Velo \cite{GV1979}, Kato \cite{Kato1987}, Tsutsumi \cite{Tsutsumi1987} and Cazenave-Weissler \cite{CW1988,CW1990} to prove low-regularity local well-posedness of the Cauchy problem of power NLS. 

In our context, the need for Strichartz estimates arises from the necessity of bounding the Dirichlet-to-Neumann operator $\mathcal{G}(\eta)$. The estimates presented in \ref{11}–\ref{12} are all \emph{linear} or \emph{quadratic} in terms of $\|\eta\|_{W^{2,\infty}_x}$ and $\|u\|_{Y^0_x}$, whereas standard energy estimates fail to yield any pointwise bound for $\partial_x^2 \eta$ and $u$. This issue is resolved by the dispersive effect of $e^{itp(D_x)}$, which allows us to establish existence and uniqueness in the space $L_t^\infty L^2_x \cap L^4_t Y^0_x$ via a fixed-point argument, even when the initial data $U_0$ belongs only to $L^2_x$.  

Notably, since the $Y^0$ norm (or its higher-regularity counterpart $Y^s$) is $L^4_t$-integrable rather than uniformly bounded in time, the total exponent of these norms must be strictly less than $4$ in the estimates \eqref{eq-intro-stri:BdFU}--\eqref{eq-intro-stri:BdFUHighReg} mentioned above. This constraint ensures that a contraction can be achieved in the fixed-point argument for short times.

Another purpose of Strichartz estimates is to relax the uniqueness condition. Actually, remembering the Lipschitz regularity \eqref{eq-intro-stri:LipFU} of the nonlinearity $\mathcal{F}(U)$ in the weaker space $H^{-1}_x$, we can prove the uniqueness in the space $L_t^\infty L^2_x \cap L^4_t Y^0_x$. However, our goal is to replace this space with the more classical one, $L_t^\infty L^2_x \cap L^4_t L^\infty_x$. The Strichartz estimate provides an \emph{a priori} time-averaged pointwise bound for $U$: if $U \in L^\infty_t L^2_x \cap L^4_t L^\infty_x$, then the right-hand side of (\ref{Strong_Intro}) gives an \emph{even better} output in the space $L^\infty_t L^2_x \cap L^4_t Y^0_x$. This improvement is based on a subtle identity, established in Proposition~\ref{P:5.6}, which enables us to estimate the $L^1$-norm of the Hilbert transform of the nonlinearity $N$.

Combining these two sets of key ingredients, we prove in Section \ref{sect:Cauchy} that the right-hand side of (\ref{Strong_Intro}) defines a contraction in some closed ball in the space $L^\infty_tL^2_x\cap L^4_tY^0_x$, equipped with a \emph{weaker} norm. The local well-posedness of the Cauchy problem then follows from a Banach fixed-point argument.  

Finally, we justify the \emph{a posteriori} regularity of the solution if the initial data is sufficiently regular, using the tame estimate (\ref{eq-intro-stri:BdFUHighReg}) and proving energy conservation. On the other hand, since local well-posedness relies only on the energy, we obtain global well-posedness.

\subsection{Notations}\label{Notations}
We use $C$ (resp.\ $K$) to denote a general constant (resp.\ nondecreasing function $K\colon\R_+\to\R_+$) whose value may change from line to line in the text. 

Given two normed vector spaces $X_1$ and $X_2$, $\mathcal{L}(X_1,X_2)$ denotes the space of bounded operators from $X_1$ to $X_2$ and $\mathcal{L}(X)$ is a shorthand notation for $\mathcal{L}(X,X)$.

Given $T\geq0$, a normed vector space $X$ and a function $\varphi$ defined on $[0,T]\times \R$, we denote by $\varphi(t)$ the function $\R\ni x\mapsto \varphi(t,x)\in\C$. If $X$ is a space of functions on $\R$, then we write 
$$
L^p_TX_x=L^p([0,T];X)
$$
for the space of $X$-valued $L^p$ functions on $[0,T]$. 

In Section \ref{sect:Reg}-\ref{sect:Para}, we will be estimating the Dirichlet-to-Neumann operator itself. Therefore, the results in Section \ref{sect:Reg}-\ref{sect:Para} are stated in terms of norms of spatial function only. For example, in inequality (\ref{eq-reg-dn:Main}), which states
$$
\|\mathcal{R}(\eta)\|_{\mathcal{L}(H^\sigma)} \le K\big(\|\eta\|_{H^2}\big) \big( \|\eta\|_{W^{2,\infty}} + 1 \big),
$$
we omit the dependence of the norms on $x$, since the time variable is not involved at all. On the other hand, in Section \ref{sect:Cauchy}, we will be working with the Cauchy problem, for which we need to control norms in the space $L^4_TY^s_x$. Therefore, the results in Section \ref{sect:Cauchy} are stated in terms of norms of spacetime functions. For example, in Proposition \ref{prop-cauchy-lwp:Unique}, we clearly indicate that the spacetime function norm of the solution map $\mathscr{G}(U_0;U)$ should be controlled by the $L^2_x$ norm of initial value $U_0$, which is a function of $x$ only.

\section{Analytic Preliminaries}\label{sect:Pre}

In this section, we introduce the function spaces that we shall work with and 
collect several variants of standard results in harmonic analysis. They will serve as the requested background for what follows. 

\subsection{Function spaces}\label{subsect-pre:Fct}
Given a real number $s\in\R$, the \emph{Sobolev space} $H^s=H^s(\R)$ is equipped with the norm
$$
\lA u\rA_{H^s}^2=\frac{1}{2\pi}\int_{\R}\langle \xi\rangle^{2s}\vert \hat{u}(\xi)\vert^2 \dxi,
$$
where $\hat{u}$ is the Fourier transform and $\langle \xi\rangle=(1+\la\xi\ra^2)^{1/2}$. An equivalent definition of Sobolev spaces through dyadic decomposition is given in Definition~\ref{def-besov-def:Sobo}. Besides $H^s$, we shall also need the H\"{o}lder spaces and the space of Lipschitz functions.

\begin{definition}\label{def-intro-main:HolSpace}
For $s\in\N$, we denote by $W^{s,\infty}$ the usual $L^\infty$-based Sobolev space, consisting of those $f\in L^\infty$ whose derivatives $\partial_x^\alpha f$ belong to $L^\infty$ for all $\alpha \le s$. For $s\in[0,+\infty)\setminus\N$, we define $W^{s,\infty}$ as the usual H{\"o}lder space $C^{[s],s-[s]}$ (which coincides with the Zygmund space $C^s_*$ when $s\not\in\N$; see~Definition~\ref{def-besov-def:Sobo} in the appendix for the definition of Zygmund spaces).
\end{definition}

We then introduce another scale of Banach spaces 
that will play a crucial role to estimate the action of the Dirichlet-to-Neumann operator on Lipschitz functions.

\begin{definition}\label{def-intro-main:HolSpaceVar}
Denote by $\mathcal{H}$ the Hilbert transform, that is the Fourier multiplier with symbol $-i\sgn(\xi)$. The space $Y^0$ consists of those functions $u\in L^\infty$ such that $\mathcal{H}u$ belongs to  $L^\infty$, equipped with the norm 
\begin{equation}\label{eq-intro-main:HolSpaceVar0}
	    \|u\|_{Y^0} := |u|_{L^\infty} + |\mathcal{H} u|_{L^\infty}.
\end{equation}
More generally, for $s\in\R$, we denote by $Y^s$ the space of those tempered distributions $u\in \mathcal{S}'(\R)$ satisfying 
\begin{equation}\label{eq-intro-main:HolSpaceVar}
	    \|u\|_{Y^s} := 
	    |\langle D_x \rangle^{s} u|_{L^\infty} + |\mathcal{H}\langle D_x \rangle^{s} u|_{L^\infty}<+\infty,
	\end{equation}	
where $\langle D_x\rangle^s$ is the Fourier multiplier with symbol $\langle \xi\rangle^s=(1+\la\xi\ra^2)^{s/2}$.
\end{definition}

It will be justified shortly that $Y^0$ is in fact the Hardy space with index $\infty$. 
It is an interesting feature that the theory of analytic functions enters into the study of this problem.

Let us state some elementary properties of $Y^s$.
\begin{proposition}\label{prop-pre-fct:Basic}For all \(s \geq 0\), we have the following properties:
\begin{enumerate}
    \item For all \(u \in Y^s\), \(\|u\|_{Y^s} = \|\langle D_x \rangle^s u\|_{Y^0}\);
    
    \item \(Y^s\) is invariant with respect to the Hilbert transform, i.e., \(\|u\|_{Y^s} = \|\mathcal{H} u\|_{Y^s}\);
    
    \item There is a continuous embedding \(H^{s + 1/2+\epsilon} \subset Y^s\) for any $\epsilon>0$;
    
    \item When \(s \notin \mathbb{N}\), we have \(Y^s \subset W^{s, \infty}\) with continuous embedding.
\end{enumerate}

\end{proposition}

Next, we check the boundedness of Fourier multipliers on $Y^s$. 
\begin{definition}\label{def-pre-fct:Multi}
    Let $m\in\R$. A function $a = a(\xi)$ is said to be a \textit{radial symbol} of order $m$, if the following three conditions are satisfied:
    \begin{enumerate}
        \item there exists $\tilde{a}\in C^\infty(\R)$ such that $a(\xi) = \tilde{a}(|\xi|)$ for all $\xi\in\R$;
        
        \item for all $|\xi|>1$ and all $k\in\N$,
        \begin{equation*}
            \left|\partial_{\xi}^k a(\xi)\right| \lesssim_k |\xi|^{m-k};
        \end{equation*}
        
        \item as $|\xi|\rightarrow +\infty$, the symbol $a$ admits the following expansion
        \begin{equation*}
            a(\xi) = c|\xi|^m + O(|\xi|^{m-1}),
        \end{equation*}
        where $c\neq 0$ is a constant.
    \end{enumerate}
\end{definition}

\begin{proposition}\label{prop-pre-fct:BdMultLinfty}
    If $a$ is a radial symbol of order $-\varepsilon<0$, then the associated Fourier multiplier $a(D_x)$ is bounded on $L^\infty$.
\end{proposition}
\begin{proof}
Consider a function $u\in L^\infty$. Directly from the definition of the Zygmund spaces 
(see Definition~ \ref{def-besov-def:Sobo}), it is clear that $u$ belongs to $C^0_*$. We shall verify that $a(D_x)$ maps $u$ to a function in $C^\varepsilon_*$, which is contained in the space $L^\infty$. Since
    \begin{equation*}
        \|a(D_x)u\|_{C^\varepsilon_*} = \sup_{j\in\N} 2^{j \varepsilon}|a(D_x)\Delta_j u|_{L^\infty},
    \end{equation*}
    the problem reduces to establishing the following two estimates of Fourier multipliers:
    \begin{align}
        &\sup_{\lambda \ge 1} \lambda^\varepsilon \| a(D_x)\varphi(D_x/\lambda) \|_{\mathcal{L}(L^\infty)} \lesssim 1, \label{eq-pre-fct:BdMultLinfty-1} \\
        &\| a(D_x)\chi(D_x) \|_{\mathcal{L}(L^\infty)} \lesssim 1, \label{eq-pre-fct:BdMultLinfty-2}
    \end{align}
    where $\varphi$ and $\chi$ are smooth radial bump functions supported in an annulus and in a ball centered at zero, respectively (see \eqref{eq-besov-def:DyaBlk}). 
    
    In fact, the proof of~\eqref{eq-pre-fct:BdMultLinfty-1} follows from well-known results on pseudo-differential operators  
    which ensure that a Fourier multiplier of order $m$ with a {\em smooth} symbol is bounded from $C^s_*$ to $C^{s-m}_*$ for any $s\in\R$.

    To prove the estimate~\eqref{eq-pre-fct:BdMultLinfty-2}, since the kernel associated with a Fourier multiplier is of convolution type, we only need to show that the Fourier transform 
    of the symbol $a(\xi)\chi(\xi)$ belongs to~$L^1$. To do so, we begin by using the fact that $a(\xi) = \tilde{a}(|\xi|)$ for some $\tilde{a}\in C^\infty(\R)$, to get that, by considering the Taylor expansion of $\tilde{a}$ at zero,
    \begin{equation*}
        a(\xi)\chi(\xi) = \left( a_0 + a_1|\xi| + a_2\xi^2 + b(\xi) \right) \chi(\xi),
    \end{equation*}
    where $a_0,a_1,a_2\in\C$ are constants and $b\in C^2(\R)$. Obviously, the Fourier transforms of $a_0\chi(\xi)$ and $a_2\xi^2\chi(\xi)$ are Schwartz functions, thus belong to $L^1$. To study the contribution of 
    $b(\xi)\chi(\xi)$, we integrate 
    by parts, to get
    \begin{equation*}
        \left| \int_{\R} e^{-i\xi z} b(\xi)\chi(\xi) \dxi \right| \lesssim \langle z \rangle^{-2},
    \end{equation*}
    which ensures the integrability of the Fourier transform of $b(\xi)\chi(\xi)$. 
    Eventually, notice that the Fourier transform of the remainder term  $a_1|\xi|\chi(\xi)$ 
    is smooth (since it is compactly supported), and 
    as $|z|\rightarrow+\infty$, it decays like $|z|^{-2}$ since $\chi(\xi)|\xi|$ equals $|\xi|$ when $|\xi| \ll 1$. 
    This completes the proof.
\end{proof}

\begin{proposition}\label{prop-pre-fct:BdMult}
    Let $a$ be a radial symbol of order $m\in\R$. Then the Fourier multipliers $a(D_x)$ and $a(D_x)\mathcal{H}$ are bounded from $Y^{s}$ to $Y^{s-m}$ for all $s\in\R$.
\end{proposition}
\begin{proof}
    From the asymptotic behavior of $a$ at high frequency, one can decompose $a$ as:
    \begin{equation*}
        a(\xi) = c\langle\xi\rangle^m + b(\xi),
    \end{equation*}
    where $b$ is a radial symbol of order $m-1$ and $c\neq 0$ is a constant. For arbitrary $u\in Y^s$, we apply the decomposition above and obtain
    \begin{align*}
        \|a(D_x)u\|_{Y^{s-m}} 
        &\lesssim \|\langle D_x \rangle^m u\|_{Y^{s-m}} + \|b(D_x)u\|_{Y^{s-m}} \\
        &\lesssim \|u\|_{Y^{s}} + |b(D_x)\langle D_x \rangle^{-m}\langle D_x \rangle^{s}u|_{L^\infty} + |b(D_x)\langle D_x \rangle^{-m}\mathcal{H}\langle D_x \rangle^{s}u|_{L^\infty}.
    \end{align*}
    Since $b(\xi)\langle\xi\rangle^{-m}$ is a radial symbol of order $-1$, 
    the previous proposition implies 
    that $b(D_x)\langle D_x \rangle^{-m}$ is bounded on $L^\infty$. Then the boundedness of $a(D_x)$ follows:
    \begin{equation*}
        \|a(D_x)u\|_{Y^{s-m}} \lesssim \|u\|_{Y^{s}} + |\langle D_x \rangle^{s}u|_{L^\infty} + |\langle D_x\rangle^{s} \mathcal{H}u|_{L^\infty} 
        \simeq\|u\|_{Y^{s}}.
    \end{equation*}
    As for $a(D_x)\mathcal{H}$, it suffices to use the fact that $Y^{s}$ is invariant by $\mathcal{H}$.
\end{proof}

Since the usual derivative $\partial_x$ can be expressed as $\Dx\mathcal{H}$, Proposition \ref{prop-pre-fct:BdMult} implies the following inclusion:
\begin{proposition}\label{prop-pre-fct:Embd}
    For all $s\in\N$, the inclusion $Y^s\subset W^{s,\infty}$ is continuous.
\end{proposition}

Eventually, we will need the following product estimate.

\begin{proposition}\label{prop-pre-fct:Prod}
Fix a real number $s>1/2$. Then 
the product is continuous from $Y^0\times H^s$ to $Y^0$: there exists a positive constant $C_s$ such that, for all 
$a\in H^s$ and for all $u\in Y^0$,
    \begin{equation}\label{eq-pre-fct:Prod}
        \|au\|_{Y^0} \leq C_s \|a\|_{H^s} \|u\|_{Y^0}.
    \end{equation}
\end{proposition}
\begin{proof}
The proof of this claim relies on decomposing each function into two components: its nonpositive and nonnegative frequency parts. Specifically, we introduce the Fourier multipliers $\mathbbm{1}_{\R_\pm}(D_x)$ defined by
$$
\mathbbm{1}_{\R_+}(D_x)=\mez(I+i\Hi),\quad \mathbbm{1}_{\R_-}(D_x)=\mez(I-i\Hi),
$$
so that, given a function $u$, the function  $u_\pm=\mathbbm{1}_{\R_\pm}(D_x)u$ is defined by 
    \begin{equation}\label{n21}
    \hat{u}_\pm(\xi)=\hat{u}(\xi) \quad\text{if}\quad \pm\xi\ge 0,\quad 
        \hat{u}_\pm(\xi)=0 \quad\text{if}\quad \pm\xi< 0.
    \end{equation}
    From Definition \ref{def-intro-main:HolSpaceVar} of $Y^0$, its norm can be characterized as
    \begin{equation}\label{n30}
        \|u\|_{Y^0} \simeq |\mathbbm{1}_{\R_+}(D_x)u|_{L^\infty} + |\mathbbm{1}_{\R_-}(D_x)u|_{L^\infty}.
    \end{equation}
    Now, we claim that, by paraproduct decomposition (see \eqref{eq-besov-bnoy:Decomp}), it suffices to estimate the $L^\infty$ norm of $\mathbbm{1}_{\R_\pm}(D_x)\left( T_a u \right)$, since the remaining terms satisfy
    \begin{align*}
        |\mathbbm{1}_{\R_\pm}(D_x)\left( T_u a + R(a,u) \right)|_{L^\infty} &\lesssim \|T_u a + R(a,u)\|_{H^s}+\|\Hi\left( T_u a + R(a,u) \right)\|_{H^s} \\
        &\le 2 \| T_u a + R(a,u) \|_{H^s} \lesssim
        |u|_{L^\infty} \|a\|_{H^s},
    \end{align*}
    which is a consequence of\eqref{eq-besov-bony:EstiParaLinfty} and \eqref{eq-besov-bony:EstiRem} and the Sobolev embedding $H^s\subset L^\infty$.
    
    An observation on the paraproduct (see definition \eqref{eq-besov-bony:DefPara} in the appendix) is that the operator $T_a$ can be chosen 
    such that it does not significantly alter the the frequency. More precisely, we claim that $T_a$ commutes with the Fourier multipliers $\mathbbm{1}_{\R_\pm}(D_x)$ provided that $N$ is large enough. Indeed, recall that, 
    by definition \eqref{eq-besov-bony:DefPara},
    \begin{equation*}
        T_a u = \sum_{j\ge N} S_{j-N}a \Delta_j u,
    \end{equation*}
    where $S_{j-N}a$ has frequency restricted in a ball centered at zero of radius $\sim 2^{j-N}$. When $u$ contains only positive frequencies, for all $j\ge N$, the Fourier transform of $\Delta_j u$ is supported in the interval $[c2^j, C2^j]$ for some constants $C>c>0$. Thus, for all $j\ge N$, the product $S_{j-N}a \Delta_j u$ merely has frequency $\xi\in[c2^j-c'2^{j-N}, C2^j+C'2^{j-N}] \subset \R_+$ for some $C'>c'>0$ and all $N \gg 1$, which proves the claim $[T_a,\mathbbm{1}_{\R_+}(D_x)]=0$. The same argument applies when $\supp\hat{u}\subset\R_-$.
    
    Based on this fact, one can write $\mathbbm{1}_{\R_\pm}(D_x)\left( T_a u \right)$ as $T_a \mathbbm{1}_{\R_\pm}(D_x)u$. An application of estimates \eqref{eq-besov-bony:EstiParaLinfty} and \eqref{eq-besov-bony:EstiRem}, together with the Sobolev embedding $H^s\subset L^\infty$, yields
    \begin{align*}
        |\mathbbm{1}_{\R_\pm}(D_x)&\left( T_a u \right)|_{L^\infty} = |T_a \mathbbm{1}_{\R_\pm}(D_x)u|_{L^\infty} \\
        &\le |a \mathbbm{1}_{\R_\pm}(D_x)u|_{L^\infty} + \left| T_{\mathbbm{1}_{\R_\pm}(D_x)u} a \right|_{L^\infty} + \left|R\left( a,\mathbbm{1}_{\R_\pm}(D_x)u \right)\right|_{L^\infty} \\
        &\lesssim |a|_{L^\infty} |\mathbbm{1}_{\R_\pm}(D_x)u|_{L^\infty} + \left\| T_{\mathbbm{1}_{\R_\pm}(D_x)u} a \right\|_{H^s} + \left\|R\left( a,\mathbbm{1}_{\R_\pm}(D_x)u \right)\right\|_{H^s} \\
        &\lesssim \|a\|_{H^s} |\mathbbm{1}_{\R_\pm}(D_x)u|_{L^\infty} \le \|a\|_{H^s} \|u\|_{Y^0}.
    \end{align*}
    This concludes the desired inequality \eqref{eq-pre-fct:Prod}.
\end{proof}

	Finally, we justify that $Y^0$ is in fact the Hardy space with index $\infty$; namely, there is some bounded holomorphic function $F(x+iy)$ defined on the upper half plane $\{y>0\}$, such that the \emph{non-tangential limit} of $\RE F$ on the real axis is $f$.

\begin{proposition}\label{prop-pre-fct:EquiHardy}
A real-valued function $f$ belongs to $Y^0$ if and only if $f+i\Hi f$ is the boundary value of a bounded holomorphic function on the upper half plane.
\end{proposition}
\begin{proof}
If \( f \in Y^0 \), then with \( P_y \) being the Poisson kernel, the function  
\[
F(x+iy) = (P_y * f)(x) + i\big(P_y * (\mathcal{H} f)\big)(x)
= \frac{1}{i\pi} \int_{-\infty}^{+\infty} \frac{f(\zeta)}{(x+iy) - \zeta} \diff\!\zeta
\]
is holomorphic in the upper half-plane. Since \( \{P_y\}_{y>0} \) forms an approximate identity, it follows that \( F(z) \) is a bounded holomorphic function satisfying \( |F|_{L^\infty} \leq \|f\|_{Y^0} \).  

Conversely, if \( F(z) \) is a bounded holomorphic function defined for \( z \) in the upper half-plane, then the \emph{non-tangential limit} of \( F(z) \) as \( z = x+iy \to x \) for \( x \in \mathbb{R} \) exists for almost every \( x \), and is given by \( f + i\mathcal{H} f \) for some locally integrable function \( f \). Since \( f + i\mathcal{H} f \) is bounded, we conclude that \( f \in Y^0 \).
\end{proof}

\subsection{Endpoint commutator estimate}
In this subsection, we prove a commutator estimate needed for the study of the 
derivative of the Dirichlet-to-Neumann operator. This estimate complements the following general commutator estimate
$$
\big\|[ \Dx^s,A ]f\big\|_{L^2}
\lesssim_s \|A\|_{H^s}|f|_{L^\infty}
\quad\text{for}\quad 0<s<1,
$$
proved in Section $2.2$ of Taylor \cite{Taylor2000}, as well as the estimate 
$$
\big\|[ \Dx,A ]f\big\|_{L^r}
\lesssim_{p,q} \|\nabla A\|_{L^q}\|f\|_{L^p},
\quad \frac{1}{r}=\frac{1}{p}+\frac{1}{q},
\quad 1<p<\infty,\,1<q\leq \infty,
$$
proved in Calder\'on \cite{Calderon1965} and Coifman--Meyer \cite{CM1975}. The point is that the previous estimates exclude the endpoint case $s=1$ or $p=\infty$. We shall make use of some delicate estimates of paraproduct operator to prove the following endpoint estimate:

\begin{lemma}\label{Commu_Esti}
Suppose the function $A\in H^1(\R)$, while $f\in L^\infty(\R)$. Then 
$$
\big\|[\Dx,A]f\big\|_{L^2}
\leq C\|A\|_{H^1}|f|_{L^\infty}.
$$
\end{lemma}
\begin{remark}Noticing that $\Dx$ is the Dirichlet-to-Neumann operator for the lower-half space, Lemma \ref{Commu_Esti} is, in fact, covered by the more general commutator estimates for Dirichlet-to-Neumann operators (see Theorem 5.1 of Kenig-Lin-Shen \cite{KLS2014}, and also Shen \cite{Shen2013}). However, the proof in these references relies on a subtle bilinear estimate obtained by Dahlberg \cite{Dahlberg1986}. Therefore, it is of interest to find a proof of Lemma \ref{Commu_Esti} that relies solely on classical tools, such as the Littlewood-Paley dyadic decomposition, standard properties of the Hilbert transform, and Carleson measures.
\end{remark}
\begin{proof}[Proof of Lemma \ref{Commu_Esti}]
We can cast paraproduct decomposition to the product operator with $A$, and write
\begin{equation*}
[\Dx,A]f = [\Dx,T_A]f + \Dx\big( T_fA + R(A,f) \big) - T_{\Dx f}A - R(A,\Dx f).
\end{equation*}
The second and third terms in the right-hand side can be estimated by a more or less 
standard computation. For example, by (2.3)-(2.6) in Chapter 2 of \cite{Taylor2000}, there holds
\begin{align*}
\|T_f A\|_{H^1} 
&\lesssim  |f|_{L^\infty} \|A\|_{H^1}, \\
\|R(A,f)\|_{H^1} 
&\lesssim |f|_{L^\infty} \|A\|_{H^1}, \\
\| T_{\Dx f}A \|_{L^2} 
&\lesssim |f|_{L^\infty} \|A'\|_{L^2}.
\end{align*}

The major difficulty is with the commutator $[|D_x|,T_A]f$ and the paraproduct remainder $R(A,\Dx f)$. We shall make use of the following classical bilinear estimate:

\begin{proposition}[cf.\ Theorem 33 in \cite{CM1978}]\label{Bilinear}
Suppose $\varphi,\eta$ are Schwartz functions on $\R^d$, such that $\eta(0)=0$. Define a bilinear operator
$$
\tau(g,f)=\sum_{j\in\N}\varphi(2^{-j}D_x)g\cdot\eta(2^{-j}D_x)f.
$$
Then for functions $g,f$ on $\R^d$, there holds
$$
\|\tau(g,f)\|_{L^2}\lesssim_{\varphi,\eta}\|g\|_{L^2}\|f\|_{\BMO}.
$$
See \cite{CM1978} for the definition of the space $\BMO$.
\end{proposition}
In fact, the original statement of Proposition \ref{Bilinear} in \cite{CM1978} is concerned with a \emph{continuous} version
$$
\tau(g,f)=\int_0^\infty \varphi(tD_x)g\cdot\eta(tD_x)f\frac{m(t)}{t}\dt,
$$
where $m$ is a bounded function, but the proof is essentially the same, making use of some basic properties of Carleson measure. 

Then, we estimate the remainder term $R(A,\Dx f)$. From $\Dx=D_x\mathcal{H}$ and the definition~\eqref{eq-besov-bony:DefPara} of $R(\cdot,\cdot)$, we can write $R(A,\Dx f)$ as
\begin{align*}
    R(A,\Dx f) =& \sum_{j\ge0} \sum_{k:|j-k|\le N} \Delta_j A \cdot \Delta_k \Dx f \\
    =& \sum_{j=0}^N \sum_{k:|j-k|\le N} \Delta_j A \cdot \Delta_k \partial_x \Hi f 
    + \sum_{l=-N}^N\sum_{j\ge N+1} \Delta_j A \cdot \Delta_{j+l} \Dx f,
\end{align*}
where $N\in\N$ is fixed. To estimate the first term in the right-hand side, 
we observe that, 
for all indices $j,k$ such that $0 \le j \le N$ and $|j-k|\le N$, we have
\begin{align*}
\left\| \Delta_j A \cdot \Delta_k D_x \Hi f \right\|_{L^2} 
&\leq \| \Delta_j A\|_{L^2}\big|\Delta_k\Dx f\big|_{L^\infty}\\
&\lesssim_{k} \| A \|_{L^2} \la  \Dx f \ra_{C^{-1}_*} \\
&\lesssim_N  \| A \|_{H^1} \la  f  \ra_{C^{0}_*}\lesssim_N  \| A \|_{H^1} \la f  \ra_{L^\infty}
\end{align*}
As for the second term, we shall use Proposition~\ref{Bilinear}. To do so, we rewrite it as
\begin{align*}
    &\sum_{l=-N}^N\sum_{j\ge N+1} \Delta_j A \cdot \Delta_{j+l} \Dx f \\
    &\hspace{4em}= \sum_{l=-N}^N 2^{-l} \sum_{j\ge N+1}  \Big( |2^{-j}D_x|^{-1} \Delta_j\ |D_x|A \Big) \cdot \Big( \Delta_{j+l} |2^{-j-l}D_x|\ f \Big),
\end{align*}
where the symbol of the Fourier multipliers $|2^{-j}D_x|^{-1} \Delta_j$, $|2^{-j}D_x|^{-1} \Delta_j$ takes the form $\varphi(2^{-j}\xi)$, $\eta(2^{-j-l}\xi)$, respectively, for some Schwartz functions $\varphi$, $\eta$ vanishing at zero. Then we adopt Proposition~\ref{Bilinear} with $(g,f)$ replaced by $(|D_x|A,f)$ and obtain 
$$
\left\| \sum_{l=-N}^N\sum_{j\ge N+1} \Delta_j A \cdot \Delta_{j+l} \Dx f \right\|_{L^2} 
\lesssim \big\||D_x|A\big\|_{L^2} \|f\|_{\BMO} 
\lesssim \| A \|_{H^1} | f |_{L^\infty}.
$$
This completes the estimate of the $L^2$ norm of $R(A,\Dx f)$.

Finally, the commutator $[\Dx, T_A]f$ can be estimated by Proposition \ref{Bilinear} again. We decompose $f$ as $f=f_-+f_+$, where $f_\pm=\mathbbm{1}_{\R_\pm}(D_x)f$, and observe that,  
$$
\Dx f_+=D_x f_+\quad,\quad \Dx f_-=-D_x f_-.
$$
Furthermore, the proof of Proposition~\ref {prop-pre-fct:Prod} shows that
$$
T_A f_+=(T_Af)_+ \quad\text{and}\quad T_Af_-=(T_Af)_-.
$$
Combining these two observations, we conclude that
\begin{align*}
[\Dx, T_A]f 
& = \sum_{\pm} [\Dx, T_A]f_\pm \\
& = \sum_{\pm}  \big(\Dx(T_Af_\pm) - T_A \Dx f_\pm \big) \\
& = \sum_{\pm}\pm \big( D_x(T_Af_\pm) - T_A D_xf_\pm \big)\\
& = \sum_{\pm} \mp i T_{A'} f_\pm = T_{A'} \Hi f.
\end{align*}
The paraproduct $T_{A'} \Hi f$ is again of the form discussed in Proposition \ref{Bilinear}, so we obtain the $L^2$ estimate of $[\Dx, T_A]f$. 
\end{proof}

\subsection{Strichartz estimate}\label{subsect-pre:Stri}
In this subsection, we state the local-in-time Strichartz type inequality that we will use later on. The symbol to be studied is
\begin{equation*}
	p(\xi) = \theta(\xi)^{1/2}\left(1+|\xi|^4\right)^{1/2},
\end{equation*}
where $\theta(\xi)\geq1/4$ is a smooth, even function that equals $1/4$ for $|\xi|\leq1/2$ and $|\xi|/(1+|\xi|)$ for $|\xi|>1$. For later purpose a simple but crucial observation 
is that one can estimate not only the $L^p_tL^\infty_x$-norm of the unknown through Strichartz estimates, but also the $L^p_tY^0_x$-norm, where $Y^0$ is the Hardy space as in Definition \ref{def-intro-main:HolSpaceVar}.

\begin{proposition}\label{prop-pre-stri:DisperEsti}
Let $(q,r)$ be an \emph{admissible pair} of positive numbers, that is,
$$
\frac{2}{q}+\frac{1}{r}=\frac{1}{2},
\quad 1\leq q,r\leq\infty.
$$
Suppose $T\in (0,1]$, and consider an initial data $U_0\in L^2_x$ and a source term $F\in L^{q'}_TL^{r'}_x$. Then the unique solution $U\in L^\infty_TL^2_x$ to the Cauchy problem
	\begin{equation*}
		\partial_t U - ip(D_x)U = F,
		\quad U(0,x)=U_0(x)
	\end{equation*}
satisfies the following estimates:
	\begin{equation}\label{eq-pre-stri:DisperEsti}
	\begin{aligned}
	\big\|U;{L^\infty_T L^2_x\cap L^4_T Y^0_x}\big\|
	&\lesssim_{r}
		\|U_0\|_{L^2_x} + \big\|F;{L^{q'}_TL^{r'}_x}\big\|,
		\quad 2\leq r<+\infty;\\
	\big\|U;{L^\infty_T L^2_x}\cap{L^4_T Y^0_x}\big\|
	&\lesssim
		\|U_0\|_{L^2_x} + \big\|F;{L^{4/3}_TL^1_x}\big\|
		+\big\|\Hi F;{L^{4/3}_TL^1_x}\big\|,
		\quad r=+\infty.
	\end{aligned}
	\end{equation}
\end{proposition}
\begin{remark}
Since the time $T$ is finite, this is a local result. We may not expect a global-in-time Strichartz inequality since $p''(\xi)$ must vanish at some point.
\end{remark}
\begin{proof}
The Strichartz estimate is a consequence of a general result proved in Section 2 of \cite{KPV1991}, which implies that
	\begin{equation*}
	\begin{aligned}
	\big\|U;{L^\infty_T L^2_x\cap L^4_T L^\infty_x}\big\|
	&\lesssim_{r}
		\|U_0\|_{L^2_x} + \big\|F;{L^{q'}_TL^{r'}_x}\big\|,
		\quad 2\leq r<+\infty;\\
	\big\|U;{L^\infty_T L^2_x}\cap{L^4_T L^\infty_x}\big\|
	&\lesssim
		\|U_0\|_{L^2_x} + \big\|F;{L^{4/3}_TL^1_x}\big\|,
		\quad r=+\infty.
	\end{aligned}
	\end{equation*}
	To get the missing estimate about the $L^4_TL^\infty_x$-norm of $\mathcal{H}U(t,x)$, we then observe that the Hilbert transform of $U$ satisfies the same equation as $U$ does, with $F$ replaced by $\mathcal{H}F$. This immediately gives the wanted result since the Hilbert transform is a bounded linear operator in $L^p(\R)$ for any $p\in (1,+\infty)$. 
\end{proof}

\section{Low Regularity Dirichlet-to-Neumann Operator}\label{sect:Reg}
In this section, we establish refined estimates for the Dirichlet-to-Neumann operator for a one-dimensional interface. As explained in Subsection \ref{Notations}, the results are stated in terms of norms of spatial functions only.
\subsection{Basics of the Dirichlet-to-Neumann operator}\label{subsect-reg:Bas }
One approach to analyzing a boundary value problem is to flatten the boundary.
To do so, the most direct method is to use a change of variables\footnote{It may be tempting to use a general change of variables of the form $y=\rho(x,z)$. 
For instance, this idea is used in \cite{CM2,KLM,lannes2005water}.
An interesting feature of our approach is that we do not need
to consider such a change of variables.} involving the unknown $\eta$:
\begin{equation*}
z=y-\eta(x).
\end{equation*}
With $\Psi$ being the variational solution of $\Delta_{x,y}\Psi=0$, we define 
$v\colon \R\times (-\infty,0] \rightarrow \R$ by
\begin{equation}\label{defiv}
v(x,z)=\Psi(x,z+\eta(x)),
\end{equation}
the expression of $\Psi$ under the new coordinate $(x,z)$. Therefore, $v$ solves the elliptic boundary value problem
\begin{equation}\label{eq-reg-bas:EllipEq}
	\left\{\begin{array}{ll}
		\Delta_g v = 0, & \text{when }z<0, \\ [0.5ex]
		v|_{z=0} = \psi, & \\ [0.5ex]
		\lim_{z\rightarrow -\infty} \nabla_{x,z}v = 0, & 
	\end{array}\right.
\end{equation}
where the Laplacian $\Delta_g$ has expression
\begin{equation}\label{eq-reg-bas:DefDeltag}
	\Delta_g := \alpha\partial_z^2 + \beta\partial_x\partial_z + \partial_x^2 - \gamma\partial_z
\end{equation}
under the new coordinates, with
\begin{equation}\label{eq-reg-bas:DefDeltagCoeff}
	\alpha := 1 + \eta_x^2,\ \ \beta := -2\eta_x,\ \ \gamma :=\eta_{xx}.
\end{equation}

\begin{proposition}\label{prop-reg-bas:EllipReg}
    Consider two real numbers $s\ge 0$ and $1/2\le \sigma \le s+3/2$. There exists a non-decreasing function $K\colon\R_+\to\R_+$ such that the following properties holds. Given $\eta\in H^{s+2}$ and $\psi\in H^{\sigma}$, the boundary value problem \eqref{eq-reg-bas:EllipEq} admits a unique variational solution $v$ (see Definition~\ref{def-ellip-var:VarSol}) satisfying
	\begin{equation}\label{eq-reg-bas:EllipRegSobo}
	    \begin{aligned}
	        &\|\nabla_{x,z} v\|_{C^0_z([-1,0];H^{\sigma-1})} + \|\nabla_{x,z} v\|_{L^2_z(-1,0;H^{\sigma-1/2})} \\
	        &\hspace{10em}\le K\big(\|\eta\|_{H^2}\big) 
	        \big( \|\eta\|_{H^{s+2}} \|\psi\|_{H^{1/2}} + \| \psi \|_{H^{\sigma}} \big).
	    \end{aligned}
	\end{equation}
	Moreover, the second order derivatives satisfy
	\begin{equation}\label{eq-reg-bas:EllipRegSoboOrd2}
	    \|\nabla_{x,z}^2 v\|_{L^2_z(-1,0;H^{\sigma-3/2})} \le K\big(\|\eta\|_{H^2}\big) 
	    \big( \|\eta\|_{H^{s+2}} \|\psi\|_{H^{1/2}} + \| \psi \|_{H^{\sigma}} \big).
	\end{equation}
\end{proposition}
The proof of this proposition is postponed to Appendix~\ref{app:Ellip}. Note that, compared to similar results obtained in \cite{alvarez2008water} (see also \cite{lannes2013water} and \cite{ABZ2014}), the key distinction is that the estimates presented here are \textit{tame}. Namely, the right-hand sides are linear in the highest regularity norms ($\|\eta\|_{H^s}$ and $\|\psi\|_{H^\sigma}$), while the coefficients $K\big(\|\eta\|_{H^2}\big)$ and $\|\psi\|_{H^{1/2}}$ involve norms no more regular than the energy level, where $(\eta, \psi) \in H^2 \times H^1$ (see \eqref{E}). This will play a crucial role to establish the propagation of regularity in Subsection~\ref{High_Reg}.

	Now let us define the \textit{Dirichlet-to-Neumann operator} $\G(\eta)$ by
	\begin{equation}\label{eq-reg-bas:DN}
		\mathcal{G}(\eta)\psi := \left[ \left(1+\eta_x^2\right)\partial_zv -
		\eta_x\partial_x v \right]\big\arrowvert_{z=0}.
	\end{equation}
	It is convenient to introduce also the following notations:
	\begin{equation}\label{eq-reg-bas:BV}
	    \mathcal{B}(\eta)\psi := \partial_y\Psi\arrowvert_{y=\eta} = \partial_zv\arrowvert_{z=0},\ \ \mathcal{V}(\eta)\psi := \partial_x\Psi\arrowvert_{y=\eta}= \left(\partial_x v - 
	    \eta_x \partial_zv\right)\arrowvert_{z=0},
	\end{equation}
so that,
\begin{equation}\label{eq-reg-bas:GbyBV}
    \mathcal{G}(\eta)\psi = \mathcal{B}(\eta)\psi - \eta_x \mathcal{V}(\eta)\psi.
\end{equation}
The following proposition, which is a direct consequence of Proposition~\ref{prop-reg-bas:EllipReg}, ensures that $\G(\eta)\psi$ is well-defined. 

\begin{proposition}\label{prop-reg-bas:Bd}
	Let $s\ge 0$ and $1/2\le \sigma \le s+3/2$. There exists a non-decreasing function $K\colon\R_+\to\R_+$ such that, 
	for all $\eta \in H^{s+2}$ and for all $\psi \in H^{\sigma}$,
	\begin{equation}\label{eq-reg-bas:BdSobo-ext}
		\|\mathcal{G}(\eta)\psi\|_{H^{\sigma-1}} 
		\le K\big(\|\eta\|_{H^2}\big) 
		\big( \|\eta\|_{H^{s+2}} \|\psi\|_{H^{1/2}} + \| \psi \|_{H^{\sigma}} \big).
	\end{equation}
	Moreover, $\G(\eta)$ is self-adjoint in the following sense: for all Schwartz functions $\psi,\varphi$,
	$$
	\int_\R \psi \mathcal{G}(\eta)\varphi\dx=\int_\R
	\varphi \mathcal{G}(\eta)\psi\dx.
	$$	
	In particular, $\G(\eta)$ can be extended as a bounded operator from $H^\sigma$ to $H^{\sigma-1}$ for all $-1/2 \le \sigma \le 3/2$, together with the operator norm estimate:
	\begin{equation}\label{eq-reg-bas:BdSobo}
		\|\mathcal{G}(\eta)\|_{\mathcal{L}(H^{\sigma},H^{\sigma-1})} \le K\big(\|\eta\|_{H^2}\big).
	\end{equation}

\end{proposition}

Since $\psi=\Psi\arrowvert_{y=\eta}$, the chain rule implies that
\begin{equation}\label{eq-reg-bas:PsibyBV}
    \partial_x\psi = \left[ \partial_x\Psi + \eta_x \Psi_y \right]|_{y=\eta} = \mathcal{V}(\eta)\psi + \eta_x \mathcal{B}(\eta)\psi,
\end{equation}
which, together with \eqref{eq-reg-bas:GbyBV}, yields that
\begin{equation}\label{eq-reg-bas:BVbyG}
    \mathcal{B}(\eta)\psi = \frac{\mathcal{G}(\eta)\psi + \eta_x\psi_x}{1+\eta_x^2},\ \ 
    \mathcal{V}(\eta)\psi = \psi_x - \eta_x \mathcal{B}(\eta)\psi.
\end{equation}
In terms of $\B(\eta)$ and $\V(\eta)$, the nonlinear term $N(\eta,u)$ defined in \eqref{eq-intro-csz:N} can also be written as
\begin{equation}\label{eq-reg-bas:NAlt}
        N(\eta,u)  = BV\eta_x + \frac{V^2 - B^2}{2},
\end{equation}
where
$$
\begin{aligned}
B & = \B(\eta)(\Id + \mathcal{G}(\eta))^{-1}u,\\ 
V & = \V(\eta)(\Id + \G(\eta))^{-1}u.
\end{aligned}
$$
We shall prove in Proposition~\ref{prop-reg-bas:InverBd} 
that $(\Id + \G(\eta))^{-1}$ is a well-defined linear operator. 
Concerning the boundedness of the operators $\B(\eta)$ and $\V(\eta)$, we easily deduce from  Proposition \ref{prop-reg-bas:Bd} and the product rule in Sobolev spaces (see Corollary \ref{cor-besov-bony:ProdLaw}) 
that they can be estimated in the same way as $\mathcal{G}(\eta)$.

\begin{proposition}\label{prop-reg-bas:BdBV}
    Let $s\ge 0$ and $1/2\le \sigma \le s+3/2$. There exists a non-decreasing function $K\colon\R_+\to\R_+$ such that, 
	for all $\eta \in H^{s+2}$ and for all $\psi \in H^{\sigma}$,
	\begin{equation}\label{eq-reg-bas:BdBV-tame}
		\|\B(\eta)\psi\|_{H^{\sigma-1}} + \|\V(\eta)\psi\|_{H^{\sigma-1}} 
		\le K\big(\|\eta\|_{H^{s+2}}\big) 
		\big( \|\eta\|_{H^{s+2}} \|\psi\|_{H^{1/2}} + \| \psi \|_{H^{\sigma}} \big).
	\end{equation}
	In particular, for all $1/2\le \sigma \le 3/2$, there holds
	\begin{equation}\label{eq-reg-bas:BdBV}
		\|\mathcal{B}(\eta)\psi\|_{H^{\sigma-1}} + \|\mathcal{V}(\eta)\psi\|_{H^{\sigma-1}} \le K\big(\|\eta\|_{H^2}\big) \|\psi\|_{H^{\sigma}}.
	\end{equation}
\end{proposition}

Simple manipulations with Fourier multipliers show that when $\eta=0$, the Dirichlet-to-Neumann operator satisfies $\mathcal{G}(0)=\Dx$. For non-trivial $\eta$, one can show that the difference between $\mathcal{G}(\eta)$ and $\Dx$ is of order strictly less than $1$.

\begin{proposition}[Proposition 3.13 of \cite{ABZ2014}]\label{prop-reg-bas:Paralin}
Let $1/2\le \sigma \le 2$ and consider $0<\varepsilon<1/2$. 
Fix $\eta \in H^{2}$ and define $\mathcal{R}(\eta)$ by
	\begin{equation}\label{eq-reg-bas:Paralin}
		\mathcal{G}(\eta) = \Dx + \mathcal{R}(\eta).
	\end{equation}
Then $\mathcal{R}(\eta)$ is a linear self-adjoint operator, satisfying 
	\begin{equation}\label{eq-reg-bas:ParalinRem}
		\|\mathcal{R}(\eta)\|_{\mathcal{L}(H^{\sigma};H^{\sigma-1+\varepsilon})} \le K\big(\|\eta\|_{H^2}\big).
	\end{equation}
\end{proposition}

From this proposition, one may expect that the Dirichlet-to-Neumann operator $\G(\eta)$ shares 
some properties with the Fourier multiplier $|D_x|$. For instance, it induces a coercive bilinear form on $\dot{H}^{1/2}$, which is deduced from the following \textit{trace estimate}.
\begin{proposition}[Proposition 4.3.2 of \cite{alazard2021note}]\label{prop-reg-bas:Trace}
There exists a constant $C>0$ such that, for all $\psi\in H^{1/2}$ and all $\eta\in W^{1,\infty}$,
\begin{equation}\label{eq-reg-bas:Trace}
\|\psi\|_{\dot{H}^{1/2}}^2
		\le C\big(1+|\eta_x|_{L^\infty} \big) \int_{\R}\psi \G(\eta)\psi \dx.
\end{equation}
Here the $\dot{H}^{1/2}$-norm stands for the homogeneous Sobolev norm, defined by
\begin{equation*}
\|\psi\|_{\dot{H}^{1/2}}^2 := \int_\R |\xi||\hat{\psi}(\xi)|^2 \dxi.
\end{equation*}
\end{proposition}
	
Another feature of the Dirichlet-to-Neumann operator $\G(\eta)$ is that the $L^2$-norm of $\G(\eta)\psi$ is equivalent to that of $\partial_x\psi$, which is a direct consequence of the \textit{Rellich inequality}  (see~\cite{McLean,Necas}). 
\begin{proposition}\label{prop-reg-bas:Rellich}
    There exists a non-decreasing function $K\colon\R_+\to\R_+$ such that, for all $\eta\in W^{1,\infty}$ and for all $\psi\in H^1$, 
    \begin{equation}\label{eq-reg-bas:Rellich}
        \frac{1}{K\big(|\partial_x\eta|_{L^\infty}\big)} \int_{\mathbb{R}}(\partial_x\psi)^2\dx \leq \int_{\mathbb{R}}(\mathcal{G}(\eta)\psi)^2\dx \leq K\big(|\partial_x\eta|_{L^\infty}\big) \int_{\mathbb{R}}(\partial_x\psi)^2\dx.
    \end{equation}
\end{proposition}

Since the Dirichlet-to-Neumann operator $\mathcal{G}(\eta)$ is self-adjoint and positive, one may expect it to be invertible, at least in homogeneous Sobolev spaces. We refer to Section 3.7.6 of Lannes' book~\cite{lannes2013water} for a detailed study of $\mathcal{G}(\eta)^{-1}$. In what follows, we shall check that $\Id+\mathcal{G}(\eta)$ is invertible in non-homogeneous Sobolev spaces and its inverse is of order $-1$. 

\begin{proposition}\label{prop-reg-bas:InverBd}
	Let $\eta \in H^{2}$. The operator $\Id+\mathcal{G}(\eta)\colon H^\mez\to H^{-\mez}$ is invertible and moreover 
	$(\Id+\mathcal{G}(\eta))^{-1}$ is self-adjoint. The inverse $(\Id+\mathcal{G}(\eta))^{-1}$ can be extended 
	to a bounded operator from $H^\sigma$ to $H^{\sigma+1}$ for any real number $\sigma$ such that $-3/2\le\sigma\le 1/2$, and 
	there exists a non-decreasing function $K\colon\R_+\to\R_+$ such that
	\begin{equation}\label{eq-reg-bas:InverBdSobo}
		\|(\Id+\mathcal{G}(\eta))^{-1}u\|_{H^{\sigma+1}} \le K\big(\|\eta\|_{H^2}\big) \|u\|_{H^{\sigma}}.
	\end{equation}
\end{proposition}

\begin{proof}
    We first check that $\Id+\mathcal{G}(\eta)$ is invertible. Let us fix $u\in H^{-1/2}$ and consider the bilinear form $b$ 
    defined on $H^{1/2}$ by
    \begin{equation*}
        b(\psi_1,\psi_2) := \int_{\R} \left( \psi_1\psi_2 + \psi_1\mathcal{G}(\eta)\psi_2 \right) \dx,
    \end{equation*}
    Since the Dirichlet-to-Neumann operator $\mathcal{G}(\eta)$ is self-adjoint and bounded from $H^{1/2}$ to $H^{-1/2}$, $b$ is symmetric and bounded, provided that $\eta\in H^2$ (see \eqref{eq-reg-bas:BdSobo}). Moreover, due to the trace estimate given by 
    Proposition~\ref{prop-reg-bas:Rellich}, $b$ is also coercive. Specifically, 
    there exist positive constants $C$ and $C'$ 
    such that, for all $\psi\in H^{1/2}$,
    \begin{equation*}
        b(\psi,\psi) \ge \|\psi\|_{L^2}^2 + \frac{1}{1+|\partial_x\eta |_{L^\infty}} \|\psi\|_{\dot{H}^{1/2}}^2 \ge C'\left( 1+ \|\eta\|_{H^2} \right)^{-1} \|\psi\|_{H^{1/2}}^2.
    \end{equation*}
    Note that, by duality, any function $u$ in $H^{-1/2}$ can be regarded as a bounded linear functional on $H^{1/2}$. 
    Therefore, by Riesz's theorem, 
    the equation $\psi + \mathcal{G}(\eta)\psi = u$ admits a unique solution in $H^{1/2}$ which we denote by $(\Id+\mathcal{G}(\eta))^{-1}u$. 
    Furthermore, this unique solution $\psi\in H^{1/2}$ satisfies
    \begin{equation*}
        \|(\Id+\mathcal{G}(\eta))^{-1}u\|_{H^{1/2}} \le K\big(\|\eta\|_{H^2}\big) \|u\|_{H^{-1/2}},
    \end{equation*}
    which is \eqref{eq-reg-bas:InverBdSobo} with $\sigma=-1/2$.
    
    Next, we prove the estimate \eqref{eq-reg-bas:InverBdSobo} for $\sigma\in (-1/2,1/2]$ by iteration. 
    Let us assume that \eqref{eq-reg-bas:InverBdSobo} is true for some $\sigma = \sigma_0\in [-1/2,1/4]$. 
    Our aim is to show that \eqref{eq-reg-bas:InverBdSobo} also holds for $\sigma=\sigma_0+1/4$. 
    Note that, once it is proved, the intermediate case $\sigma\in(\sigma_0,\sigma_0+1/4)$ can be concluded by interpolation. 
    Consider now an arbitrary function $u$ in $H^{\sigma_0}$. 
    From the assumption of iteration, $\psi$ lies in $H^{\sigma_0+1}$. 
    Since \begin{equation*}
        u = \psi+ \mathcal{G}(\eta)\psi = \left( \Id+\Dx\right) \psi + \mathcal{R}(\eta)\psi,
    \end{equation*}
    we can 
    apply Proposition \ref{prop-reg-bas:Paralin} with $\sigma$ replaced by $\sigma_0+1 \in [1/2,5/4]$ and $\varepsilon=1/4$, to get
    \begin{align*}
        \| \psi \|_{H^{\sigma_0+5/4}} =& \left\| \left( \Id+\Dx\right)^{-1}\left( u- \mathcal{R}(\eta)\psi\right) \right\|_{H^{\sigma_0+5/4}} \\
        \le& \left\| u- \mathcal{R}(\eta)\psi \right\|_{H^{\sigma_0+1/4}} \\
        \le& \| u\|_{H^{\sigma_0+1/4}} + K\big(\|\eta\|_{H^2}\big) \|\psi\|_{H^{\sigma_0+1}} \le K\big(\|\eta\|_{H^2}\big)\| u\|_{H^{\sigma_0+1/4}}.
    \end{align*}
    This completes the proof of \eqref{eq-reg-bas:InverBdSobo} for $\sigma\in [-1/2,1/2]$. Since the Dirichlet-to-Neumann operator $\mathcal{G}(\eta)$ is self-adjoint, the inverse $(\Id+\mathcal{G}(\eta))^{-1}$ is also self-adjoint. Therefore, by duality, the adjoint of $(\Id+\mathcal{G}(\eta))^{-1}$ (which equals itself) is bounded from $H^{-\sigma-1}$ to $H^{-\sigma}$ with $\sigma\in [-1/2,1/2]$, which corresponds to the case $\sigma\in [-3/2,-1/2]$ of the desired estimate \eqref{eq-reg-bas:InverBdSobo}.
\end{proof}

Proposition~\ref{prop-reg-bas:InverBd} can be extended to higher regularity cases, where $(\eta, u) \in H^{s+2} \times H^s$ for $s \geq 0$, with the corresponding estimate being tame. Achieving this extension necessitates a detailed analysis of the paralinearization of the Dirichlet-to-Neumann operator $\mathcal{G}(\eta)$, which is the objective of the following section.

\begin{remark}
The theory of elliptic boundary value problems indeed provides the definition of \(\G(\eta)\) 
for Lipschitz domains, along with delicate estimates of the operator norms. 
We refer the reader to \cite{Kenig1985,Kenig1994} for details. 
However, these well-established results do not seem to be sufficient 
for our purpose: the key Proposition~\ref{prop-reg-dn:Main} below cannot be directly deduced from them.
\end{remark}

\subsection{Paralinearization of the Dirichlet-to-Neumann operator}\label{subsect-reg:DN}

The core of this section is the proof of a sharp paralinearization formula of the Dirichlet-to-Neumann operator. Assuming that $\eta$ has two derivatives in $L^\infty$, we shall prove that the difference between $\G(\eta)$ and $\Dx$ is of order $0$. When $\eta$ has $2+\varepsilon$ derivatives in~$L^\infty$, this follows from the analysis in~\cite{AM2009,ABZ2014}. However, for 
our purposes, we need to consider the case $\varepsilon=0$, which poses some interesting difficulties. 

Moreover, to establish the propagation of regularity in Section~\ref{High_Reg}, we also consider the higher regularity case: the difference between $\mathcal{G}(\eta)$ and its principal part $\Dx$ is bounded on higher order Sobolev spaces $H^{s+1}$ with $s\ge 0$, provided that $\eta \in W^{2+s,\infty}$. At the same time, we need to ensure that the corresponding estimates are tame. In fact, tame estimates are expected, as the proof primarily relies on paradifferential calculus (see Appendices~\ref{app:Besov} and~\ref{app:ParaDiff}), where such estimates naturally arise.

\begin{proposition}\label{prop-reg-dn:Main}
    $i)$ Consider a real number $s\ge 0$. Then, for all $\eta\in H^{s+2} \cap W^{s+2,\infty}$, the operator $\mathcal{R}(\eta)$ defined by 
	\begin{equation*}
		\mathcal{R}(\eta) = \mathcal{G}(\eta) - \Dx,
	\end{equation*}
	is a bounded operator from $H^{s+1}$ to itself, and self-adjoint in the following sense:
	$$
	\langle \psi , \mathcal{R}(\eta)\varphi\rangle=\langle \varphi , \mathcal{R}(\eta)\psi\rangle
	$$
	for all functions $\psi,\varphi$ in the Schwartz space $\mathcal{S}(\R)$.
	Moreover, there exists a non-decreasing function $K\colon \R_+\to\R_+$ such that, for all $\eta\in H^{s+2} \cap W^{s+2,\infty}$ and $\psi\in H^{s+1}$,
	\begin{equation}\label{eq-reg-dn:Main-tame}
	\begin{aligned}
	    \|\mathcal{R}(\eta)\psi\|_{H^{s+1}} \le K\big(\|\eta\|_{H^2}\big) &\Big[ \big( \|\eta\|_{W^{2,\infty}}+1 \big)  \big( \|\eta\|_{H^{s+2}} \|\psi\|_{H^{1/2}} + \|\psi\|_{H^{s+1}} \big) \\
		&+ \|\eta\|_{W^{s+2,\infty}}\|\psi\|_{H^{1}} \Big].
	\end{aligned}
	\end{equation}
	$ii)$ For all $\eta\in H^{2} \cap W^{2,\infty}$ and for all index $\sigma\in[-1,1]$, the operator~$\mathcal{R}(\eta)$ can be extended as a bounded operator from $H^\sigma$ to itself, together with the operator norm estimate
	\begin{equation}\label{eq-reg-dn:Main}
		\|\mathcal{R}(\eta)\|_{\mathcal{L}(H^\sigma)} \le K\big(\|\eta\|_{H^2}\big) \big( \|\eta\|_{W^{2,\infty}} + 1 \big).
	\end{equation}
\end{proposition}
This proposition shows that the ``principal part" $\Dx$ of $\G(\eta)$ does not depend on $\eta$, whence the symbol of the Dirichlet-to-Neumann operator simply reads $|\xi|$ without any lower order term. 
In general dimension, the symbol does depend on $\eta$, as shown in \cite{Calderon63,SU,AM2009}.

The proof of Proposition \ref{prop-reg-dn:Main} follows a similar idea to that of Proposition \ref{prop-reg-bas:Paralin}, but is technically much more involved. We divide it into four steps:
\begin{enumerate}[label=\textbf{(\arabic*)}]
\item\label{S1} Paralinearize the elliptic equation \eqref{eq-reg-bas:EllipEq}.

\item\label{S2} Factorize the paralinearized version of \eqref{eq-reg-bas:EllipEq}.

\item\label{S3} Solve a paralinear elliptic evolution problem, allowing comparison between the normal and tangential derivatives. 

\item\label{S4} Deduce the estimate \eqref{eq-reg-dn:Main}.
\end{enumerate}
Compared to the proof of Proposition \ref{prop-reg-bas:Paralin}, the main difference lies in the use of the H\"{o}lder regularity of \(\eta\) to improve the regularity of the remainder terms at each step.

Most of the calculations are elementary adaptations of those used in the proof of Proposition \ref{prop-reg-bas:Paralin}. However, the term \(\gamma\partial_z\) in the expression \eqref{eq-reg-bas:DefDeltag} of the Laplacian \(\Delta_g\) poses significant difficulties. Specifically, the low regularity of the coefficient \(\gamma\) prevents us from paralinearizing \(\gamma\partial_z\) simply as \(T_\gamma\partial_z\). When \(\eta\) has sufficient regularity so that \(\partial_z^3 v\) can be estimated, a possible solution is to use \textit{Alinhac's good unknown} (see, for instance, Section 4 of \cite{AM2009}). Unfortunately, in our case, it is impossible to directly apply \textit{Alinhac's good unknown} because \(\partial_z^3 v\) cannot be defined due to the low regularity of \((\eta, \psi)\). Instead, we incorporate the correction from \textit{Alinhac's good unknown} via the factorization introduced in the second step. This approach eliminates the need for \(\partial_z^3 v\), and the final formula matches that of the high regularity cases.

Our first step is to fufill \ref{S1}: replace the multiplication by $\alpha$ (resp.\ by $\beta$) by the paramultiplication 
by $T_\alpha$ (resp.\ $T_\beta$). 

\begin{lemma}\label{lem-reg-dn:Paralin}
	Fix a real number  $s\ge 0$. 
	For all $\eta\in H^{s+2}\cap W^{s+2,\infty}$ and for all $\psi\in H^{s+1}$, denote by $v$ the unique solution to the elliptic equation \eqref{eq-reg-bas:EllipEq} (see Proposition \ref{prop-reg-bas:EllipReg}). Then,
	\begin{equation}\label{eq-reg-dn:Paralin}
		\left( T_\alpha \partial_z^2 + T_\beta \partial_x\partial_z +\partial_x^2\right) v  = \gamma\partial_z v + f_1,
	\end{equation}
	where $\alpha,\beta,\gamma$ are as defined in \eqref{eq-reg-bas:DefDeltagCoeff} and the remainder $f_1$ satisfies
	\begin{equation}\label{eq-reg-dn:EstiF1}
	\begin{aligned}
	\| f_1 \|_{L^2_z(-1,0;H^{s+1/2})} 
	&\le K\big(\|\eta\|_{H^2}\big) (\|\eta\|_{W^{s+2,\infty}}+1) \|\psi\|_{H^{1}}\\
	&\quad	+K\big(\|\eta\|_{H^2}\big) (\|\eta\|_{W^{2,\infty}}+1)\|\psi\|_{H^{s+1}}
	\end{aligned}
	\end{equation}
\end{lemma}
\begin{proof}
	By comparing \eqref{eq-reg-dn:Paralin} with the original equation \eqref{eq-reg-bas:EllipEq}, we see that 
	the remainder $f_1$ reads
	\begin{equation}\label{eq-reg-dn:DefF1}
		\begin{aligned}
			f_1 =& - \left( \alpha - T_{\alpha} \right)\partial_z^2 v - \left( \beta - T_{\beta} \right)\partial_x\partial_z v \\
			=& - (1 - T_{1}) \partial_z^2 v - T_{\partial_z^2v}(\eta_x^2) - R(\partial_z^2v,\eta_x^2) + 2T_{\partial_x\partial_zv}\eta_x + 2 R(\partial_x\partial_zv,\eta_x),
		\end{aligned}
	\end{equation}
	where we used $\alpha=1+\eta_x^2$ and $\beta=-2\eta_x$ (see \eqref{eq-reg-bas:DefDeltagCoeff}).  Recall from Proposition~\ref{prop-reg-bas:EllipReg} that the function $v$ verifies
	$$
	\| \nabla_{x,z}^2v \|_{L^2(-1,0;H^{-1/2})} \le K\big(\|\eta\|_{H^2}\big) \|\psi\|_{H^1}.
	$$
	Then the first term in the right-hand side of \eqref{eq-reg-dn:DefF1} belong to $L^2(-1,0;H^s)$ for all $s\in\R$, since $\Id-T_{1}$ is a smoothing operator. To deal with the remaining terms, we first observe that Proposition~\ref{prop-besov-bony:Paralin} implies that
	\begin{equation}\label{eq-reg-dn:EstiAlpha}
	    \|\eta_x^2\|_{C^{s+1}_*} \le K\big( |\eta_x|_{L^\infty} \big)\|\eta_x\|_{C^{s+1}_*} \le K\big(\|\eta\|_{H^2}\big) \|\eta\|_{W^{s+2,\infty}}.
	\end{equation}
	Then we utilize \eqref{eq-besov-para:NegInd} and \eqref{eq-besov-bony:EstiRem} to estimate the paraproducts and reminders, respectively, to get
	\begin{align*}
	\|T_{\partial_z^2v}(\eta_x^2)\|_{L^2_zH^{s+1/2}} & + \|R(\partial_z^2v,\eta_x^2)\|_{L^2_zH^{s+1/2}} \\
	& \quad+ \|T_{\partial_x\partial_zv}\eta_x\|_{L^2_zH^{s+1/2}}+ \|R(\partial_x\partial_zv,\eta_x)\|_{L^2_zH^{s+1/2}} \\
		&\lesssim \|\partial_z^2v\|_{L^2_zH^{-1/2}} \|\eta_x^2\|_{C^{s+1}_*} + \|\partial_x\partial_zv\|_{L^2_zH^{-1/2}} \|\eta_x\|_{C^{s+1}_*} \\
		&\le K\big(\|\eta\|_{H^2}\big) \|\eta\|_{W^{s+2,\infty}} \| \nabla_{x,z}^2v \|_{L^2_zH^{-1/2}} \\
		&\le K\big(\|\eta\|_{H^2}\big) \|\eta\|_{W^{s+2,\infty}} \|\psi\|_{H^1},
	\end{align*}
	where $L^2_zH^{r}$ is a compact notation for $L^2(-1,0;H^{r})$. This completes the proof.
\end{proof}

The next task is to perform a decoupling into a forward and a backward elliptic evolution equations, fulfilling \ref{S2}.

\begin{lemma}\label{lem-reg-dn:Facto}
	Consider $\eta\in H^{2}\cap W^{2,\infty}$ and $\psi\in H^{s+1}$ with $s\ge 0$. Then there exist two symbols $a,A$ in $\Gamma^{1}_{1}$ (see Definition~\ref{def-paradiff:SymbClass}) such that
	\begin{equation}\label{eq-reg-dn:Facto}
		T_{\alpha}( \partial_z + T_{a}) (\partial_z - T_{A})v =\gamma\partial_z v+f_1+f_2,
	\end{equation}
	for some remainder $f_2$ satisfying
	\begin{equation}\label{eq-reg-dn:EstiF2}
		\lA f_2 \rA_{L^2_z(-1,0;H^{s+1/2})} 
		\le K\big(\|\eta\|_{H^2}\big) 
		( \|\eta\|_{W^{2,\infty}}+1)  \big( \|\eta\|_{H^{s+2}} \|\psi\|_{H^{1/2}} + \|\psi\|_{H^{s+1}} \big).
	\end{equation}
\end{lemma}
\begin{proof}
	We set $a,A$ as
	\begin{equation}\label{eq-reg-dn:FactoSymb}
		a:= \frac{|\xi|}{\alpha} + \frac{\beta}{2\alpha}i\xi,\ \ A:= \frac{|\xi|}{\alpha} - \frac{\beta}{2\alpha}i\xi,
	\end{equation}
	where $\alpha,\beta$ are defined in \eqref{eq-reg-bas:DefDeltagCoeff}. Since $\eta\in W^{2,\infty}$, the symbols $a,A$ belong to $\Gamma^1_1$ with
	\begin{equation}\label{eq-reg-dn:FactoSymbEsti}
		M^1_1(a) + M^1_1(A) \le K\big(\|\eta\|_{H^2}\big) \|\eta\|_{W^{2,\infty}}, \quad M^1_0(a) + M^1_0(A) \le K\big(\|\eta\|_{H^2}\big),
	\end{equation}
	where the semi-norms $M^1_1, M^1_0$ are as defined in \eqref{eq-paradiff:SymbNorm}. Herefater, we use the tangential  paradifferential calculus: given a symbol $p=p(x,\xi)$, we still denote by $T_p$ the operator acting on functions $u=u(x,z)$ so that for each fixed $z$, $(T_p u)(z)=T_{p}u(z)$. For such symbols $p$ independent of $z$, notice that $\partial_z$ commutes with $T_p$. Hereafter, we also make extensive use of the fact that, given a paradifferential operator $T_p$ with symbol $p=p(x,\xi)$ and a Fourier multiplier $m(D_x)$ with symbol $m(\xi)$, we have $T_p m(D_x)=T_{pm}$.

    Now, concerning the symbols $a$ and $A$, the important facts are that
	\begin{equation*}
		aA = \frac{4+\beta^2}{4\alpha^2}\xi^2=\frac{ \xi^2}{\alpha},\quad a-A=\frac{\beta}{\alpha}i\xi.
	\end{equation*}
	Hence
	\begin{align*}
		&\hspace{-2em}T_{\alpha}( \partial_z + T_{a}) (\partial_z - T_{A}) \\
		=& T_{\alpha}\partial_z^2 + T_{\alpha}T_{a-A}\partial_z - T_{\alpha}T_{a}T_{A} \\
		=& T_{\alpha}\partial_z^2 + T_{\alpha(a-A)}\partial_z - T_{\alpha aA} + \left(T_{\alpha}T_{a-A} - T_{\alpha(a-A)}\right)\partial_z - \left( T_{\alpha}T_{a}T_{A} - T_{\alpha aA} \right) \\
		=& T_\alpha \partial_z^2 + T_\beta \partial_x\partial_z +\partial_x^2 \\
		&+ \left( -\partial_x^2 - T_{\xi^2} \right) + \left(T_{\alpha}T_{a-A} - T_{\alpha(a-A)}\right)\partial_z - \left( T_{\alpha}T_{a}T_{A/|\xi|} - T_{\alpha aA/|\xi|} \right)\Dx
	\end{align*}
	This gives~\eqref{eq-reg-dn:Facto} with 
	\begin{equation*}
		f_2 = \left( -\partial_x + T_{i\xi} \right) \partial_xv + \left(T_{\alpha}T_{a-A} - T_{\alpha(a-A)}\right)\partial_z v - \left( T_{\alpha}T_{a}T_{A/|\xi|} - T_{\alpha aA/|\xi|} \right) |D_x| v,
	\end{equation*}
	where the first term lies in $L^2_zH^{1/2}$ since $\left( -\partial_x + T_{i\xi} \right)$ is a smoothing operator and $v$ satisfies
	\begin{equation}\label{eq-reg-dn:EllipRegOrd1}
		\| \nabla_{x,z}v \|_{L^2(-1,0;H^{s+1/2})} \le K\big(\|\eta\|_{H^2}\big) \big( \|\eta\|_{H^{s+2}} \|\psi\|_{H^{1/2}} + \|\psi\|_{H^{s+1}} \big),
	\end{equation}
	due to Proposition~\ref{prop-reg-bas:EllipReg}. To handle the remaining two terms, it suffices to show that the operators $\left(T_{\alpha}T_{a-A} - T_{\alpha(a-A)}\right)$ and $\left( T_{\alpha}T_{a}T_{A/|\xi|} - T_{\alpha aA/|\xi|} \right)$ are of order $0$. To do so, we first use the Sobolev embedding $H^2 \subset W^{1,\infty}$ to verify that
	$$
	M^0_1(\alpha) \le K\big(\|\eta\|_{H^2}\big) \big( \|\eta\|_{W^{2,\infty}} + 1 \big), \quad M^0_0(\alpha) \le K\big(\|\eta\|_{H^2}\big).
	$$
	Then the symbolic calculus for paradifferential operators (see Proposition~\ref{prop-paradiff:SymbCal}) implies that, for all $r\in\R$,
	\begin{align*}
		\| T_{\alpha}T_{a-A} - T_{\alpha(a-A)} \|_{\mathcal{L}(H^{r})} \lesssim& M^0_1(\alpha)M^1_0(a-A) + M^0_0(\alpha)M^1_1(a-A) \\
		\le& K\big(\|\eta\|_{H^2}\big) \big( \|\eta\|_{W^{2,\infty}}+1 \big),
	\end{align*}
	where we used \eqref{eq-reg-dn:FactoSymbEsti}, and similarly,
	\begin{equation}\label{eq-reg-dn:FactoLap}
		\begin{aligned}
			&\hspace{-2em}\| T_{\alpha}T_{a}T_{A/|\xi|} - T_{\alpha aA/|\xi|} \|_{\mathcal{L}(H^{r})} \\
			\lesssim& M^0_1(\alpha)M^1_0(a)M^1_0(A) + M^0_0(\alpha)M^1_1(a)M^1_0(A) + M^0_0(\alpha)M^1_0(a)M^1_1(A) \\
			\le& K\big(\|\eta\|_{H^2}\big) \big( \|\eta\|_{W^{2,\infty}}+1 \big),
		\end{aligned}
	\end{equation}
	which completes the proof of \eqref{eq-reg-dn:EstiF2}.
\end{proof}

Now, we are in position to deal with the term $\gamma\partial_z v = \eta_{xx}\partial_zv$. This is where we need to do something different from~\cite{AM2009,ABZ2014}.
\begin{lemma}\label{lem-reg-dn:GoodUnk}
	Consider $\eta\in H^{s+2}\cap W^{s+2,\infty}$ and $\psi\in H^{s+1}$ with $s\ge 0$. Then the unique solution $v$ to \eqref{eq-reg-bas:EllipEq} satisfies
	\begin{equation}\label{eq-reg-dn:GoodUnk}
		\eta_{xx}\partial_zv = -T_\alpha (\partial_z + T_a) T_A T_{v_z}\eta + f_3,
	\end{equation}
	where $\alpha,a$ and $A$ are as above and $f_3$ verifies
	\begin{equation}\label{eq-reg-dn:EstiF3}
	    \begin{aligned}
	        \| f_3 \|_{L^2_z(-1,0;H^{s+1/2})} \le &K\big(\|\eta\|_{H^2}\big) \Big[ \big( \|\eta\|_{W^{2,\infty}}+1 \big)  \big( \|\eta\|_{H^{s+2}} \|\psi\|_{H^{1/2}} + \|\psi\|_{H^{s+1}} \big) \\
	        & + \|\eta\|_{W^{s+2,\infty}}\|\psi\|_{H^{1}} \Big].
	    \end{aligned}
	\end{equation}
\end{lemma}
\begin{proof}
	Note that
	\begin{align*}
		&\hspace{-2em}T_\alpha (\partial_z + T_a) T_A T_{v_z}\eta \\
		=& T_\alpha T_a T_A T_{v_z}\eta + T_\alpha T_A T_{v_{zz}}\eta \\
		=& \left( T_\alpha T_a T_A + \partial_x^2 \right) \left( T_{v_z}\eta \right) + T_\alpha T_A T_{v_{zz}}\eta - \partial_x^2\left( T_{v_z}\eta \right) \\
		=& \left( T_\alpha T_a T_A + \partial_x^2 \right) \left( T_{v_z}\eta \right) + T_\alpha T_A T_{v_{zz}}\eta - T_{v_{xxz}}\eta - 2 T_{v_{xz}} \eta_x - T_{v_z}\eta_{xx}.
	\end{align*}
	By comparing this with \eqref{eq-reg-dn:GoodUnk}, we can write $f_3$ as
	\begin{equation}\label{eq-reg-dn:DefF3}
	\begin{aligned}
	f_3 
	& = \left( T_\alpha T_a T_A + \partial_x^2 \right) \left( T_{v_z}\eta \right) + T_\alpha T_A T_{v_{zz}}\eta \\
	&\quad- T_{v_{xxz}}\eta - 2 T_{v_{xz}} \eta_x + T_{\eta_{xx}}v_z + R(\eta_{xx},v_z).
	\end{aligned}
	\end{equation}
	In what follows, we will check that each term in the right-hand side belongs to $L^2_zH^{s+1/2}$ by using the regularity of $v$ given by Proposition~\ref{prop-reg-bas:EllipReg}, that is
	\begin{align*}
		&\| \nabla_{x,z}v \|_{L^2(-1,0;H^{s+1/2})} + \| \nabla_{x,z}^2v \|_{L^2(-1,0;H^{s-1/2})} \\
		&\hspace{10em}\le K\big(\|\eta\|_{H^2}\big) \big( \|\eta\|_{H^{s+2}} \|\psi\|_{H^{1/2}} + \|\psi\|_{H^{s+1}} \big).
	\end{align*}
	Recall that this estimate holds for all real number $s\ge 0$.
	
	The first term can be estimated through \eqref{eq-reg-dn:FactoLap} and Proposition~\ref{prop-besov-para:NegInd},
	\begin{align*}
	\big\| \left( T_\alpha T_a T_A + \partial_x^2 \right) \left( T_{v_z}\eta \right) \big\|_{L^2_zH^{s+1/2}}
		&\le K\big(\|\eta\|_{H^2}\big) \big( \|\eta\|_{W^{2,\infty}}+1 \big) \| T_{v_z}\eta \|_{L^2_zH^{s+3/2}} \\
		&\le K\big(\|\eta\|_{H^2}\big) \big( \|\eta\|_{W^{2,\infty}}+1 \big) \| v_z \|_{L^2_zL^2} \|\eta\|_{C_*^{s+3/2}} \\
		&\le K\big(\|\eta\|_{H^2}\big) \big( \|\eta\|_{W^{2,\infty}}+1 \big) \|\eta\|_{W^{s+2,\infty}} \|\psi\|_{H^1}.
	\end{align*}
	To deal with the second term, we use the fact that, from definition \eqref{eq-reg-dn:FactoSymb} of $A$, $\alpha\in\Gamma^0_0$ and $A\in\Gamma^1_0$ (see Definition~\ref{def-paradiff:SymbClass}), provided that $\eta\in H^2\subset W^{1,\infty}$. Then due to Proposition~\ref{prop-paradiff:Bd}, the operators $T_\alpha$ and $T_A$ are of order $0$ and $1$, respectively. Consequently,
	\begin{equation*}
		\| T_\alpha T_A T_{v_{zz}}\eta \|_{L^2_zH^{s+1/2}} \le K\big(\|\eta\|_{H^2}\big) \| T_{v_{zz}}\eta \|_{L^2_zH^{s+3/2}}.
	\end{equation*}
	The paraproduct $T_{v_{zz}}\eta$ in the right-hand side needs to be controlled via Proposition~\ref{prop-besov-para:NegInd},
	\begin{equation*}
		\| T_{v_{zz}}\eta \|_{L^2_zH^{s+3/2}} \lesssim \|v_{zz}\|_{L^2_zH^{-1/2}} \|\eta\|_{C^{s+2}_*} \le K\big(\|\eta\|_{H^2}\big) \|\psi\|_{H^1} \|\eta\|_{W^{s+2,\infty}}.
	\end{equation*}
	In a similar way, we can estimate the third and fourth terms in the right-hand side of \eqref{eq-reg-dn:EstiF3},
	\begin{align*}
		\| T_{v_{xxz}}\eta \|_{L^2_zH^{s+1/2}} \lesssim& \| v_{xxz} \|_{L^2_zH^{-3/2}} \| \eta \|_{C^{s+2}_*} \le K\big(\|\eta\|_{H^2}\big) \|\psi\|_{H^1} \|\eta\|_{W^{s+2,\infty}}, \\
		\| T_{v_{xz}} \eta_x \|_{L^2_zH^{s+1/2}} \lesssim& \| v_{xz} \|_{L^2_zH^{-1/2}} \| \eta_x \|_{C^{s+1}_*} \le K\big(\|\eta\|_{H^2}\big) \|\psi\|_{H^1} \|\eta\|_{W^{s+2,\infty}}.
	\end{align*}
	The last two terms in the right-hand side of \eqref{eq-reg-dn:EstiF3} can be studied via Proposition~\ref{prop-besov-bony:EstiPara} and~\ref{prop-besov-bony:EstiRem}, which imply that
	\begin{align*}
	    &\hspace{-2em}\| T_{\eta_{xx}}v_z \|_{L^2_zH^{s+1/2}} + \| R(\eta_{xx},v_z) \|_{L^2_zH^{s+1/2}} \\
		&\lesssim \| v_z \|_{L^2_zH^{s+1/2}} |\eta_{xx}|_{L^\infty} \\
		&\le K\big(\|\eta\|_{H^2}\big) \big( \|\eta\|_{H^{s+2}} \|\psi\|_{H^{1/2}} + \|\psi\|_{H^{s+1}} \big) \|\eta\|_{W^{2,\infty}},
	\end{align*}
	where the last inequality follows from Proposition~\ref{prop-reg-bas:EllipReg}. The proof of \eqref{eq-reg-dn:EstiF3} is completed.
\end{proof}

From the lemmas above, we can deduce that $v$ solves the following elliptic evolution problem, therefore fufilling \ref{S3}:
\begin{lemma}\label{lem-reg-dn:Parabolic}
	Consider $\eta\in H^{s+2}\cap W^{s+2,\infty}$ and $\psi\in H^{s+1}$ with $s\ge 0$. Then the unique solution $v$ to \eqref{eq-reg-bas:EllipEq} satisfies
	\begin{equation}\label{eq-reg-dn:MainEqV}
		(\partial_z + T_a)(v_z - T_Av - T_A T_{v_z}\eta ) = f,
	\end{equation}
	where the source term $f$ belongs to $L^2_z(-1,0;H^{s+1/2})$ with
	\begin{equation}\label{eq-reg-dn:EstiF}
	    \begin{aligned}
	        \| f \|_{L^2_z(-1,0;H^{s+1/2})} \le &K\big(\|\eta\|_{H^2}\big) \Big[ \big( \|\eta\|_{W^{2,\infty}}+1 \big)  \big( \|\eta\|_{H^{s+2}} \|\psi\|_{H^{1/2}} + \|\psi\|_{H^{s+1}} \big) \\
            & + \|\eta\|_{W^{s+2,\infty}}\|\psi\|_{H^{1}} \Big].
	    \end{aligned}
	\end{equation}
\end{lemma}
\begin{proof}
	By inserting \eqref{eq-reg-dn:GoodUnk} into \eqref{eq-reg-dn:Facto}, the equation of $v$ can be written as
	\begin{equation*}
		T_\alpha(\partial_z + T_a)(v_z - T_Av - T_A T_{v_z}\eta ) = f_1 +f_2 + f_3.
	\end{equation*}
	We apply $T_{\alpha^{-1}}$ on both sides and obtain \eqref{eq-reg-dn:MainEqV} with $f$ equal to
	\begin{equation}\label{eq-reg-dn:DefF}
		f = (1-T_{\alpha^{-1}}T_\alpha) (v_z - T_Av - T_A T_{v_z}\eta ) + T_{\alpha^{-1}} (f_1 +f_2 + f_3).
	\end{equation}
	
	To obtain the estimate \eqref{eq-reg-dn:EstiF}, we notice that, under the condition $\eta\in H^2\subset W^{1,\infty}$, the operator $T_{A/|\xi|}$ and $T_{\alpha^{-1}}$ are of order $0$. Thus, through Proposition~\ref{prop-besov-bony:EstiPara} and the estimates \eqref{eq-reg-dn:EstiF1}, \eqref{eq-reg-dn:EstiF2}, and \eqref{eq-reg-dn:EstiF3}, we have
	\begin{align*}
		&\hspace{-2em}\|T_{\alpha^{-1}} (f_1 +f_2 + f_3)\|_{L^2_zH^{s+1/2}} \\
		\le& K\big(\|\eta\|_{H^2}\big) \|f_1 +f_2 + f_3\|_{L^2_zH^{s+1/2}} \\
		\le& K\big(\|\eta\|_{H^2}\big) \Big[ \big( \|\eta\|_{W^{2,\infty}}+1 \big)  \big( \|\eta\|_{H^{s+2}} \|\psi\|_{H^{1/2}} + \|\psi\|_{H^{s+1}} \big) + \|\eta\|_{W^{s+2,\infty}}\|\psi\|_{H^{1}} \Big].
	\end{align*}
	Thanks to the regularity of $v$ (see Proposition~\ref{prop-reg-bas:EllipReg}) and Proposition~\ref{prop-besov-para:NegInd}, we also have
	\begin{align*}
		\|v_z - T_Av & - T_A T_{v_z}\eta\|_{L^2_zH^{s+1/2}} \\
		&\le K\big(\|\eta\|_{H^2}\big) \left( \|v_z\|_{L^2_zH^{s+1/2}} + \||D_x|v\|_{L^2_zH^{s+1/2}} + \|T_{v_z}\eta\|_{L^2_zH^{s+3/2}} \right) \\
		&\le K\big(\|\eta\|_{H^2}\big) \left( \|\nabla_{x,z}v\|_{L^2_zH^{s+1/2}} + \|v_z\|_{L^2_zL^2} \|\eta\|_{C_*^{s+3/2}} \right) \\
		&\le K\big(\|\eta\|_{H^2}\big) \big( \|\eta\|_{H^{s+2}} \|\psi\|_{H^{1/2}} + \|\psi\|_{H^{s+1}} + \|\eta\|_{W^{s+2,\infty}}\|\psi\|_{H^{1}} \big).
	\end{align*}
	
	Besides, a direct calculus gives that 
	\begin{equation*}
		M^0_1(\alpha) + M^0_1(\alpha^{-1}) \le K\big(\|\eta\|_{H^2}\big) \big( \|\eta\|_{W^{2,\infty}} + 1 \big), \quad M^0_0(\alpha) + M^0_0(\alpha^{-1}) \le K\big(\|\eta\|_{H^2}\big),
	\end{equation*}
	which, thanks to symbolic calculus (see Proposition~\ref{prop-paradiff:SymbCal}), guarantees that
	\begin{align*}
		\| 1-T_{\alpha^{-1}}T_\alpha \|_{\mathcal{L}(H^{s+1/2})} \lesssim& M^0_1(\alpha^{-1}) M^0_0(\alpha) + M^0_0(\alpha^{-1}) M^0_1(\alpha) \\
		\le& K\big(\|\eta\|_{H^2}\big) \big( \|\eta\|_{W^{2,\infty}} + 1 \big).
	\end{align*}
	Then the estimate \eqref{eq-reg-dn:EstiF} follows.
\end{proof}

Once we have derived the parabolic equation \eqref{eq-reg-dn:MainEqV}, our next step will be to apply Proposition~\ref{prop-paradiff:Parabolic} to prove that the normal derivative $v_z$ can be expressed in terms of the action of the tangential operator $T_A$ on $v$, leading to the desired result.

\begin{proof}[Proof of Proposition~\ref{prop-reg-dn:Main}]
	From the definition \eqref{eq-reg-dn:FactoSymb} of $a$, we see that $a\in \Gamma^1_{1/2}$ (requiring only $\eta\in H^2$) and that $\RE a\ge c|\xi|$ for some constant $c>0$ depending on the $H^2$ norm of $\eta$. This fact enables us to apply Proposition~\ref{prop-paradiff:Parabolic} to the equation \eqref{eq-reg-dn:MainEqV} whose source term $f$ is estimated in \eqref{eq-reg-dn:EstiF}, to obtain
	\begin{align*}
		&\left\| v_z - T_Av - T_A T_{v_z}\eta \right\|_{C^0([-1,0];H^{s+1})} \\
		&\hspace{1em}\le K\big(\|\eta\|_{H^2}\big) \Big[ \big( \|\eta\|_{W^{2,\infty}}+1 \big)  \big( \|\eta\|_{H^{s+2}} \|\psi\|_{H^{1/2}} + \|\psi\|_{H^{s+1}} \big) + \|\eta\|_{W^{s+2,\infty}}\|\psi\|_{H^{1}} \Big],
	\end{align*}
	while the estimate \eqref{eq-reg-bas:EllipRegSobo} of $v$ and Proposition~\ref{prop-besov-para:NegInd} guarantee that
	\begin{align*}
		\left\| T_A T_{v_z}\eta \right\|_{C^0([-1,0];H^{s+1})} \le& K\big(\|\eta\|_{H^2}\big) \left\| T_{v_z}\eta \right\|_{C^0([-1,0];H^{s+2})} \\
		\le& K\big(\|\eta\|_{H^2}\big) \left\| v_z \right\|_{C^0([-1,0];L^2)} \|\eta\|_{W^{s+2,\infty}} \\
		\le& K\big(\|\eta\|_{H^2}\big) \|\eta\|_{W^{s+2,\infty}} \|\psi\|_{H^1}.
	\end{align*}
	Therefore, the trace of $v_z$ at $z=0$ satisfies
	\begin{equation}\label{eq-reg-dn:Nor-Tan-tame}
	\begin{aligned}
	    &\left\| v_z|_{z=0} - T_A \psi \right\|_{H^{s+1}} \\
		&\hspace{1em}\le K\big(\|\eta\|_{H^2}\big) \Big[ \big( \|\eta\|_{W^{2,\infty}}+1 \big)  \big( \|\eta\|_{H^{s+2}} \|\psi\|_{H^{1/2}} + \|\psi\|_{H^{s+1}} \big) + \|\eta\|_{W^{s+2,\infty}}\|\psi\|_{H^{1}} \Big].
	\end{aligned}
	\end{equation}
	Note that, when $s=0$, this estimate can be simplified as
	\begin{equation}\label{eq-reg-dn:Nor-Tan}
	    \left\| v_z|_{z=0} - T_A \psi \right\|_{H^{1}} \le K\big(\|\eta\|_{H^2}\big) \big( \|\eta\|_{W^{2,\infty}}+1 \big)  \|\psi\|_{H^{1}}.
	\end{equation}

	From the definition \eqref{eq-reg-bas:DN} of the Dirichlet-to-Neumann operator $\G(\eta)$, we have
	\begin{align*}
		\G(\eta)\psi =& (1+\eta_x^2)v_z|_{z=0} - \eta_x \partial_x\psi \\
		=& \left((1+\eta_x^2)T_A - \eta_x \partial_x \right)\psi + (1+\eta_x^2) (v_z|_{z=0} - T_A \psi) \\
		=& \left((1+\eta_x^2)\left(\frac{|D_x|}{\alpha} - \frac{\beta\partial_x}{2\alpha}\right) - \eta_x \partial_x \right)\psi \\
		&- (1+\eta_x^2)\left(\frac{|D_x|}{\alpha} - \frac{\beta\partial_x}{2\alpha} - T_A\right)\psi + (1+\eta_x^2) (v_z|_{z=0} - T_A \psi).
	\end{align*}
	Consequently, from the definition \eqref{eq-reg-bas:DefDeltagCoeff} of $\alpha,\beta$, the calculus above gives $\G(\eta)\psi = |D_x|\psi + \mathcal{R}(\eta)\psi$, with
	\begin{equation}\label{eq-reg-dn:DefR}
		\mathcal{R}(\eta)\psi = - (1+\eta_x^2)\left(\alpha^{-1}|D_x| - \beta\alpha^{-1}\partial_x - T_A\right)\psi + (1+\eta_x^2) (v_z|_{z=0} - T_A \psi).
	\end{equation}
	In the sequel, we will show that the right-hand side of \eqref{eq-reg-dn:DefR} belongs to 
	$H^{s+1}$ and satisfies the wanted estimate~\eqref{eq-reg-dn:Main-tame}, that is
	\begin{align*}
		\|\mathcal{R}(\eta)\psi\|_{H^{s+1}} \le K\big(\|\eta\|_{H^2}\big) &\Big[ \big( \|\eta\|_{W^{2,\infty}}+1 \big)  \big( \|\eta\|_{H^{s+2}} \|\psi\|_{H^{1/2}} + \|\psi\|_{H^{s+1}} \big) \\
		&+ \|\eta\|_{W^{s+2,\infty}}\|\psi\|_{H^{1}} \Big].
	\end{align*}
	In particular, this immediately implies the estimate \eqref{eq-reg-dn:Main} for $\sigma=1$. 
	By a duality argument, the operator $\mathcal{R}(\eta)$ can be extended as a bounded operator from $H^{-1}$ (the dual space of $H^1$) to itself with the same bound for the operator norm $\mathcal{L}(H^{-1})$, which proves~\eqref{eq-reg-dn:Main} for the case $\sigma=-1$. Then the intermediate cases where $\sigma\in(-1,1)$ can be obtained by interpolation.
	
	The second term in the right-hand side of \eqref{eq-reg-dn:DefR} can be treated simply by using the estimate \eqref{eq-reg-dn:Nor-Tan-tame}, \eqref{eq-reg-dn:Nor-Tan}, and Proposition~\ref{lemPa},~\ref{prop-besov-para:NegInd},
	\begin{align*}
		&\hspace{-2em}\| (1+\eta_x^2) (v_z|_{z=0} - T_A \psi) \|_{H^{s+1}} \\
		\le& \left\| \left( 1 - T_{1} \right) (v_z|_{z=0} - T_A \psi) \right\|_{H^{s+1}} + \left\| \left( \eta_x^2 - T_{\eta_x^2} \right) (v_z|_{z=0} - T_A \psi) \right\|_{H^{s+1}} \\
		&+ \left\| T_{\eta_x^2} (v_z|_{z=0} - T_A \psi) \right\|_{H^{s+1}} \\
		\lesssim& \|v_z|_{z=0} - T_A\|_{H^1} + \|v_z|_{z=0} - T_A\|_{H^{1}} \|\eta_x^2\|_{H^{s+1}} +  \|\eta_x^2\|_{H^1} \|v_z|_{z=0} - T_A\|_{H^{s+1}} \\
		\le& K\big(\|\eta\|_{H^2}\big) \Big[ \big( \|\eta\|_{W^{2,\infty}}+1 \big)  \big( \|\eta\|_{H^{s+2}} \|\psi\|_{H^{1/2}} + \|\psi\|_{H^{s+1}} \big) + \|\eta\|_{W^{s+2,\infty}}\|\psi\|_{H^{1}} \Big].
	\end{align*}
	In the last inequality, we used the fact that 
	$$
	\|\eta^2_x\|_{H^{s+1}} \le K\big(\|\eta\|_{H^2}\big) \|\eta\|_{H^{s+1}},
	$$
	which is a consequence of Proposition~\ref{prop-besov-bony:Paralin}.
	
	As for the first term, notice that, from the definition \eqref{eq-reg-dn:FactoSymb} and \eqref{eq-reg-bas:DefDeltagCoeff} of $A$ and $\alpha,\beta$, respectively, we have
	\begin{align*}
		&\hspace{-2em}- (1+\eta_x^2)\left(\frac{|D_x|}{\alpha} - \frac{\beta\partial_x}{2\alpha} - T_A\right)\psi \\
		=& (\alpha T_{\alpha^{-1}} -1 )|D_x|\psi + ( \alpha T_{\alpha^{-1}\eta_x} -\eta_x ) \partial_x\psi \\
		=& (1+\eta_x^2) \left(T_{|D_x|\psi} \alpha^{-1} + R(|D_x|\psi, \alpha^{-1}) + T_{\partial_x\psi} (\alpha^{-1}\eta_x) + R(\partial_x\psi, \alpha^{-1}\eta_x)  \right) .
	\end{align*}
	Thanks to Proposition~\ref{prop-besov-bony:Paralin}, the functions $\alpha^{-1}$ and $\alpha^{-1}\eta_x$ belong to $C_*^{s+1}$ with
	$$
	\|\alpha^{-1}\|_{C_*^{s+1}} + \|\alpha^{-1}\eta_x\|_{C_*^{s+1}} \le K\big(\|\eta\|_{H^2}\big) \|\eta\|_{W^{s+2,\infty}}.
	$$
	As a result, we may apply \eqref{eq-besov-bony:EstiRem} and \eqref{eq-besov-para:L2Hol} to obtain the following estimates
	\begin{align*}
	    &\left\| T_{|D_x|\psi} \alpha^{-1} + R(|D_x|\psi, \alpha^{-1}) + T_{\partial_x\psi} (\alpha^{-1}\eta_x) + R(\partial_x\psi, \alpha^{-1}\eta_x) \right\|_{H^{s+1}} \\
	    &\hspace{2em}\le K\big(\|\eta\|_{H^2}\big) \|\eta\|_{W^{s+2,\infty}} \|\psi\|_{H^{1}}.
	\end{align*}
	By repeating the argument for the second term, we are able to conclude the estimate \eqref{eq-reg-dn:Main-tame}.
\end{proof}

As a corollary, we extend Proposition~\ref{prop-reg-bas:InverBd} to the case with $\eta$ belonging to $H^{s+2}\cap W^{s+2,\infty}$. Moreover, a tame estimate will be established below.
\begin{corollary}\label{cor-reg-dn:InverBd}
    Given a real number $s\ge 0$, we assume that $\eta\in H^{s+2}\cap W^{s+2,\infty}$ and $u \in H^{s}$. Recall that the operator $\Id + \G(\eta) \colon H^{1/2} \to H^{-1/2}$ is invertible (see Proposition~\ref{prop-reg-bas:InverBd}). Then its inverse $(\Id+\G(\eta))^{-1}$ satisfies
    \begin{equation}\label{eq-reg-dn:InverBd}
    \begin{aligned}
        \left\| (\Id+\G(\eta))^{-1} u \right\|_{H^{s+1}} \le K\big(\|\eta\|_{H^2}\big) &\Big[ (\|\eta\|_{W^{2,\infty}} + 1) (\|\eta\|_{H^{s+2}}\|u\|_{H^{-1/2}} + \|u\|_{H^{s}}) \\
        &+ \|\eta\|_{W^{s+2,\infty}} \|u\|_{L^2} \Big].
    \end{aligned}
    \end{equation}
\end{corollary}
\begin{proof}
    As in the proof of Proposition~\ref{prop-reg-bas:InverBd}, we will use an induction in $s$ to prove the tame estimate \eqref{eq-reg-dn:InverBd}. Notice that the case $s=0$ is covered by Proposition~\ref{prop-reg-bas:InverBd}. In what follows, we assume that \eqref{eq-reg-dn:InverBd} holds for $s\in[1,s_1]$ for some $s_1\ge 0$, and show that \eqref{eq-reg-dn:InverBd} also holds for $s=s_1+\epsilon$ and any $\epsilon\in[0,1]$.

    Set $\psi=(\Id+\G(\eta))^{-1}u$. 
	According to the formula \eqref{eq-reg-bas:Paralin}, we can write $\G(\eta)$ as $|D_x| + \mathcal{R}(\eta)$ and obtain
	\begin{equation*}
		u = \psi + \mathcal{G}(\eta)\psi = (\Id+\Dx)\psi + \mathcal{R}(\eta) \psi,
	\end{equation*}
	which implies
	\begin{equation}\label{eq-reg-dn:InverFormu}
		\psi = (\Id+\Dx)^{-1} u - (\Id+\Dx)^{-1}\mathcal{R}(\eta) \psi.
	\end{equation}
	Then the $H^{s_1+\epsilon+1}$ estimate of $\psi$ reduces to the bound of $H^{s_1+\epsilon}$ norm of $\mathcal{R}(\eta)\psi$, which is given by \eqref{eq-reg-dn:Main-tame}. More precisely, we have
	\begin{align*}
	    &\hspace{-2em}\|\psi\|_{H^{s_1+\epsilon+1}} \lesssim \|u\|_{H^{s_1+\epsilon}} + \|\mathcal{R}(\eta) \psi\|_{H^{s_1+\epsilon}} \lesssim \|u\|_{H^{s_1+\epsilon}} + \|\mathcal{R}(\eta) \psi\|_{H^{s_1+1}} \\
	    \le& C\|u\|_{H^{s_1+\epsilon}} + K\big(\|\eta\|_{H^2}\big) \Big[ (\|\eta\|_{W^{2,\infty}} + 1) (\|\eta\|_{H^{s_1+1}}\|\psi\|_{H^{1}} + \|\psi\|_{H^{s_1+1}}) \\
	    &+ \|\eta\|_{W^{s_1+2,\infty}} \|\psi\|_{H^1} \Big].
	\end{align*}
	Note that the control of $H^1$ and $H^{s_1+1}$ norm of $\psi$ has been given by \eqref{eq-reg-bas:InverBdSobo} and the assumption that \eqref{eq-reg-dn:InverBd} is true for $s\in[1,s_1]$, respectively. By inserting these estimates into the inequality above, we complete the proof of \eqref{eq-reg-dn:InverBd} for $s=s_1+\epsilon$.
\end{proof}

To end this section, we use a similar argument to obtain the $Y^{s+1}$ (defined in Definition~\ref{def-intro-main:HolSpaceVar}) estimate for $(\Id+\G(\eta))^{-1}u$. Roughly speaking, when $\eta\in H^2 \cap Y^2$, $(\Id+\G(\eta))^{-1}$ elevates $1$ regularity in the sense that it sends functions in $H^s\cap Y^s$ into the space $H^{s+1}\cap Y^{s+1}$. For the sake of simplicity, we focus on the case $s\in[-1,0]$, which is enough for its application in the next section.

\begin{corollary}\label{cor-reg-dn:RegPsi}
	Consider a real number $s\in [-1 ,0]$. 
	There exists a non-decreassing function $K\colon \R_+\to\R_+$ such that, for all 
	$\eta\in H^2\cap Y^2$ and for all 
	$u\in H^s\cap Y^s$, 
	\begin{equation}\label{eq-reg-dn:RegPsi}
		\|(\Id+\G(\eta))^{-1}u\|_{Y^{s+1}} 
		\le K\big( \|\eta\|_{H^{2}} \big) \Big( \|u\|_{Y^s} + \big(\|\eta\|_{W^{2,\infty}} + 1\big)\|u\|_{H^{s}} \Big).
	\end{equation}
\end{corollary}
\begin{proof}
Set $\psi=(\Id+\G(\eta))^{-1}u$. Recall that, in the proof of Corollary~\ref{cor-reg-dn:InverBd}, we have obtained the formula \eqref{eq-reg-dn:InverFormu},
	\begin{equation*}
		\psi = (\Id+\Dx)^{-1} u - (\Id+\Dx)^{-1}\mathcal{R}(\eta) \psi.
	\end{equation*}
	Since $(\Id+\Dx)^{-1}$ is a Fourier multiplier whose symbol $(1+|\xi|)^{-1}$ is a radial of order $-1$, it follows from 
	Proposition \ref{prop-pre-fct:BdMult} combined with part (3) of Proposition \ref{prop-pre-fct:Basic}, that
	\begin{align*}
		\|\psi\|_{Y^{s+1}} &\le \|(\Id+\Dx)^{-1} u\|_{Y^{s+1}} + \|(\Id+\Dx)^{-1}\mathcal{R}(\eta) \psi\|_{Y^{s+1}} \\[1ex]
		&\lesssim \|u\|_{Y^s} + \|(\Id+\Dx)^{-1}\mathcal{R}(\eta) \psi\|_{H^{s+2}} \lesssim \|u\|_{Y^s} + \|\mathcal{R}(\eta) \psi\|_{H^{s+1}}.
	\end{align*}
	Now, using \eqref{eq-reg-bas:InverBdSobo} and \eqref{eq-reg-dn:Main}, we get
	\begin{align*}
		\|\psi\|_{Y^{s+1}} &\lesssim \|u\|_{Y^s} + \|\mathcal{R}(\eta) (\Id+\mathcal{G}(\eta))^{-1}u\|_{H^{s+1}} \\
		\lesssim & \|u\|_{Y^s} + K\left( \|\eta\|_{H^{2}}\right)\left(\|\eta\|_{W^{2,\infty}} + 1\right)\|(\Id+\mathcal{G}(\eta))^{-1}u\|_{H^{s+1}} \\
		\lesssim & \|u\|_{Y^s} + K\left( \|\eta\|_{H^{2}}\right)\left(\|\eta\|_{W^{2,\infty}} + 1\right) \|u\|_{H^{s}},
	\end{align*}
	which gives the desired inequality.
\end{proof}

\subsection{Shape derivative in low regularity}\label{subsect-reg:ShpDer}
Finally, we recall the \emph{shape derivative formula} due to Lannes, which clarifies the dependence of the Dirichlet-to-Neumann operator \(\mathcal{G}(\eta)\) on \(\eta\). For a comprehensive treatment of the general theory, we refer to Appendix A.1 of Lannes' book \cite{lannes2013water}. For our purposes, we extend Lannes' result in one direction: we allow for more general functions \(\psi\). Specifically, we consider \(\psi \in H^{1/2}\) (instead of \(\psi \in H^{3/2}\) in \cite{lannes2013water}), which is of independent interest since this is optimal in view of the variational definition of \(\mathcal{G}(\eta)\).

\begin{theorem}[Shape derivative]\label{thm-reg-shpder:Main}
    Consider two real numbers $s,\sigma$ satisfying
    \begin{equation*}
    	s\ge 1\quad \text{and} \quad \frac{1}{2} \le \sigma \le s+\frac{1}{2}.
    \end{equation*}
    For all $\eta,\delta\eta\in H^{s+1}$ and $\psi\in H^\sigma$, the following limit exists in $H^{\sigma-1}$,
    \begin{equation}\label{eq-reg-shpder:Main}
        \lim_{\varepsilon\rightarrow 0}\frac{ \mathcal{G}(\eta+\varepsilon\delta\eta)\psi - \mathcal{G}(\eta)\psi }{\varepsilon} 
        = -\mathcal{G}(\eta)\left(\delta\eta \mathcal{B}(\eta)\psi \right) - \partial_x\left(\delta\eta\mathcal{V}(\eta)\psi \right).
    \end{equation}
    Note that the right-hand side is defined in the sense of distribution $\mathcal{D}'(\R)$. Namely, 
    for all test function $\varphi\in C^\infty_0(\R)$,
    \begin{equation*}
    	\begin{aligned}
    		&\langle -\mathcal{G}(\eta)\left(\delta\eta \, \mathcal{B}(\eta)\psi \right) - \partial_x\left(\delta\eta\, \mathcal{V}(\eta)\psi \right), \varphi \rangle_{\mathcal{D}'(\R)\times\mathcal{D}(\R)} \\
    		&\hspace{6em}= \int_{\R}\left[ -(\delta\eta\,\G(\eta)\varphi) \mathcal{B}(\eta)\psi + (\delta\eta\,\partial_x\varphi)\mathcal{V}(\eta)\psi \right] \dx,
    	\end{aligned}
    \end{equation*}
    where the integral in the right-hand side makes sense due to Corollary~\ref{cor-besov-bony:ProdLaw}. Moreover, it satisfies
    \begin{equation}\label{eq-reg-shpder:BasicEsti}
        \big\|\mathcal{G}(\eta)\left(\delta\eta \mathcal{B}(\eta)\psi \right) 
        +\partial_x\left(\delta\eta\mathcal{V}(\eta)\psi \right) \big\|_{H^{\sigma-1}} 
        \le K(\|\eta\|_{H^{s+1}}) \|\delta\eta\|_{H^{s+1}} \|\psi\|_{H^{\sigma}}.
    \end{equation}
\end{theorem}
\begin{proof}	A detailed proof in general dimension is given in Appendix~\ref{app:ShpDer}.\end{proof}

To introduce our next result, the important observation is that the three operators
\[
\psi \mapsto \mathcal{G}(\eta)\psi, \quad \psi \mapsto \mathcal{B}(\eta)\psi, \quad\text{and}\quad 
\psi \mapsto \mathcal{V}(\eta)\psi\] are of order \(1\). More precisely, they are bounded from \(H^\sigma\) to \(H^{\sigma-1}\) under appropriate assumptions on \(\sigma\) and regularity of \(\eta\). In this context, the right-hand side of \eqref{eq-reg-shpder:Main} might initially appear to be an operator of order \(2\). However, according to the estimate \eqref{eq-reg-shpder:BasicEsti}, it is actually an operator of order \(1\) in the sense that it belongs to \(L^2\) for any \(\eta, \delta\eta \in H^2\) and \(\psi \in H^1\). This indicates the presence of a hidden cancellation. The purpose of this section is to clarify this cancellation and to prove an alternative version of \eqref{eq-reg-shpder:BasicEsti}, tailored to our problem, that requires \emph{less regularity} of~\(\delta\eta\). 

This is one reason why we need to work with the Hardy spaces $Y^1$ introduced in Section~\ref{subsect-pre:Fct}. Indeed, we shall need to control the $L^\infty$-norm of  $\B(\eta)\psi$ and $\V(\eta)\psi$. Now, notice that, already when $\eta=0$, we have  $\B(0)\psi=\Dx \psi$ 
and $\V(0)\psi=\partial_x\psi$. Therefore,
$$
\lA \psi\rA_{Y^1}\sim \la \psi\ra_{L^\infty}+\la \B(0)\psi\ra_{L^\infty}+\la \V(0)\psi\ra_{L^\infty}.
$$

\begin{proposition}\label{prop-reg-shpder:sharp}
Suppose $\eta\in H^2\cap W^{2,\infty}$ and $\psi\in H^1\cap Y^1$. Then for all $\delta\eta\in H^2\cap W^{2,\infty}$, we have
    \begin{equation}\label{eq-reg-shpder:sharp}
        \begin{aligned}
            &\big\| \G(\eta)(\B(\eta)\psi \cdot\delta\eta) + \partial_x(\V(\eta)\psi \cdot\delta\eta) \big\|_{L^2} \\
            &\hspace{8em}\le K\big(\|\eta\|_{H^2}\big)\big( \|\psi\|_{Y^1} + \left( \|\eta\|_{W^{2,\infty}}+1 \big)\|\psi\|_{H^1} \right)\|\delta\eta\|_{H^1}.
        \end{aligned}
    \end{equation}
\end{proposition}

The key feature of Proposition \ref{prop-reg-shpder:sharp} is that the right-hand side involves \emph{only the $H^1$ norm of $\delta\eta$}, although it is assumed to be of higher regularity. We would like to explain the core difficulty of using only the $H^1$ information of the variation $\delta\eta$. Through the cancellation \eqref{eq-reg-shpder:Cancel} below, we can rewrite the right-hand side of \eqref{eq-reg-shpder:Main} as
$$
-[\G(\eta),\delta\eta]\B(\eta)\psi - (\partial_x\delta\eta)\V(\eta)\psi.
$$
Thus Proposition~\ref{prop-reg-shpder:sharp} follows from the fact that the commutator $[\G(\eta),\delta\eta]$ maps $H^1$ to $L^2$, which can be further reduced to a commutator estimate of the form
$$
\big\|[ \Dx,\delta\eta ]B\big\|_{L^2}
\lesssim \|\delta\eta\|_{H^1}|B|_{L^\infty},
$$
established in Lemma \ref{Commu_Esti}. 

Before proving Proposition~\ref{prop-reg-shpder:sharp}, we first show that if the boundary value $\psi$ is in $H^1 \cap Y^1$, then the boundary velocities $\B(\eta)\psi$ and $\V(\eta)\psi$ belong to $Y^0$.

\begin{proposition}\label{bdr_vel_bound}
    Suppose $\eta\in H^2\cap W^{2,\infty}$ and $\psi\in H^1\cap Y^1$. Then $\B(\eta)\psi$ and $\V(\eta)\psi$ belong to $L^2\cap Y^0$, together with the estimate
    \begin{equation}\label{eq-reg-shpder:bdr_vel_bound}
        \|\B(\eta)\psi\|_{Y^0} + \|\V(\eta)\psi\|_{Y^0} \le K\big(\|\eta\|_{H^2}\big)\big( \|\psi\|_{Y^1} + \big( \|\eta\|_{W^{2,\infty}}+1 \big)\|\psi\|_{H^1} \big).
    \end{equation}
\end{proposition}
\begin{proof}
    As indicated in \eqref{eq-reg-bas:BV}, $\B(\eta)\psi,\V(\eta)\psi$ can be expressed in terms of $\G(\eta)\psi$ and $\partial_x\psi$. Hence, our strategy is to check that $\G(\eta)\psi, \partial_x\psi \in L^2\cap Y^0$ and then apply Proposition \ref{prop-pre-fct:Prod} to complete the proof.
    
    Writing $\partial_x = \Dx\Hi$, we can apply Proposition \ref{prop-pre-fct:BdMult} to deduce that
    \begin{equation*}
        \|\partial_x\psi\|_{Y^0} \lesssim \|\psi\|_{Y^1}.
    \end{equation*}
    To show that $\G(\eta)\psi \in Y^0$, one may use the decomposition $\G(\eta) = \Dx + \mathcal{R}(\eta)$ (see \eqref{eq-reg-bas:Paralin}), to infer that
    \begin{align*}
        \|\G(\eta)\psi\|_{Y^0} &\le \|\Dx\psi\|_{Y^0} + \|\mathcal{R}(\eta)\psi\|_{Y^0} \\
        &\le C \|\psi\|_{Y^1} + K\big(\|\eta\|_{H^2}\big)\big( \|\eta\|_{W^{2,\infty}}+1 \big)\|\psi\|_{H^1},
    \end{align*}
    where we used Proposition \ref{prop-pre-fct:BdMult}, the continuous embedding $H^1\subset Y^0$ and the fact that $\mathcal{R}(\eta)$ is bounded on $H^1$ (see Proposition \ref{prop-reg-dn:Main}).
    
    To estimate $\B(\eta)\psi$, we express it as in \eqref{eq-reg-bas:BV},
    \begin{equation*}
    \begin{aligned}
    \B(\eta)\psi
    &= \frac{1}{1+\eta_x^2}\G(\eta)\psi + \frac{\eta_x}{1+\eta_x^2}\partial_x\psi \\
    &= \G(\eta)\psi - \frac{\eta_x^2}{1+\eta_x^2}\G(\eta)\psi + \frac{\eta_x}{1+\eta_x^2}\partial_x\psi.
    \end{aligned}
    \end{equation*}
    Thanks to Proposition \ref{prop-besov-bony:Paralin}, the coefficients before $\G(\eta)\psi$ and $\partial_x\psi$ belong to $H^1$, provided that $\eta\in H^2$. Then the desired estimate \eqref{eq-reg-shpder:bdr_vel_bound} for $\B(\eta)\psi$ follows from \eqref{eq-pre-fct:Prod}. The estimate for $\V(\eta)\psi$ can be obtained in the same way. Indeed, it suffices to observe from \eqref{eq-reg-bas:PsibyBV} that 
    $$
    \V(\eta)\psi = \partial_x\psi - \eta_x \B(\eta)\psi,
    $$
    where $\eta_x\in H^1$ and $\B(\eta)\psi \in Y^0$. Then, the desired estimate \eqref{eq-reg-shpder:bdr_vel_bound} for $\V(\eta)\psi$ also follows from \eqref{eq-pre-fct:Prod}.
\end{proof}

\begin{proof}[Proof of Proposition \ref{prop-reg-shpder:sharp}]
For simplicity, we denote
\begin{equation*}
B := \mathcal{B}(\eta)\psi,\ \ V := \mathcal{V}(\eta)\psi,
\end{equation*}
which, due to Proposition \ref{prop-reg-bas:BdBV} and \ref{bdr_vel_bound}, satisfy
\begin{align}
	&\|B\|_{L^2} + \|V\|_{L^2} \le K\big(\|\eta\|_{H^2}\big)\|\psi\|_{H^1},\label{n31}\\
&\|B\|_{Y^0} + \|V\|_{Y^0} \le K\big(\|\eta\|_{H^2}\big)\left( \|\psi\|_{Y^1} + \big( \|\eta\|_{W^{2,\infty}}+1 \big)\|\psi\|_{H^1} \right).\label{n32}
\end{align}
The crux of the proof is to use the cancellation 
\begin{equation}\label{eq-reg-shpder:Cancel}
   \mathcal{G}(\eta)B =- V_x, 
\end{equation}
proved for instance in Lemma 1 of \cite{bona2008asynptotic} (see also Remark 2.13 of \cite{ABZ2011}). To use this identity, we shall commute $\delta\eta$ and $\G(\eta)$ which in turn will be made possible by paralinearization of the Dirichlet-to-Neumann operator. More precisely, we begin by writing $\mathcal{G}(\eta)$ as $\Dx+\mathcal{R}(\eta)$ (see \eqref{eq-reg-bas:Paralin}). Then, we compute
\begin{equation}\label{DN-Shape-Cancel}
\begin{aligned}
\mathcal{G}(\eta)&(B \delta\eta) + \partial_x(V \delta\eta) \\
&= \Dx(B \delta\eta) + \mathcal{R}(\eta)(B \delta\eta) + V_x\delta\eta + V (\delta\eta)_x \\
&= ( \Dx B + V_x )\delta\eta + \mathcal{R}(\eta)(B \delta\eta) + [ \Dx,\delta\eta ]B + (\delta\eta)_x V \\
&= (\mathcal{G}(\eta)B + V_x) \delta\eta - (\mathcal{R}(\eta)B) \delta\eta + \mathcal{R}(\eta)(B \delta\eta) 
+ [ \Dx,\delta\eta ]B + (\delta\eta)_xV\\
&=  - (\mathcal{R}(\eta)B) \delta\eta + \mathcal{R}(\eta)(B \delta\eta) + (\delta\eta)_x V + [ \Dx,\delta\eta ]B.
\end{aligned}
\end{equation}
Once the cancellation (\ref{eq-reg-shpder:Cancel}) is used, we will estimate the $L^2$-norm of each term separately.

The first three terms in the right-hand-side of (\ref{DN-Shape-Cancel}) are easily estimated. Using \eqref{eq-reg-dn:Main} applied with $\sigma=0$ and \eqref{n31}, we find
\begin{align*}
\|(\mathcal{R}(\eta)B) \delta\eta\|_{L^2} 
&\le \|\mathcal{R}(\eta)B\|_{L^2} |\delta\eta|_{L^\infty} \\
&\le K\big(\|\eta\|_{H^2}\big) \big( \|\eta\|_{W^{2,\infty}} + 1 \big)\|B\|_{L^2} \|\delta\eta\|_{H^1} \\
&\le K\big(\|\eta\|_{H^2}\big) \big( \|\eta\|_{W^{2,\infty}} + 1 \big)\|\psi\|_{H^1} \|\delta\eta\|_{H^1},
\end{align*}
and similarly
\begin{align*}
\|\mathcal{R}(\eta)(B \delta\eta)\|_{L^2} 
&\le K\big(\|\eta\|_{H^2}\big) \big( \|\eta\|_{W^{2,\infty}} + 1 \big)\|B \delta\eta\|_{L^2} \\
&\le K\big(\|\eta\|_{H^2}\big) \big( \|\eta\|_{W^{2,\infty}} + 1 \big)\|B\|_{L^2} \|\delta\eta\|_{H^1} \\
&\le K\big(\|\eta\|_{H^2}\big) \big( \|\eta\|_{W^{2,\infty}} + 1 \big)\|\psi\|_{H^1} \|\delta\eta\|_{H^1}.
\end{align*}
The third term in (\ref{DN-Shape-Cancel}) is estimated by Proposition \ref{bdr_vel_bound}:
$$
\begin{aligned}
\|(\delta\eta)_xV\|_{L^2}
&\le \|\delta\eta\|_{H^1}|V|_{L^\infty}\\
&\le K\big(\|\eta\|_{H^2}\big)
\big( \|\psi\|_{Y^1} + \big( \|\eta\|_{W^{2,\infty}}+1 \big)\|\psi\|_{H^1} \big) \|\delta\eta\|_{H^1}.
\end{aligned}
$$

The most problematic term in (\ref{DN-Shape-Cancel}) is the fourth one. This is where we use the commutator estimate given by  Lemma \ref{Commu_Esti}. Namely, we write
$$
\begin{aligned}
\big\|[ \Dx,\delta\eta ]B\big\|_{L^2}
&\le C\|\delta\eta\|_{H^1}|B|_{L^\infty}\\
&\le
K\big(\|\eta\|_{H^2}\big)
\big( \|\psi\|_{Y^1} + \left( \|\eta\|_{W^{2,\infty}}+1 \big)\|\psi\|_{H^1} \right) \|\delta\eta\|_{H^1}.
\end{aligned}
$$
This completes the proof.
\end{proof}

\section{Paralinearization of the Problem}\label{sect:Para}

Recall that our aim is to solve the following Cauchy problem:
\begin{equation}\label{n40}
	\left\{\begin{aligned}
		&\partial_t \eta = \mathcal{G}(\eta)\psi, \\ 
		&\partial_t \psi + \eta + \partial_t^2\eta + \partial_x^4 \eta+ \frac{1}{2}(\partial_x\psi)^2 -\frac{1}{2}\frac{\left(\partial_x\eta\cdot\partial_x\psi + \mathcal{G}(\eta)\psi\right)^2}{1+(\partial_x\eta)^2} = 0, \\
		&(\eta,\psi)|_{t=0} = (\eta_0,\psi_0).
	\end{aligned}\right.
\end{equation}
As explained in the introduction, in the low regularity regime, the time derivative $\partial_t^2\eta$ might be ill-defined, while $\partial_tu=\partial_t(\psi+\partial_t\eta)$ as a whole can still make sense. Therefore, it is convenient to introduce the unknown
\begin{equation}\label{n41}
	u := \psi + \mathcal{G}(\eta)\psi = \psi + \partial_t\eta.
\end{equation}
The equations satisfied by $(\eta,u)$ then become
\begin{equation}\label{n42}
	\left\{\begin{array}{l}
		\partial_t \eta = \mathcal{G}(\eta)(\Id + \mathcal{G}(\eta))^{-1}u, \\ [0.5ex]
		\partial_t u + \partial_x^4 \eta + \eta + N(\eta,u) = 0, \\ [0.5ex]
		(\eta,u)|_{t=0} = (\eta_0,u_0),
	\end{array}\right.
\end{equation}
where 
\begin{equation}\label{n40.5}
    N(\eta,u) = \frac{1}{2}(\partial_x\psi)^2 -\frac{1}{2}\frac{\left(\partial_x\eta\cdot\partial_x\psi + \mathcal{G}(\eta)\psi\right)^2}{1+(\partial_x\eta)^2}\quad\text{with}\quad\psi=(\Id+\mathcal{G}(\eta))^{-1}u.
\end{equation}
Notice that, by Proposition~\ref{prop-reg-bas:InverBd}, the term  $(\Id+\mathcal{G}(\eta))^{-1}u$ is well-defined for $u\in L^2_x$. 

The key to deal with the nonlinear system \eqref{n42} is to extract the nonlinear components and treat them as source terms. In this section, we shall examine the regularity of these source terms and prove their Lipschitz dependence on the unknown $(\eta,u)$ in a weaker norm. Importantly, all estimates derived in this section are uniform in time. As explained in Subsection \ref{Notations}, the results are therefore stated in terms of norms of spatial functions only.

\subsection{Paralinearization of the system}\label{subsect-para:Sys}

In this section, we establish the regularity of the nonlinear components, assuming that the regularity of $(\eta, u)$ corresponds to the energy level $H^2 \times L^2$ or a higher level, $H^{s+2} \times H^s$ with $s > 0$. The high order estimates obtained here will be used to deduce the propagation of regularity in Subsection~\ref{High_Reg}, which requires all the estimates to be tame. Namely, the upper bounds of the nonlinear components should be linear in the highest regularity norm, with coefficients depending only on the norms at the energy level. Furthermore, as suggested by the Strichartz estimate (see Proposition~\ref{prop-pre-stri:DisperEsti}), the total power of the H{\"o}lder norms $W^{2,\infty}$ and $W^{s+2,\infty}$ need to be strictly less than $4$ to ensure that the high order estimates in Subsection~\ref{High_Reg} hold at least for small time.

We begin by paralinearizing the operator $\mathcal{G}(\eta)(\Id + \mathcal{G}(\eta))^{-1}$ which appears in \eqref{n42}. 
To do so, we apply Proposition \ref{prop-reg-dn:Main}. Remembering the notation $\mathcal{R}(\eta)=\mathcal{G}(\eta)-\Dx$, and using the estimates \eqref{eq-reg-bas:InverBdSobo} and \eqref{eq-reg-dn:Main}, we obtain the following proposition.

\begin{proposition}\label{prop-para-sys:ParaEq1}
Consider a real number $s\ge 0$. Let $\theta(\xi)\geq1/4$ be a smooth, even function that equals $1/4$ for $|\xi|\leq1/2$ and $|\xi|/(1+|\xi|)$ for $|\xi|>1$. If $\eta\in H^{s+2} \cap W^{s+2,\infty}$, then for all $u\in H^{s+2}$, the right-hand side of the first equation in \eqref{n42} 
	admits the following decomposition
	\begin{equation}\label{eq-para-sys:ParaEq1}
		\mathcal{G}(\eta)(\Id+\mathcal{G}(\eta))^{-1} u = 
		\theta(D_x)u
		+Q(\eta,u)
	\end{equation}
	where the remainder $Q(\eta,u)$ equals 
	\begin{equation}\label{eq-para-sys:Q}
	    Q(\eta,u) = \big(\Dx(\Id+\Dx)^{-1}-\theta(D_x)\big)u + (\Id+\Dx)^{-1}\mathcal{R}(\eta)(\Id+\mathcal{G}(\eta))^{-1}u.
	\end{equation}
	
	Moreover, $Q(\eta,u)$ is linear in $u$ and satisfies
	\begin{equation}\label{eq-para-sys:ParaEq1Rem}
		\begin{aligned}
        \left\| Q(\eta,u) \right\|_{H^{s+2}} \le K\big(\|\eta\|_{H^2}\big) &\Big[ (\|\eta\|_{W^{2,\infty}} + 1) (\|\eta\|_{H^{s+2}}\|u\|_{L^2} + \|u\|_{H^{s}}) \\
        &+ \|\eta\|_{W^{s+2,\infty}} \|u\|_{L^2} \Big].
    \end{aligned}
	\end{equation}
\end{proposition}

Regarding the nonlinear term $N(\eta,u)$ in the evolution equation for $u$, we have the following estimate.

\begin{proposition}\label{prop-para-sys:ParaEq2}
	Given a real number $s\ge 0$, if $\eta\in H^{s+2} \cap W^{s+2,\infty}$ and $u\in H^{s} \cap Y^{s}$, the nonlinearity $N(\eta,u)$ in \eqref{eq-intro-csz:MainSys} is a symmetric bilinear form of $u$, depending on $\eta$, satisfying
	\begin{equation}\label{eq-para-sys:EstiF2}
		\|N(\eta,u)\|_{L^2} \le K\big( \|\eta\|_{H^{2}} \big) \Big(\big(\|\eta\|_{W^{2,\infty}}+1\big)\|u\|_{L^2} + \|u\|_{Y^0} \Big) \|u\|_{L^2},
	\end{equation}
	and 
	\begin{equation}\label{eq-para-sys:EstiF2-tame}
	\begin{aligned}
	    \|N(\eta,u)\|_{H^{s}} \le& K\big( \|\eta\|_{H^{2}} \big) \Big(\big(\|\eta\|_{W^{2,\infty}}+1\big)\|u\|_{L^2} + \|u\|_{Y^0} \Big) \Big[\|\eta\|_{H^{s+2}}\|u\|_{L^2}  \\
        &+ \|\eta\|_{W^{s+2,\infty}} \|u\|_{L^2} + (\|\eta\|_{W^{2,\infty}} + 1) (\|\eta\|_{H^{s+2}}\|u\|_{L^2} + \|u\|_{H^{s}}) \Big].
	\end{aligned}
	\end{equation}
\end{proposition}
\begin{proof}
Let us recall that 
$$
N(\eta,u)
=\frac{1}{2}(\partial_x\psi)^2-\frac{1}{2}\frac{\left(\partial_x\eta\partial_x\psi + \mathcal{G}(\eta)\psi\right)^2}{1+(\partial_x\eta)^2},
$$
where $u$ and $\psi$ are related by $u = (\Id+ \mathcal{G}(\eta))\psi$. Note that the right-hand side can be considered as the sum of terms taking the form
$$
F(\eta_x) \mathcal{A}_1\psi \mathcal{A}_2 \psi,
$$
where $F$ is a smooth function, and the operators $\mathcal{A}_1,\mathcal{A}_2$ are taken from $\{ \partial_x,\ \G(\eta) \}$. The proof of $H^s$ regularity of $N(\eta,u)$ is based on the following decomposition :
\begin{align*}
    F(\eta_x) \mathcal{A}_1\psi \mathcal{A}_2 \psi =& \left( (F(\eta_x)-F(0)) - T_{F(\eta_x)-F(0)} \right) (\mathcal{A}_1\psi \mathcal{A}_2 \psi) \\
    &+ T_{\mathcal{A}_1\psi \mathcal{A}_2 \psi} (F(\eta_x)-F(0)) + F(0)\mathcal{A}_1\psi \mathcal{A}_2 \psi.
\end{align*}
In fact, on the right-hand side of the formula above, one may apply Proposition~\ref{lemPa} and~\ref{prop-besov-para:NegInd} to the first term and second term, respectively, together with Proposition~\ref{prop-besov-bony:Paralin} to deal with $F(\eta_x)-F(0)$. More precisely, for all $s\ge0$, we have
\begin{align*}
    \left\| F(\eta_x) \mathcal{A}_1\psi \mathcal{A}_2 \psi \right\|_{H^s} \lesssim& \left( \left\| F(\eta_x)-F(0) \right\|_{H^1} + |F(0)| \right) \|\mathcal{A}_1\psi \mathcal{A}_2 \psi\|_{H^s} \\
    &+ \left\| F(\eta_x)-F(0) \right\|_{H^{s+1}} \|\mathcal{A}_1\psi \mathcal{A}_2 \psi\|_{L^2} \\
    \le& \Big( K\big(|\eta_x|_{L^\infty}\big)\|\eta_x\|_{H^1} + |F(0)| \Big) \|\mathcal{A}_1\psi \mathcal{A}_2 \psi\|_{H^s} \\
    &+ K\big(|\eta_x|_{L^\infty}\big)\|\eta_x\|_{H^{s+1}} \|\mathcal{A}_1\psi \mathcal{A}_2 \psi\|_{L^2} \\
    \le& K\big(\|\eta\|_{H^2}\big) \left( \|\mathcal{A}_1\psi \mathcal{A}_2 \psi\|_{H^s} + \|\eta\|_{H^{s+2}} \|\mathcal{A}_1\psi \mathcal{A}_2 \psi\|_{L^2} \right).
\end{align*}
Then it suffices to study the $H^s$ and $L^2$ norm of $\mathcal{A}_1\psi \mathcal{A}_2 \psi$. Thanks to the estimate \eqref{eq-besov-bony:ProdLaw-tame}, we have, for all $s\ge 0$,
$$
\|\mathcal{A}_1\psi \mathcal{A}_2 \psi\|_{H^s} \lesssim  |\mathcal{A}_1\psi |_{L^\infty} \|\mathcal{A}_2\psi\|_{H^s} + |\mathcal{A}_2\psi |_{L^\infty} \|\mathcal{A}_1\psi\|_{H^s},
$$
which reduces the problem to the following inequalities :
\begin{align}
    &|\mathcal{A}_j \psi|_{L^\infty} \le K\big( \|\eta\|_{H^{2}} \big) \Big(\big(\|\eta\|_{W^{2,\infty}}+1\big)\|u\|_{L^2} + \|u\|_{Y^0} \Big), \label{eq-para-sys:EstiF2-S1} \\
    &\|\mathcal{A}_j \psi\|_{L^2} \le K\big( \|\eta\|_{H^{2}} \big) \|u\|_{L^2} \label{eq-para-sys:EstiF2-S2} \\
    &\begin{aligned}
        \|\mathcal{A}_j \psi\|_{H^s} \le& K\big( \|\eta\|_{H^{2}} \big) \Big[\|\eta\|_{H^{s+2}}\|u\|_{L^2} + \|\eta\|_{W^{s+2,\infty}} \|u\|_{L^2} \\
        &+ (\|\eta\|_{W^{2,\infty}} + 1) (\|\eta\|_{H^{s+2}}\|u\|_{L^2} + \|u\|_{H^{s}}) \Big], \label{eq-para-sys:EstiF2-S3}
    \end{aligned}
\end{align}
where $j=1,2$. In fact, the proof of \eqref{eq-para-sys:EstiF2-S2} and \eqref{eq-para-sys:EstiF2-S3} is no more than a combination of Proposition~\ref{prop-reg-bas:Bd} and Corollary~\ref{cor-reg-dn:InverBd} (one should apply Proposition~\ref{prop-reg-bas:InverBd} instead of Corollary~\ref{cor-reg-dn:InverBd} to deduce \eqref{eq-para-sys:EstiF2-S2}). To prove the first estimate \eqref{eq-para-sys:EstiF2-S1}, we use Proposition~\ref{prop-pre-fct:BdMult} and Corollary~\ref{cor-reg-dn:RegPsi} to conclude the case $\mathcal{A}_j=\partial_x$,
$$
|\partial_x\psi|_{L^\infty}
\lesssim \|\psi\|_{Y^1}
\le C\big( \|\eta\|_{H^{2}} \big) \Big(\big(\|\eta\|_{W^{2,\infty}}+1\big)\|u\|_{L^2} + \|u\|_{Y^0} \Big).
$$
When $\mathcal{A}_j=\G(\eta)$, by writing $\mathcal{G}(\eta) = \Dx + \mathcal{R}(\eta)$ and $\psi=(\Id + \mathcal{G}(\eta))^{-1}u$, we have
\begin{equation}\label{DN_Linfty}
\begin{aligned}
    |\mathcal{G}(\eta)\psi|_{L^\infty} &\le \left| \Dx\psi \right|_{L^\infty} + |\mathcal{R}(\eta)(\Id + \mathcal{G}(\eta))^{-1}u|_{L^\infty} \\
    &\lesssim \|\psi\|_{Y^1} 
    + \big\|\mathcal{R}(\eta)(\Id + \mathcal{G}(\eta))^{-1}u\big\|_{H^1} \\
    &\lesssim 
    \|\psi\|_{Y^1} + \big\|\mathcal{R}(\eta);\mathcal{L}(H^1)\big\|
    \cdot \big\|(\Id + \mathcal{G}(\eta))^{-1};\mathcal{L}(L^2;H^1)\big\|
    \cdot\|u\|_{L^2} \\
    &\le K\big( \|\eta\|_{H^{2}} \big) \Big(\big(\|\eta\|_{W^{2,\infty}}+1\big)\|u\|_{L^2} + \|u\|_{Y^0} \Big),
\end{aligned}
\end{equation}
where, to obtain the last inequality, we used the previous inequality for  $\|\psi\|_{Y^1}$ together with the boundedness of $\mathcal{R}(\eta)$ and $(\Id + \mathcal{G}(\eta))^{-1}$, as stated in Proposition~\ref{prop-reg-dn:Main} and Proposition~\ref{prop-reg-bas:InverBd}, respectively.
\end{proof}

In conclusion, with $\theta(\xi)\geq1/4$ being a smooth, even function that equals $1/4$ for $|\xi|\leq1/2$ and $|\xi|/(1+|\xi|)$ for $|\xi|>1$, the differential system \eqref{n42} admits the following paralinearization:
\begin{equation}\label{Sec4_Main}
	\left\{\begin{array}{l}
		\partial_t \eta - \theta(D_x)u = Q(\eta,u), \\ [0.5ex]
		\partial_t u + (\Dx^4+1)\eta = N(\eta,u), \\ [0.5ex]
		(\eta,u)|_{t=0} = (\eta_0,u_0),
	\end{array}\right.
\end{equation}
where $Q(\eta,u)$ and $N(\eta,u)$ satisfy \eqref{eq-para-sys:ParaEq1Rem} and \eqref{eq-para-sys:EstiF2} (or \eqref{eq-para-sys:EstiF2-tame}), respectively.	

\subsection{Lipschitz continuity of source terms}\label{subsect-para:Lip}

The goal of this subsection is to prove that the source terms $Q,N$ in the right hand side of \eqref{Sec4_Main} are Lipschitz in $(\eta,u)$ in a weaker norm. More precisely, we shall prove the following Lipschitz estimates.

\begin{proposition}\label{prop-para-lip:Main}
    There exists a nondecreasing function $K\colon \R_+\to\R_+$ such that, for all $\eta_1,\eta_2\in H^2\cap W^{2,\infty}$ and for all $u_1,u_2\in L^2\cap Y^0$, there holds
    \begin{equation}\label{eq-para-lip:Main}
        \|Q(\eta_1,u_1) - Q(\eta_2,u_2)\|_{H^1} + \|N(\eta_1,u_1) - N(\eta_2,u_2)\|_{H^{-1}}\le K(A_0) (A_1+1) B_0,
    \end{equation}
    where
    \begin{equation}\label{eq-para-lip:AuxQuan}
        \begin{aligned}
            A_0:&= \|(\eta_1,u_1)\|_{H^2\times L^2} + \|(\eta_2,u_2)\|_{H^2\times L^2}, \\
            A_1:&= \|(\eta_1,u_1)\|_{W^{2,\infty}\times Y^0} + \|(\eta_2,u_2)\|_{W^{2,\infty}\times Y^0}, \\
            B_0:&= \|(\eta_1-\eta_2,u_1-u_2)\|_{H^1\times H^{-1}}.
        \end{aligned}
    \end{equation}
\end{proposition}
In fact, it suffices to check that the derivatives with respect to $\eta,u$ of $Q$ and $N$ belong to the corresponding spaces. The derivative in $u$ is easy to study since $Q$ is linear in $u$ and $N$ is a quadratic form of $u$. The main difficulty is to study the derivative with respect to $\eta$. To do so, we will use the shape derivative formula. Hence, Proposition \ref{prop-para-lip:Main} will be deduced from the following two propositions:

\begin{proposition}\label{prop-para-lip:DerInU}Let $\eta\in H^2\cap W^{2,\infty}$ and $u\in L^2$.

$i)$  The mapping $(\eta,u)\mapsto Q(\eta,u)$ is linear in $u$ and
    \begin{equation}\label{eq-para-lip:QDerInU}
        \|Q(\eta,u)\|_{H^1} \le K\big(\|\eta\|_{H^2}\big)\big(\|\eta\|_{W^{2,\infty}}+1\big)\|u\|_{H^{-1}}.
    \end{equation}
    
$ii)$ $N(\eta,u)$ is a quadratic mapping in $u$ and the associated bilinear mapping $N^\mathrm{bi}(\eta;u_1,u_2)$ satisfies, for all $u_1,u_2$ in $L^2\cap Y^0$,
    \begin{equation}\label{eq-para-lip:NDerInU}
        \|N^\mathrm{bi}(\eta;u_1,u_2)\|_{H^{-1}} \le K\big(\|\eta\|_{H^2}\big) \Big( \big(\|\eta\|_{W^{2,\infty}}+1\big)\|u_1\|_{L^2} + \|u_1\|_{Y^{0}} \Big) \|u_2\|_{H^{-1}}.
    \end{equation}
\end{proposition}

\begin{proposition}\label{prop-para-lip:DerInEta}
    Let $\eta\in H^2\cap W^{2,\infty}$ and $u\in L^2\cap Y^0$. Given $\delta\eta\in H^2\cap W^{2,\infty}$, define the operator $\delta$ as the derivative in $\eta$ along $\delta\eta$. Namely, for any function $Q$ of $(\eta,u)$,
    \begin{equation*}
        \delta Q(\eta,u) := \left.\fraceps\right|_{\varepsilon=0} Q(\eta+\varepsilon\delta\eta,u),
    \end{equation*}
    which is a linear form of $\delta\eta$. Then 
    the derivatives of $Q$ and $N$ satisfy
    \begin{align}
        \|\delta Q(\eta,u)\|_{H^1} &\le K\big(\|\eta\|_{H^2}\big) \Big( \big(\|\eta\|_{W^{2,\infty}}+1\big)\|u\|_{L^2} + \|u\|_{Y^0} \Big) \|\delta\eta\|_{H^1}, \label{eq-para-lip:QDerInEta} \\
        \|\delta N(\eta,u)\|_{H^{-1}} &\le K\big(\|\eta\|_{H^2}\big) \Big( \big(\|\eta\|_{W^{2,\infty}}+1\big)\|u\|_{L^2} + \|u\|_{Y^0} \Big) \|u\|_{L^2} \|\delta\eta\|_{H^1}. \label{eq-para-lip:NDerInEta}
    \end{align}
\end{proposition}

We highlight the non-trivial feature of \eqref{eq-para-lip:QDerInEta}-\eqref{eq-para-lip:NDerInEta}: the right-hand-side relies only on the $H^1$ norm of $\delta\eta$, although we assume $\delta\eta\in H^2\cap W^{2,\infty}$ in order the shape derivative to be well-defined.

\begin{proof}[Proof of \eqref{eq-para-lip:QDerInU}]
    The proof of \eqref{eq-para-lip:QDerInU} is easy. In the formula \eqref{eq-para-sys:Q} of $Q(\eta,u)$,
    \begin{equation*}
		Q(\eta,u) = 
		\big(\Dx(\Id+\Dx)^{-1}-\theta(D_x)\big)u + (\Id+\Dx)^{-1}\mathcal{R}(\eta)(\Id+\mathcal{G}(\eta))^{-1}u,
	\end{equation*}
	the linear operators $(\Id+\Dx)^{-1} \in \mathcal{L}(L^2;H^1)$ and $(\Id+\mathcal{G}(\eta))^{-1} \in \mathcal{L}(H^{-1};L^2)$ due to Proposition \ref{prop-reg-bas:InverBd}, where only $H^2$ norm of $\eta$ is required. In the mean time, Proposition \ref{prop-reg-dn:Main} implies that the linear operator $\mathcal{R}(\eta)$ is bounded on $L^2$, provided that $\eta\in W^{2,\infty}$, which completes the proof of \eqref{eq-para-lip:QDerInU}.
\end{proof}

\begin{proof}[Proof of \eqref{eq-para-lip:NDerInU}]
By Theorem~\ref{thm-reg-shpder:Main}, we immediately obtain the formula
\begin{equation}\label{eq-reg-shpder:LinkShpDerNonlin}
        \begin{aligned}
            \left.\fraceps\right|_{\varepsilon=0} \int_\R \psi_1 \G(\eta+\varepsilon\delta\eta)\psi_2 \dx &= \left.\fraceps\right|_{\varepsilon=0} \int_\R \psi_2 \G(\eta+\varepsilon\delta\eta)\psi_1 \dx \\
            &= 2\int_\R N^\mathrm{bi}(\eta;u_1,u_2)\delta\eta \dx.
        \end{aligned}
\end{equation}
where $\psi_k = (\Id + \G(\eta))^{-1}u_k$, $k=1,2$. The targeted inequality \eqref{eq-para-lip:NDerInU} is equivalent to the assertion that, for arbitrary $\delta\eta\in H^2\cap W^{2,\infty}$,
    \begin{equation*}
        \begin{aligned}
            \Big| \left.\fraceps\right|_{\varepsilon=0} &
            \int_\R \psi_2\G(\eta+\varepsilon\delta\eta)\psi_1 \dx \Big| \\
            &\le K\big(\|\eta\|_{H^2}\big) \Big( \big(\|\eta\|_{W^{2,\infty}}+1\big)\|u_1\|_{L^2} + \|u_1\|_{Y^{0}} \Big)\|u_2\|_{H^{-1}} \|\delta\eta\|_{H^1}.
        \end{aligned}
    \end{equation*}
    Under the condition $\eta\in H^2$, we find that $\|\psi_1\|_{L^2}$ can be controlled by $\|u_1\|_{H^{-1}}$ (see Proposition \ref{prop-reg-bas:InverBd}). Then our aim can be further reduced to the inequality
    \begin{equation*}
        \left\| \left.\fraceps\right|_{\varepsilon=0}\G(\eta+\varepsilon\delta\eta)\psi_1 \right\|_{L^2} \le K\big(\|\eta\|_{H^2}\big) \Big( \big(\|\eta\|_{W^{2,\infty}}+1\big)\|u_1\|_{L^2} + \|u_1\|_{Y^{0}} \Big) \|\delta\eta\|_{H^1},
    \end{equation*}
    which is a consequence of Theorem~\ref{thm-reg-shpder:Main} and Proposition~\ref{prop-reg-shpder:sharp}. Note that we also need Proposition~\ref{prop-reg-bas:InverBd} and Corollary~\ref{cor-reg-dn:RegPsi}: the $H^1$ and $Y^1$ norms of $\psi_1$ can be replaced by the $L^2$ and $Y^0$ norms of $u_1$. This completes the proof of inequality \eqref{eq-para-lip:NDerInU}.
\end{proof}

\begin{proof}[Proof of \eqref{eq-para-lip:QDerInEta}]
    Due to the formula \eqref{eq-para-sys:Q} of $Q(\eta,u)$, its derivative in $\eta$ reads 
    \begin{align*}
        \delta Q(\eta,u) &= \delta \left( \G(\eta)(\Id + \G(\eta))^{-1}u \right) \\
        &= - \delta \left( (\Id + \G(\eta))^{-1}u \right) \\
        &= (\Id + \G(\eta))^{-1} \left.\fraceps\right|_{\varepsilon=0}\G(\eta+\varepsilon\delta\eta) (\Id + \G(\eta))^{-1}u.
    \end{align*}
    Using the shape derivative formula \eqref{eq-reg-shpder:Main}, denoting $\psi = (\Id + \G(\eta))^{-1}u$, we have
    \begin{align*}
        \delta Q(\eta,u) &= - (\Id + \G(\eta))^{-1} \Big( \G(\eta)(\delta\eta\B(\eta)\psi) + \partial_x(\delta\eta\V(\eta)\psi) \Big).
    \end{align*}
    Then the desired result follows from Proposition \ref{prop-reg-bas:InverBd} and \ref{prop-reg-shpder:sharp},
    \begin{align*}
        \|\delta Q(\eta,u)\|_{H^1} &\le K\big(\|\eta\|_{H^2}\big) \left\| \G(\eta)(\delta\eta\B(\eta)\psi) + \partial_x(\delta\eta\V(\eta)\psi) \right\|_{L^2} \\
        &\le K\big(\|\eta\|_{H^2}\big) \left( \|\psi\|_{Y^1} + (\|\eta\|_{W^{2,\infty}}+1) \|\psi\|_{H^1} \right) \|\delta\eta\|_{H^1} \\
        &\le K\big(\|\eta\|_{H^2}\big) \left( \|u\|_{Y^0} + (\|\eta\|_{W^{2,\infty}}+1) \|u\|_{L^2} \right) \|\delta\eta\|_{H^1},
    \end{align*}
    where in the last inequality we also utilize Corollary \ref{cor-reg-dn:RegPsi}.
\end{proof}

\begin{proof}[Proof of \eqref{eq-para-lip:NDerInEta}]
    According to the formula \eqref{n40.5} of $N(\eta,u)$, it can be regarded as a summation of terms in the form
    \begin{equation*}
        F(\eta_x)\left( \eta_x\partial_x\psi + \G(\eta)\psi \right)^2,
    \end{equation*}
    where $F$ is a $C^\infty$ function. In what follows, we shall focus on the study of such terms.
    
    If $\delta$, the derivative in $\eta$, acts on $F(\eta_x)$, one obtains
    \begin{equation*}
        F'(\eta_x)\left( \partial_x\psi + \G(\eta)\psi \right)^2 \partial_x\delta\eta.
    \end{equation*}
    Applying Proposition \ref{prop-besov-bony:Paralin}, $F'(\eta_x)\partial_x\delta\eta$ lies in $L^2$. Moreover, by applying Proposition \ref{prop-reg-bas:InverBd} and \ref{prop-reg-bas:Bd}, we have
    \begin{equation}\label{eq-para-lip:NDerInEta-L2ofGPsi}
        \| \psi_x + \G(\eta)\psi \|_{L^2} \le K\big(\|\eta\|_{H^2}\big) \|u\|_{L^2},
    \end{equation}
    while Corollary \ref{cor-reg-dn:RegPsi} yields
    \begin{align*}
        | \psi_x + \G(\eta)\psi |_{L^\infty} &\le \|\psi\|_{Y^1} + | u - (\Id+\G(\eta))^{-1}u |_{L^\infty} \\
        &\le \|u\|_{Y^0} + \|\psi\|_{Y^1} \\
        &\le K\big(\|\eta\|_{H^2}\big) \big( \|u\|_{Y^0} + (\|\eta\|_{W^{2,\infty}}+1) \|u\|_{L^2} \big).
    \end{align*}
    Combining these estimates, we conclude that
    \begin{equation*}
        \begin{aligned}
            \big\| F'(\eta_x)&( \partial_x\psi + \G(\eta)\psi )^2 \partial_x\delta\eta \big\|_{L^1} \\
            &\le K\big(\|\eta\|_{H^2}\big) 
            \big( \|u\|_{Y^0} + (\|\eta\|_{W^{2,\infty}}+1) \|u\|_{L^2} \big)
            \|u\|_{L^2} \|\delta\eta\|_{H^1}.
        \end{aligned}
    \end{equation*}
    Via Sobolev embedding $L^1 \subset H^{-1}$, the left hand side can be replaced by $H^{-1}$ norm.
    
    When $\delta$ acts on $\psi_x$ or $\G(\eta)\psi$, we claim that
    \begin{equation}\label{eq-para-lip:NDerInEta-Claim}
    \begin{aligned}
    \|\delta(\psi_x)\|_{L^2} & + \|\delta(\G(\eta)\psi)\|_{L^2} \\
    &\le K\big(\|\eta\|_{H^2}\big) \big( \|u\|_{Y^0} + (\|\eta\|_{W^{2,\infty}}+1) \|u\|_{L^2} \big)\|\delta\eta\|_{H^1}.
    \end{aligned}
    \end{equation}
    As long as this claim is proved, the desired estimate will follow from \eqref{eq-para-lip:NDerInEta-L2ofGPsi} and the fact that $F(\eta_x)\in L^\infty$, by writing
    \begin{align*}
        \Big\| F(\eta_x) \delta\Big(\big( \partial_x\psi & + \G(\eta)\psi \big)^2\Big) \Big\|_{H^{-1}} \\
        &\lesssim  \left\| F(\eta_x) \delta\Big(\left( \partial_x\psi + \G(\eta)\psi \right)^2\Big) \right\|_{L^1} \\
        &\lesssim |F(\eta_x)|_{L^\infty} \| \delta(\psi_x) + \delta(\G(\eta)\psi) \|_{L^2} \| \psi_x + \G(\eta)\psi \|_{L^2} \\
        &\le K\big(\|\eta\|_{H^2}\big) \big( \|u\|_{Y^0} + (\|\eta\|_{W^{2,\infty}}+1) \|u\|_{L^2} \big) \|u\|_{L^2} \|\delta\eta\|_{H^1},
    \end{align*}
    which completes the proof of \eqref{eq-para-lip:NDerInEta}.
    
    It remains to check \eqref{eq-para-lip:NDerInEta-Claim}. For the derivative of $\psi_x$, one may apply the shape derivative formula \eqref{eq-reg-shpder:Main}, to get
    \begin{equation*}
    \begin{aligned}
        \delta(\psi_x) 
        & = \partial_x \delta\left( (\Id+\G(\eta))^{-1}u \right) \\
        & = -\partial_x(\Id+\G(\eta))^{-1} \Big( \G(\eta)(\delta\eta\B(\eta)\psi) + \partial_x(\delta\eta\V(\eta)\psi) \Big),
    \end{aligned}
    \end{equation*}
    whose $L^2$ norm can be estimated by using Propositions \ref{prop-reg-bas:InverBd}, \ref{prop-reg-shpder:sharp}, and Corollary \ref{cor-reg-dn:RegPsi},
    \begin{align*}
        \|\delta(\psi_x)\|_{L^2} &\le K\big(\|\eta\|_{H^2}\big) \left\|\mathcal{G}(\eta)(\delta\eta\B(\eta)\psi) + \partial_x(\delta\eta\V(\eta)\psi) \right\|_{L^2} \\
        &\le K\big(\|\eta\|_{H^2}\big) \left( \|\psi\|_{Y^1} + (\|\eta\|_{W^{2,\infty}}+1) \|\psi\|_{H^1} \right)\|\delta\eta\|_{H^1} \\
        &\le K\big(\|\eta\|_{H^2}\big) \left( \|u\|_{Y^0} + (\|\eta\|_{W^{2,\infty}}+1) \|u\|_{L^2} \right)\|\delta\eta\|_{H^1}.
    \end{align*}
    The same argument also works for $\delta(\G(\eta)\psi)$. In fact, one observes that
    \begin{equation*}
    \begin{aligned}
        \delta(\G(\eta)\psi) 
        & = \delta\Big( u- (\Id+\G(\eta))^{-1}u \Big) \\
        & = -(\Id+\G(\eta))^{-1} \Big( \G(\eta)(\delta\eta\B(\eta)\psi) + \partial_x(\delta\eta\V(\eta)\psi) \Big),
    \end{aligned}
    \end{equation*}
    where the right-hand side can be estimated in the same way as above.
\end{proof}

\subsection{Reduction to a Schr{\"o}dinger type equation}\label{subsect-para:Schr}

Based on \eqref{Sec4_Main}, we introduce the dispersive relation
\begin{equation}\label{p(Dx)}
p(\xi)=\theta(\xi)^{1/2}\left(1+|\xi|^4\right)^{1/2},
\end{equation}
and, as explained in the introduction part, we also introduce 
the complex variable $U$ defined by
\begin{equation}\label{Complex_U}
	U:= q(D_x)\eta + i u,\ \ q(\xi)=\theta(\xi)^{-1/2}\left(1+|\xi|^4\right)^{1/2}.
\end{equation}
The equation for $\eta,u$ is therefore equivalent to the following nonlinear dispersive equation:
\begin{equation}\label{Schrodinger}
	\left\{\begin{array}{l}
		\partial_t U + i p(D_x)U = \mathcal{F}(U), \\ [0.5ex]
		U|_{t=0} = U_0 := q(D_x)\eta_0 + i u_0.
	\end{array}\right.
\end{equation}
From \eqref{Sec4_Main}, the source term $\mathcal{F}(U)$ can be expressed as
\begin{equation}\label{eq-para-schr:Sour}
	\mathcal{F}(U) = q(D_x)Q(\eta,u) + iN(\eta,u).
\end{equation}
In what follows, we shall check that the source term $\mathcal{F}(U)$ is bounded in $L^2$, and is Lipschitz in $U$ with respect to the weaker $H^{-1}$ norm. As we have seen in previous sections, this will be possible provided that $U$ has regularity $L^2 \cap Y^0$ in $x$. Moreover, thanks to the high order estimates in Section~\ref{subsect-para:Sys}, when $U$ has higher regularity $H^{s}\cap Y^s$, the source term $\mathcal{F}(U)$ lies in $H^{s}$, while the corresponding estimate is tame.

To begin with, we clarify the regularity of $(\eta, u)$ when $U \in Y^s$. Notably, the real part of $U$ is given by $q(D_x)\eta$, where $q(D_x)$ is a second-order elliptic Fourier multiplier. This suggests that \( U \in Y^s \) implies \( \eta \in W^{s+2,\infty} \), which is proved in the lemma below. As for the imaginary part of \( U \), we immediately have \( u \in Y^s \).

\begin{lemma}\label{prop-reg-var:Main}
Given $s\ge 0$, we assume that $q(D_x)\eta$ belongs to $Y^s$, where the symbol $q(\xi)$ is defined in \eqref{Complex_U}. Then $\eta$ belongs to $W^{s+2,\infty}$ with 
\begin{equation}\label{eq-reg-var:Main}
		\|\eta\|_{W^{s+2,\infty}} \lesssim \|q(D_x)\eta\|_{Y^s}.
\end{equation}
\end{lemma}
\begin{proof}
We notice that $1/q(\xi)$ is smooth on $\{\xi\in\R\}$, while
$$
\frac{1}{q(\xi)}=|\xi|^{-2}+O(|\xi|^{-3}),\quad 
|\xi|\to\infty.
$$
Then, by writing $\eta = q(D_x)^{-1} q(D_x) \eta$, we are able to conclude by Proposition~\ref{prop-pre-fct:BdMult},
$$
\|\eta\|_{W^{s+2,\infty}} \lesssim \|\eta\|_{Y^{s+2}} = \|q(D_x)^{-1} q(D_x) \eta\|_{Y^{s+2}} \lesssim \|q(D_x)\eta\|_{Y^s}.
$$
Note that the first inequality is ensured by the inclusion $Y^{s+2} \subset W^{s+2,\infty}$ (see (4) of Proposition~\ref{prop-pre-fct:Basic}).
\end{proof}

So far, we have seen that the $H^{s+2} \times H^s$ and $Y^{s+2} \times Y^s$ norms of $(\eta, u)$ can be controlled by the $H^s$ and $Y^s$ norm of $U$, respectively. Consequently, we can replace the norms of $(\eta, u)$ with those of the complex variable $U$ in all the estimates obtained in Section~\ref{subsect-para:Sys} and ~\ref{subsect-para:Lip} to obtain self-contained estimates for the source term $\mathcal{F}(U)$.

\begin{proposition}\label{prop-para-schr:SourEsti}
	Given $s\ge 0$, when $U\in H^{s} \cap Y^{s}$, the source term $\mathcal{F}(U)$ defined by \eqref{eq-para-schr:Sour} satisfies
	\begin{equation}\label{eq-para-schr:SourEsti}
		\| \mathcal{F}(U) \|_{L^2} \le K\big(\|U\|_{L^2}\big) \left( \|U\|_{Y^0} + 1 \right),
	\end{equation}
	and
	\begin{equation}\label{eq-para-schr:SourEsti-tame}
		\| \mathcal{F}(U) \|_{H^s} \le K\big(\|U\|_{L^2}\big) \left( \|U\|_{Y^0} + 1 \right) \Big[ (\|U\|_{Y^0}+1) \|U\|_{H^s} + \|U\|_{Y^s} \Big].
	\end{equation}
\end{proposition}
\begin{proof}
    By definition \eqref{eq-para-schr:Sour} of $\mathcal{F}(U)$, its $H^s$-norm is bounded by the $H^{s+2}$-norm of $Q(\eta,u)$ and the $H^s$-norm of $N(\eta,u)$, which have been estimated 
    in Proposition \ref{prop-para-sys:ParaEq1} and \ref{prop-para-sys:ParaEq2}, respectively. The only task is to replace the norms of $(\eta,u)$ by those of $U$. It is clear that the $H^{s+2}$-norm of $\eta$ and the $H^s$-norm of $u$ can be controlled by the $H^{s}$-norm of $U$. Meanwhile, Lemma~\ref{prop-reg-var:Main} ensures that the $W^{s+2,\infty}$-norm of $\eta$ can also be bounded by the $Y^{s}$-norm of $U$, which gives the desired inequality \eqref{eq-para-schr:SourEsti}.
\end{proof}

\begin{proposition}\label{prop-para-schr:SourLip}
	Let $U_1,U_2\in L^2 \cap Y^0$. Then there holds
	\begin{equation}\label{eq-para-schr:SourLip}
	    \begin{aligned}
	        \| \mathcal{F}(U_1) - \mathcal{F}(U_2) \|_{H^{-1}} &\le K\left( \|U_1\|_{L^2} + \|U_2\|_{L^2} \right) \\
	        &\hspace{4em}\times\left( \|U_1\|_{Y^0} + \|U_2\|_{Y^0} + 1 \right) \|U_1 - U_2\|_{H^{-1}}.
	    \end{aligned}
	\end{equation}
\end{proposition}
\begin{proof}
    By definition \eqref{eq-para-schr:Sour} of $F$, the difference $\mathcal{F}(U_1) - \mathcal{F}(U_2)$ equals
    \begin{equation*}
        \mathcal{F}(U_1) - \mathcal{F}(U_2) = q(D_x) \Big( Q(\eta_1,u_1)-Q(\eta_2,u_2) \Big) + i\Big( N(\eta_1,u_1)-N(\eta_2,u_2) \Big),
    \end{equation*}
    hence it follows from Proposition \ref{prop-para-lip:Main} that
    \begin{align*}
        \|\mathcal{F}(U_1) - \mathcal{F}(U_2)\|_{H^{-1}} &\lesssim \|Q(\eta_1,u_1)-Q(\eta_2,u_2)\|_{H^1} + \|N(\eta_1,u_1)-N(\eta_2,u_2)\|_{H^{-1}} \\
        &\le K(A_0)(A_1+1)B_0,
    \end{align*}
    where recall from \eqref{eq-para-lip:AuxQuan} that
    \begin{equation*}
        \begin{aligned}
            A_0:&= \|(\eta_1,u_1)\|_{H^2\times L^2} + \|(\eta_2,u_2)\|_{H^2\times L^2}, \\
            A_1:&= \|(\eta_1,u_1)\|_{W^{2,\infty}\times Y^0} + \|(\eta_2,u_2)\|_{W^{2,\infty}\times Y^0}, \\
            B_0:&= \|(\eta_1-\eta_2,u_1-u_2)\|_{H^1\times H^{-1}}.
        \end{aligned}
    \end{equation*}
    Therefore, it suffices to check that
    \begin{equation*}
        A_0 \le \|U_1\|_{L^2} + \|U_2\|_{L^2},\ \ B_0 \le \|U_1-U_2\|_{H^{-1}},\ \ A_1\le \|U_1\|_{Y^0} + \|U_2\|_{Y^0}.
    \end{equation*}
    The first and second one can be deduced simply from the definition (\ref{Complex_U}) of $U$, and the last one is a consequence of Lemma~\ref{prop-reg-var:Main}.
\end{proof}

\section{Cauchy Problem}\label{sect:Cauchy}

We are now in a position to solve the Cauchy problem~\eqref{n40}. Recall that in the previous section, we first transformed~\eqref{n40} into a nonlinear Schr\"odinger-type equation (see~\eqref{Schrodinger}). Moreover, we have established tame estimates for the nonlinearity \( \mathcal{F}(U) \). In this section, we use this preliminary work to study the following integral equation:
\begin{equation}\label{Strong}
U(t,x) = e^{itp(D_x)}U_0(x) + \int_0^t e^{i(t-\tau)p(D_x)}\mathcal{F}(U(\tau,x))\dtau =: \mathscr{G}(U_0;U).
\end{equation}  
Specifically, we will prove the well-posedness of the Cauchy problem (\ref{n42}) by verifying the following five properties:

\begin{enumerate}[label=\textbf{(\arabic*)}]
\item\label{51} Any solution $U$ to (\ref{Strong}) in the space \( L^\infty_T L^2_x\) with $\text{Re}U\in L^4_TL^{\infty}_x$ is automatically in the smaller space \( L^\infty_T L^2_x \cap L^4_T Y^0_x \). Therefore, it suffices to solve (\ref{Strong}) in the latter smaller space. 

\item\label{52} Solutions $U=q(D_x)\eta+iu\in L^\infty_T L^2_x\cap L^4_TY^0_x$ of (\ref{Strong}) are in one-to-one correspondence to solutions of (\ref{n42}) that satisfy
$$
\eta\in L^\infty_TH^2_x\cap L^4_TW^{2,\infty}_x\cap W^{1,4}_TL^2_x,
\quad
u\in L^\infty_TL^2_x\cap W^{1,4}_TH^{-2}_x.
$$

\item\label{53} Given \( U_0 \in L^2 \), there exists \( T = T(\|U_0\|_{L^2}) > 0 \) such that \( \mathscr{G}(U_0;U) \) satisfies the requirements of the Banach fixed-point theorem in the space \( L^\infty_T L^2_x \cap L^4_T Y^0_x \), equipped with a \emph{weaker metric}.  

\item\label{54} If \( U \) solves (\ref{Strong}) with initial value $U_0\in H^s$ for some \( s>0 \), then necessarily \( U \in L^\infty_T H^s_x \). If in addition $s\geq2$, then the functions \( \eta \) and \( u \) satisfy (\ref{n42}) in the classical sense, and giving rise to a classical solution $(\eta,\psi)$ of (\ref{n40}).

\item\label{55} If \( U^n \) is a sequence of \( L^\infty_T H^s_x \) solutions with initial data \( U_0^n \in H^s \) converging to \( U_0 \) in the \( L^2 \) norm, then \( U^n \to U \) in the \( L^\infty_T H^{-\varepsilon}_x \) norm for any \( \varepsilon > 0 \).  
\end{enumerate}  
The \emph{total energy}  
\begin{equation}\label{eq-intro-csz:Energy}  
	E(t) = \frac{1}{2}  
\int_\R u\cdot\G(\eta)\big(\Id+\G(\eta)\big)^{-1}u \dx  
	+ \frac{1}{2}\int_\R \big((\partial_x^2\eta)^2 + \eta^2\big) \dx.  
\end{equation}  
of the system at a fixed time will play a crucial role in the proof. The important fact is that 
\( E(t) \leq C\|U\|_{L^2_x} \). We will justify that the total energy is non-increasing in time and exploit this fact in the sequel; in particular, it is conserved if the solution has sufficient regularity.  

\begin{remark}  
The proof of energy conservation indeed establishes that the energy serves as the Hamiltonian function for the system, with \( \eta \) and \( u \) being conjugate variables. However, we shall not explore this aspect in the present paper.  
\end{remark}

\subsection{Automatic improved integrability}
In this subsection, we verify Assertion \ref{51}: for $U_0\in L^2_x$, the mapping $\mathscr{G}(U_0;U)$ defined in (\ref{Strong}) necessarily maps a function 
$$
U=q(D_x)\eta+iu\in L^\infty_TL^2_x,
\quad
\text{such that }\eta\in L^4_TW^{2,\infty}_x,
$$
into the smaller space $L^\infty_TL^2_x\cap L^4_TY^0_x$. Therefore, it suffices to solve (\ref{Strong}) in the smaller subspace $L^\infty_TL^2_x\cap L^4_TY^0_x$.

The main difficulty is to estimate the nonlinearity $N=N(\eta,u)$ and its Hilbert transform in terms of the energy. 
Recall that, with $\psi=(\Id+\G(\eta))^{-1}u$,
\begin{equation}\label{eq-intro-csz:N2}
N(\eta,u) = \frac{1}{2}\psi_x^2 -\frac{1}{2}\frac{\left(\eta_x\psi_x + \mathcal{G}(\eta)\psi\right)^2}{1+\eta_x^2}.
\end{equation}
With the total energy $E$ defined as in (\ref{eq-intro-csz:Energy}), we shall prove that, at any time $t$, 
\begin{equation}\label{n90}
\begin{aligned}
&\lA N(\eta(t),u(t))\rA_{L^1_x}\le K(E(t)),
\\ 
&\lA \Hi N(\eta(t),u(t))\rA_{L^1_x+L^2_x}\le K(E(t))\big(\| \eta(t)\|_{W^{2,\infty}_x}+1\big), 
\end{aligned}
\end{equation}  
for some universal, non-decreasing function $K\colon\R_+\to\R_+$. These two inequalities will 
then imply the following improved integrability:
\begin{proposition}\label{Automatic}
Let $T\in(0,1]$ and $U_0\in L^2_x$. Given two functions $\eta\in L^\infty_TH^2_x\cap L^4_TW^{2,\infty}_x$ and 
$u\in L^\infty_TL^2_x$, set 
$$
U=q(D_x)\eta+iu\in L^\infty_TL^2_x.
$$
Then the mapping $\mathscr{G}$ defined in (\ref{Strong}) satisfies $\mathscr{G}(U_0;U)\in L^\infty_TL^2_x\cap L^4_TY^0_x$, with
\begin{equation}\label{n2006}
\big\|\mathscr{G}(U_0;U);L^\infty_TL^2_x\cap L^4_TY^0_x\big\|
\leq C\|U_0\|_{L^2}
+K\big(\|U;L^\infty_TL^2_x\|\big)\big(\|\eta;L^4_TW^{2,\infty}_x\|+1\big).
\end{equation}
\end{proposition}    
\begin{proof}
Let us prove Proposition~\ref{Automatic}, assuming that the claim \eqref{n90} has already been established.

The fact that $U$ belongs to $L^\infty_T L^2_x$ follows immediately from Lemma~\ref{prop-reg-var:Main}. 
Also, the estimate of the $L^\infty_T L^2_x \cap L^4_T Y^0_x$ norm of $e^{itp(D_x)} U_0$ follows directly from the Strichartz inequality (\ref{eq-pre-stri:DisperEsti}). Thus, the proof reduces to estimating  
the Duhamel integral in the definition of $\mathscr{G}$, which we decompose into its real and imaginary parts:  
\begin{equation}\label{Duhamel}
\begin{aligned}
\int_0^t &e^{i(t-\tau)p(D_x)}\mathcal{F}(U(\tau,x))\dtau\\
&=\int_0^t e^{i(t-\tau)p(D_x)}\RE\mathcal{F}(U(\tau,x))\dtau
+i\int_0^t e^{i(t-\tau)p(D_x)}\IM\mathcal{F}(U(\tau,x))\dtau\\
&=:I_1+I_2.
\end{aligned}
\end{equation}
Recall from (\ref{eq-para-schr:Sour}) that  
\(
\mathcal{F}(U) = q(D_x)Q(\eta,u) + iN(\eta,u).
\)
Therefore, we consider 
the real and imaginary parts of $\mathcal{F}(U)$ separately. 
We estimate the real part using (\ref{eq-para-sys:ParaEq1Rem}):  
\[
\begin{aligned}
\|\RE\mathcal{F}(U)\|_{L^2_x}
&\lesssim \|Q(\eta,u)\|_{H^2_x}\\
&\leq K\big( \|\eta\|_{H^2_x} \big) \big( \|\eta\|_{W^{2,\infty}_x} + 1 \big) \|u\|_{L^2_x}\\
&\leq K'\big(\|U;L^\infty_TL^2_x\|\big)\big(\|\eta\|_{W^{2,\infty}_x}+1\big).
\end{aligned}
\]
Hereafter, we extensively use the obvious inequality  
\begin{equation}\label{n2040}
\sup_{t\in [0,T]}E(t)\leq C\|U\|_{L^\infty_TL^2_x}.
\end{equation}
If $\|\eta;L^4_T W^{2,\infty}_x\|$ is finite, then the right-hand side is in $L^4$ in time; therefore, applying the Strichartz inequality (\ref{eq-pre-stri:DisperEsti}) with $(q,r)=(\infty,2)$ ensures the desired bound for $\|I_1;L^\infty_T L^2_x \cap L^4_T Y^0_x\|$.  

As for the imaginary part, the two important facts that we deduce from (\ref{n90}) and the inequality~\eqref{n2040} are the following: firstly, we have  
\[
\|\IM\mathcal{F}(U)\|_{L^\infty_TL^1_x}
=\|N(\eta,u)\|_{L^\infty_TL^1_x}\\
\leq K\big(\|U;L^\infty_TL^2_x\|\big),
\]
and secondly, we can decompose the Hilbert transform $\mathcal{H} \IM \mathcal{F}(U)$ as the sum $F_1 + F_2$ of two terms satisfying
\[
\|F_1(t)\|_{L^1_x}+\|F_2(t)\|_{L^2_x}
\leq K(E(t))\big(\|\eta\|_{W^{2,\infty}_x}+1\big).
\]
Since $F_2$ can be written as $\mathcal{H} \tilde{F}_2$ where $\tilde{F}_2=-\mathcal{H}F_2$ satisfies the same estimate as $\RE \mathcal{F}(U)$, its contribution is analyzed in the same way as above.  
Finally, using \eqref{n2040} again, we see that if $\|\eta;L^4_TW^{2,\infty}_x\|$ is finite, then the function $\|F_1(t)\|_{L^1_x}$ is in $L^4$ in time. Therefore, by combining the two Strichartz inequalities in (\ref{eq-pre-stri:DisperEsti}), we obtain  
the desired bound for $\|I_2;L^\infty_T L^2_x \cap L^4_T Y^0_x\|$.   
This concludes the proof.
\end{proof}

Let us proceed to the proof of (\ref{n90}). As we shall see, the estimate for $N$ follows directly from known results. Regarding the estimate for $\Hi N$, we shall need to exploit certain new identities. We start by reducing the proof of \eqref{n90} to the following statement:

\begin{proposition}\label{P:5.6}
The nonlinearity $N$ and its Hilbert transform can be written as
\begin{align}
N&=\mez\psi_x^2-\mez(1+\eta_x^2)B^2,\label{n70}\\
\mathcal{H}N&=\mez\eta_x\cdot(V^2-B^2)+VB+R,\label{n71}
\end{align}
where as in Section \ref{sect:Reg},
$$
B=\frac{\mathcal{G}(\eta)\psi+\eta_x\psi_x}{1+\eta_x^2},
\quad V=\psi_x-B\eta_x,
$$
and the remainder $R$ satisfies
\begin{equation}\label{n72}
\lA R\rA_{L^2_x}\le K(E)\big(\|\eta\|_{W^{2,\infty}_x}+1\big).
\end{equation}
\end{proposition}

Before proving this proposition, let us explain how to deduce \eqref{n90} from it. To do so, 
we begin by gathering some elementary inequalities. Using the Sobolev embedding $H^2\subset W^{1,\infty}$ and Proposition~\ref{prop-reg-bas:BdBV}, we have
$$
\begin{aligned}
|\eta_x|_{L^\infty_x}&\lesssim \lA \eta\rA_{H^2_x},\\
\quad 
\lA V\rA_{L^2_x}+\lA B\rA_{L^2_x}
&\le K(\lA \eta\rA_{H^2_x})\big(\lA \psi_x\rA_{L^2_x}+\lA G(\eta)\psi\rA_{L^2_x}\big).
\end{aligned}
$$
On the other hand, the Rellich inequality~\eqref{eq-reg-bas:Rellich} implies that
$$
\lA\psi_x\rA_{L^2_x}^2\le K(|\eta_x|_{L^\infty_x})
+\int_{\R} (\mathcal{G}(\eta)\psi)^2\dx.
$$
We infer that
$$
|\eta_x|_{L^\infty_x}+ \lA V\rA_{L^2_x}+\lA B\rA_{L^2_x}\le K(E),
$$
and it immediately follows from the Cauchy-Schwarz inequality that
\begin{equation}\label{n91}
\lA N\rA_{L^1_x}+\lA \mez\eta_x(V^2-B^2)+VB\rA_{L^1_x}\le K(E).
\end{equation}
Finally, the desired inequality for $\mathcal{H}N$ in~\eqref{n90} follows from the estimate~\eqref{n72} for the remainder $R$.

We now proceed with the proof of Proposition~\ref{P:5.6}.
\begin{proof}[Proof of Proposition~\ref{P:5.6}]
The identity \eqref{n70} follows directly from the definitions of $N$ and $B$. 
The only difficulty is to prove \eqref{n71}. For that, we begin 
with the following result whose proof is inspired by identities introduced by Lokharu~(\cite{LokharuJFM2021}) to study Stokes waves.

\begin{lemma}
Introduce the functions
$$
g(x)=\int_{-\infty}^{\eta(x)}\Psi_x(x,y_1)\Psi_y(x,y_1)\dy_1,\quad 
w(x)=\int_{-\infty}^{\eta(x)}\frac{1}{2}\big(\Psi_x^2(x,y_1)-\Psi_y^2(x,y_1)\big)\dy_1.
$$
Recall that $\Psi$ is the harmonic extension of $\psi$ (see \eqref{eq-intro-csz:HarExt} and \eqref{eq-intro-csz:psi}).
Then, $N$ and its Hilbert transform are given by 
$$
N=\partial_xg,\quad 
\Hi N=\partial_xw-\mathcal{R}(\eta)g,
$$
where $\mathcal{R}(\eta)=\mathcal{G}(\eta)-\Dx$.
\end{lemma}
\begin{proof}
The relation $N=\partial_xg$ follows directly from the chain rule and is rigorously justified in \cite{AKT2024}, so we focus only on the identity for $\mathcal{H}N$. Introduce two functions
\begin{align*}
G(x,y)&=\int_{-\infty}^y\Psi_x(x,y_1)\Psi_y(x,y_1)\dy_1,\\
W(x,y)&=\int_{-\infty}^y\frac{1}{2}\big(\Psi_x^2(x,y_1)-\Psi_y^2(x,y_1)\big)\dy_1,
\end{align*}
so that $g=G\arrowvert_{y=\eta}$ and $w=W\arrowvert_{y=\eta}$. 

Let us introduce the complex variable $z=x+iy$ and the \emph{complex potential} in the fluid domain, namely, 
the holomorphic function $F(z)=\Psi(x,y)+i\Theta(x,y)$, 
where $\Theta(x,y)$ is the conjugate harmonic function 
of $\Psi(x,y)$, that tends to zero when $y\to-\infty$. 
By the Cauchy-Riemann equations, we know that $\overline{F'(z)}$ exactly 
represents the velocity field $\nabla_{x,y}\Psi(x,y)$. 
Note that 
$$
F'(z)^2=\Psi_x^2(x,y)-\Psi_y^2(x,y)-2i\Psi_x(x,y)\Psi_y(x,y),
$$
so the complex function $G(x,y)+iW(x,y)$ is nothing but the primitive of $F'(z)^2/2$ with limit value zero when $\text{Im}z\to-\infty$. Therefore, $G,W$ are conjugate harmonic functions in the fluid domain and satisfy the Cauchy-Riemann equations	
$$
\partial_xG=\partial_yW,\quad \partial_xW=-\partial_yG.
$$
Directly from the definition of the Dirichlet-to-Neumann operator, we obtain
$$
\G(\eta)g=(\partial_yG-\eta_x\partial_x G)\arrowvert_{y=\eta}=
(\partial_xW+\eta_x\partial_y W)\arrowvert_{y=\eta}
=\partial_x(W\arrowvert_{y=\eta})=\partial_xw.
$$
Therefore, we simply decompose $\mathcal{G}(\eta)g=\Dx g+\mathcal{R}(\eta)g$ and then write $\Dx=\mathcal{H}\partial_x$ to obtain the wanted identity 
$\mathcal{H}N=\partial_xw-\mathcal{R}(\eta)g$.
\end{proof}

Now, observe that $g$ belongs to $L^1_x$ since
\begin{align*}
\int_\R \la g(x)\ra\dx&\le \iint_\Omega \la \Psi_x\Psi_y\ra\dy\dx\le 
\mez \iint_\Omega \la \nabla_{x,y}\Psi\ra^2\dydx\\
&\le \mez \int_\R\psi \mathcal{G}(\eta)\psi\dx\\
&\le E.
\end{align*}
For our purpose, we need in fact to estimate the $L^2$-norm of $g$. To do so, we start by exploiting the identity $N=\partial_xg$ and the estimate~\eqref{n91}
to get that $\lA \partial_xq\rA_{L^1_x}\le K(E)$. Then, by combining this with the previous estimate for the $L^1$-norm of $g$, we conclude that $\lA g\rA_{W^{1,1}_x}\le K(E)$. 
Since $W^{1,1}_x$ is continuously embedded in $L^2_x$, it follows that $\lA g\rA_{L^2_x}\le K(E)$.
We are now in position to use the paralinearization result for the Dirichlet-to-Neumann operator, namely the operator norm 
estimate
$$
\lA \mathcal{R}(\eta)\rA_{L^2_x\to L^2_x}\le K(\lA \eta\rA_{H^2_x})(\lA \eta\rA_{W^{2,\infty}_x}+1),
$$ 
to conclude that $R=-\mathcal{R}(\eta)g$ satisfies the wanted estimate \eqref{n72}. 
This completes the proof of Proposition~\ref{P:5.6}.
\end{proof}

\subsection{Meaning of solution}
We then justify Assertion \ref{52} and give a precise meaning to the solution of the Cauchy problem (\ref{n42}). 

Suppose $U=q(D_x)\eta+iu\in L^\infty_T L^2_x\cap L^4_TY^0_x$ solves (\ref{Strong}). By Proposition \ref{prop-para-schr:SourEsti}, we find $\mathcal{F}(U)\in L^4_TL^2_x$. Using the fundamental properties of the one-parameter unitary group $e^{itp(D_x)}$ on Sobolev spaces, we easily find that $U$ solves the Schr\"{o}dinger equation (\ref{Schrodinger}), in the sense that $U$ belongs to $W^{1,4}_TH^{-2}_x$, with its time derivative understood in both the distributional and almost everywhere sense:
$$
\partial_tU=-ip(D_x)U+\mathcal{F}(U)\in L^\infty_TH^{-2}_x+L^4_TL^2_x
\subset L^4_TH^{-2}_x.
$$
By separating the real and imaginary parts, we find that the functions satisfy $\eta\in L^\infty_TH^2_x\cap L^4_TY^{2}_x$ and $u\in L^\infty_TL^2_x$, while their (distributional and almost everywhere) time derivatives are given by
\begin{equation}\label{Dt_Main_Var}
\begin{aligned}
\partial_t\eta&=\G(\eta)(\Id+\G(\eta))^{-1}u\in L^\infty_TL^2_x, \\
\partial_tu&=-\partial_x^4\eta-\eta-N(\eta,u)
\in L^\infty_TH^{-2}_x+L^\infty_TH^2_x+L^4_TL^2_x.
\end{aligned}
\end{equation}
This then justifies that $\eta\in W^{1,4}_TL^2_x$ and $u\in W^{1,4}_TH^{-2}_x$.

Conversely, suppose that the functions
$$
\eta\in L^\infty_TH^2_x\cap L^4_TY^{2}_x\cap W^{1,4}_TL^2_x,
\quad
u\in L^\infty_TL^2_x\cap W^{1,4}_TH^{-2}_x
$$
satisfy equation (\ref{n42}), meaning that their distributional 
time derivatives satisfy (\ref{Dt_Main_Var}). The complex-valued function $U=q(D_x)\eta+iu$ belongs to the space $L^\infty_TL^2_x\cap W^{1,4}_TH^{-2}_x$, with $\eta\in L^4_TY^{2}_x$, and satisfies the dispersive equation. On the other hand, by Proposition \ref{prop-para-schr:SourEsti}, we find that $\mathcal{F}(U)\in L^4_TL^2_x$. Thus, $U$ actually satisfies the integral equation (\ref{Strong}). By Assertion \ref{51}, we conclude that $U$ actually belongs to the smaller space $L^\infty_TL^2_x\cap L^4_TY^0_x$.

\subsection{Local well-posedness}\label{subsect-cauchy:LWP}

In this subsection, we shall justify Assertion \ref{53}. We first prove uniqueness of solution of (\ref{Strong}) in the space $L^\infty_TL^2_x \cap L^4_TY^0_x$, and then construct a ball in the energy space, equipped with a weaker norm, on which we can apply a Banach fixed point argument to obtain a solution. 

\begin{proposition}\label{prop-cauchy-lwp:Unique}
    Consider $T\in (0,1]$,  $U_1,U_2\in L^\infty_TL^2_x \cap L^4_TY^0_x$, and $U_{10},U_{20}\in L^2_x$. Set
    \begin{equation*}
        M := \max_{k=1,2} \big\|U_k;L^\infty_T L^2_x \cap L^4_TY^0_x\big\|.
    \end{equation*}
    Let $C$ be the constant that appeared in the Strichartz inequality (\ref{eq-pre-stri:DisperEsti}). There exists a non-decreasing function $K\colon\R_+ \rightarrow\R_+$, such that
    \begin{equation}\label{eq-cauchy-lwp:Contraction}
        \begin{aligned}
        &\big\|\mathscr{G}(U_{10};U_1)-\mathscr{G}(U_{20};U_2);L^\infty_{T}H^{-1}_x\cap L^4_{T}Y^{-1}_x \big\| \\
            &\hspace{8em}\le C \|U_{10}-U_{20}\|_{H^{-1}_x} + T^{3/4}K(M) \|U_1 - U_2\|_{L^\infty_T H^{-1}_x}.
        \end{aligned}
    \end{equation}
    Consequently, there exists a time $T' = T'(M) >0 $, which is a bounded function of $M$, such that the following estimate holds: if 
    $U_1,U_2\in L^\infty_TL^2_x \cap L^4_TY^0_x$  are two solutions of \eqref{Strong} with initial data $U_{10}$ and $U_{20}$ respectively, then 
    \begin{equation}\label{eq-cauchy-lwp:LipIniData}
        \big\|U_1-U_2;L^\infty_{T'}H^{-1}_x\cap L^4_{T'}Y^{-1}_x\big\|
        \lesssim \|U_{10}-U_{20}\|_{H^{-1}_x}.
    \end{equation}
    In particular, if $U_{10}=U_{20}$, then the solutions $U_1,U_2$ coincide on $[0,T'(M)]$.
\end{proposition}
\begin{proof}
    Due to definition \eqref{Strong} of $\Phi$, we have 
    \begin{equation*}
    \begin{aligned}
    \mathscr{G}(U_{10};U_1) & - \mathscr{G}(U_{20};U_2) \\
    & = e^{itp(D_x)}(U_{10}-U_{20}) + \int_0^te^{i(t-\tau)p(D_x)} \left( \mathcal{F}(U_1(\tau)) - \mathcal{F}(U_2(\tau)) \right) d\tau.
    \end{aligned}
    \end{equation*}
    Then, using the Strichartz estimate (\ref{eq-pre-stri:DisperEsti}), noticing $e^{itp(D_x)}$ commutes with any Fourier multiplier, 
    we deduce from Proposition~\ref{prop-para-schr:SourLip} that
    $$
    \begin{aligned}
        &\big\|\mathscr{G}(U_{10};U_1) - \mathscr{G}(U_{20};U_2);L^\infty_{T'}H^{-1}_x\cap L^4_{T'}Y^{-1}_x\big\|\\
        &\quad\le C \|U_{10}-U_{20}\|_{H^{-1}_x} + \int_0^T \big\|\mathcal{F}(U_1)-\mathcal{F}(U_2)\big\|_{H^{-1}_x}\dtau\\
        &\quad\le C \|U_{10}-U_{20}\|_{H^{-1}_x} \\
        &\hspace{2em}+ \int_0^T K\big(\|U_1\|_{L^2_x}+\|U_2\|_{L^2_x}\big) \big(\|U_1\|_{Y^0_x}+\|U_2\|_{Y^0_x}+1\big)\|U_1-U_2\|_{H^{-1}_x}\dtau\\
        &\quad\le C \|U_{10}-U_{20}\|_{H^{-1}_x} + K(M)\|U_1-U_2\|_{L^\infty_TH^{-1}_x} \int_0^T\big(1+\|U_1\|_{Y^0_x}+\|U_2\|_{Y^0_x}\big)\dtau\\
        &\quad\le C \|U_{10}-U_{20}\|_{H^{-1}_x} +  T^{3/4}K(M)\|U_1-U_2\|_{L^\infty_TH^{-1}_x},
    \end{aligned}
    $$
    which proves \eqref{eq-cauchy-lwp:Contraction}. In the last inequality, we used H\"{o}lder's inequality to estimate e.g. 
    $$
    \|U_1;L^1_TY^0_x\|\leq T^{3/4}\|U_1;L^4_TY^0_x\|.
    $$
    To deduce \eqref{eq-cauchy-lwp:LipIniData}, it suffices to choose any $T'>0$ such that $(T')^{3/4}K(M) < 1/2$.
\end{proof}

Now, we are going to find the fixed point of $\mathscr{G}(U_0;U)$ in the space $L^\infty_TL^2_x\cap L^4_TY^0_x$, for some time $T\in(0,1]$  depending on the initial value $U_0\in L^2_x$, to be determined.

In the following, we write
$$
\bar{\mathscr{B}}_R^T
:=\left\{\big\|U; L^\infty_TL^2_x\cap L^4_TY^0_x\big\|\leq R\right\}
$$
for the closed ball of radius $R$ in $L^\infty_TL^2_x\cap L^4_TY^0_x$. We first justify that $\bar{\mathscr{B}}_R^T$ is in fact a complete metric space under a \emph{weaker} norm.

\begin{proposition}\label{complete}
Let $R>0$. Equip $\bar{\mathscr{B}}_R^T$ with the weaker distance
$$
\rho(U_1,U_2):=\big\|U_1-U_2; L^\infty_TH^{-1}_x\cap L^4_TY^{-1}_x\big\|.
$$
Then $(\bar{\mathscr{B}}_R^T,\rho)$ is a complete metric space.
\end{proposition}
\begin{proof}
First, we observe that $L^\infty_TL^2_x$ is the dual space of $L^1_TL^2_x$, which is a separable Banach space. By the separable version of the Banach–Alaoglu theorem (which avoids the use of the axiom of choice), it follows that any closed ball in $L^\infty_TL^2_x$ is necessarily sequentially compact under the weak-$\ast$ topology. 

Now, suppose that $\{U_k\}\subset\bar{\mathscr{B}}_R^T$ is a Cauchy sequence with respect to the metric $\rho$, with $\rho$-limit $U\in L^\infty_TH^{-1}_x$. Then, by definition of $\bar{\mathscr{B}}_R^T$, the sequence has a bounded $L^\infty_TL^2_x$ norm, and hence, by weak-$\ast$ sequential compactness, we can extract a weak-$\ast$ convergent subsequence from it, which we may assume to be $\{U_k\}$ itself. Since we already have $\|U_k-U\|_{L^\infty_TH^{-1}_x}\to0$, we conclude that $U$ belongs to $L^\infty_TL^2_x$. Thus, for all $h\in L^1_TL^2_x$, it follows that  
$$
\int_0^T\int_{\R}U(t,x)h(t,x)\dx\dt
=\lim_{k\to+\infty}\int_0^T\int_{\R}U_k(t,x)h(t,x)\dx\dt.
$$

It remains only to show that the $L^4_TY^0_x$ norm $U$ is bounded by $R$. Let us fix test functions $\varphi(t)\in C^\infty_0[0,T]$ and $f_1(x),f_2(x)\in C_{0}^\infty(\R)$, such that $\|\varphi\|_{L^{4/3}_T}=1$, $\|f_1\|_{L^1_x}=\|f_2\|_{L^1_x}=1$, and compute
\begin{align}\label{dual}
\left|\int_{0}^T\varphi\int_{\R_x}
\big(Uf_1+(\Hi U)f_2\big)\dx\dt\right|
&=\left|\lim_k\int_{0}^T\varphi\int_{\R_x}
\big(U_kf_1+(\Hi U_k)f_2\big)\dx\dt\right|\notag\\
&\leq \limsup_k \left|\int_{0}^T\varphi\int_{\R_x}
\big(U_kf_1+(\Hi U_k)f_2\big)\dx\dt\right| \\
&\leq \left(\int_{0}^T
\left|\int_{\R_x}\big(U_kf_1+(\Hi U_k)f_2\big)\dx\right|^4\dt\right)^{1/4}.\notag
\end{align}
Note that taking the limit in the first step is justified for $\varphi,f\in C_0^\infty$, since $U_k$ already converges to $U$ in the weak-$\ast$ topology of $L^\infty_TL^2_x$. We continue transforming the right-hand side of (\ref{dual}): observing that 
$$
\begin{aligned}
\left|\int_{\R_x}\big(U_kf_1+(\Hi U_k)f_2\big)\dx\right|
&\leq |U_k|_{L^\infty_x}+|\Hi U_k|_{L^\infty_x}
=\|U_k\|_{Y^0_x},
\end{aligned}
$$
we deduce that (\ref{dual}) becomes
$$
\left|\int_{0}^T\varphi\int_{\R_x}
\big(Uf_1+(\Hi U)f_2\big)\dx\dt\right|
\leq \left(\int_{0}^T\|U_k\|_{Y^0_x}^4\dt\right)^{1/4}
\leq R.
$$
This proves that $\|U;L^4_TY^0_x\|\leq R$.
\end{proof}

We are ready to show that $\mathscr{G}(U_0;U)$ is a contraction on the complete metric space $(\bar{\mathscr{B}}_R^T,\rho)$, provided that $T$ is suitably small.

\begin{proposition}\label{Self}
Given $U_0\in L^2_x$, set $R=2C\|U_0\|_{L^2_x}$, where $C$ is the constant that appeared in the Strichartz inequality (\ref{eq-pre-stri:DisperEsti}). There exists $T=T(R)\in(0,1]$ such that $\mathscr{G}(U_0;U)$ maps the closed ball $\bar{\mathscr{B}}_R^T$ into itself.
\end{proposition}
\begin{proof}
This is a simple consequence of the Strichartz estimate for $e^{itp(D_x)}$ established in Proposition \ref{prop-pre-stri:DisperEsti}. By (\ref{eq-pre-stri:DisperEsti}), we simply compute
$$
\begin{aligned}
\big\|\mathscr{G}(U_0;U); L^\infty_TL^2_x\cap L^4_TY^0_x\big\|
&\leq C\|U_0\|_{L^2_x}+C\|\mathcal{F}(U)\|_{L^1_TL^2_x}\\
&\overset{\text{Prop.} \ref{prop-para-schr:SourEsti}}{\leq}
\frac{R}{2}+C\int_0^T K\big(\|U\|_{L^2_x}\big) \big( \|U\|_{Y^0_x} + 1 \big) \dtau\\
&\leq \frac{R}{2}+CK(R)\int_0^T \big(1+\|U\|_{Y^0_x}\big)\dtau\\
&\leq \frac{R}{2}+CT^{3/4}K(R)(1+R).
\end{aligned}
$$
In the last inequality we used H\"{o}lder's inequality to estimate $\int_0^T\|U\|_{Y^0_x}\dtau$. Then it is straightforward to choose $T=T(R)$ such that the right-hand side is striclty smaller than $R$.
\end{proof}

\begin{proposition}\label{Contraction}
Let $R$ be as in Proposition \ref{Self}. There is a $T=T(R)\in(0,1]$, such that $\mathscr{G}(U_0;U)$ has a unique fixed point $\mathcal{S}(U_0)\in\bar{\mathscr{B}}_R^T$, satisfying the estimate
$$
\big\|\mathcal{S}(U_0)-\mathcal{S}(\tilde{U}_0);
L^\infty_TH^{-1}_x\cap L^4_TY^{-1}_x\big\|
\leq \|U_0-\tilde{U}_0\|_{H^{-1}_x}.
$$
\end{proposition}
\begin{proof}
Let us first show that $\mathscr{G}(U_0;U)$ can be made into a contraction when $T$ is sufficiently small. We simply examine the $H^{-1}_x$ norm of the difference $\mathscr{G}(U_0;U_1)-\mathscr{G}(U_0;U_2)$: we compute, for $U_1,U_2\in\bar{\mathscr{B}}_R^T$, 
$$
\begin{aligned}
\big\|\mathscr{G}(U_0;U_1)-\mathscr{G}(U_0;U_2);L^\infty_TH^{-1}_x\cap L^4_TY^{-1}_x\big\|
\overset{(\ref{eq-cauchy-lwp:Contraction})}{\leq} T^{3/4}K(R)\|U_1-U_2\|_{L^\infty_TH^{-1}_x}.
\end{aligned}
$$
Therefore, we can choose $T=T(R)$ even smaller to make the constant multiple in the right-hand-side to be $1/2$. This shows that $U\to\mathscr{G}(U_0;U)$ is a contraction on the metric space $(\bar{\mathscr{B}}_R^T,\rho)$, which is complete by Lemma \ref{complete}. By the Banach fixed point theorem, $U\to\mathscr{G}(U_0;U)$ has a unique fixed point $U=\mathcal{S}(U_0)$. 

Finally, to obtain the Lipschitz dependence of $\mathcal{S}(U_0)$ on $U_0$, we just need to show that the mapping $\mathscr{G}(U_0;U)$ has Lipschitz dependence on $U_0$. This again follows from (\ref{eq-cauchy-lwp:Contraction}).
\end{proof}

We summarize the results in this subsection as the following local well-posedness theorem: 
\begin{theorem}\label{thm-intro-main:LWP}
For all $U_0\in L^2$, there exists $T>0$ depending only on the $L^2$ norm of $U_0$, such that the integral equation \eqref{Strong} has a unique solution 
$$
U=q(D_x)\eta+iu\in L^\infty_TL^2_x\cap L^4_TY^0_x.
$$
Moreover, we have the following blow-up criterion:
\begin{equation}\label{eq-intro-main:BlowUp}
\text{either } T^* = +\infty,\text{ or }\lim_{t\rightarrow T^*} \|U(t)\|_{L^2_x} = +\infty,
\end{equation}
where $T^*$ is the maximal existing time.
\end{theorem}

\subsection{A posteriori regularity}\label{High_Reg}
In this subsection, we prove Assertion \ref{54}: if the initial value $U_0\in H^s$ with $s>0$, then the strong solution $U=\mathcal{S}(U_0)\in L^\infty_TL^2_x\cap L^4_TY^0_x$ in Theorem \ref{thm-intro-main:LWP} in fact belongs to $L^\infty_TH^s_x$. Therefore, if $s\geq2$, such a solution $U$ actually solves the dispersive equation (\ref{Schrodinger}), hence gives rise to a solution $(\eta,\psi)$ of the original system (\ref{n40}) in the classical sense. 

\begin{proposition}\label{U_High_reg}
Suppose $s>0$, and the initial value $U_0\in H^s$. Then the solution $\mathcal{S}(U_0)$ of the integral equation (\ref{Strong}) is automatically of class $L^\infty_TH^s_x$.
\end{proposition}
\begin{proof}
The proof is much easier since we have already obtained the solution $U$ and are now working with functions of much higher regularity in $x$.

Let us recall \eqref{eq-para-schr:SourEsti-tame}: given $U\in H^s_x\cap Y^s_x$ with $s\geq0$, there holds
$$
\| \mathcal{F}(U) \|_{H^s_x} 
\le K\big(\|U\|_{L^2_x}\big) \big( \|U\|_{Y^0_x} + 1 \big) \Big[ \big(\|U\|_{Y^0_x}+1\big) \|U\|_{H^s_x} + \|U\|_{Y^s_x} \Big].
$$
Combining with the Strichartz inequality (\ref{eq-pre-stri:DisperEsti}), we compute, for any $U\in L^\infty_TH^s_x\cap L^4_TY^s_x$ with $T\in(0,1]$,
\begin{equation}\label{G_High}
\begin{aligned}
\|\mathscr{G}(U_0;U);&L^\infty_TH^s_x\cap L^4_TY^s_x\|\\
&\leq \|U_0\|_{H^s}+C_s\int_0^T\|\mathcal{F}(U(\tau))\|_{H^s_x}\dtau\\
&\leq \|U_0\|_{H^s}
+K\big(\|U\|_{L^\infty_TL^2_x}\big)
\|U\|_{L^\infty_TH^s_x}\int_0^T\big(1+\|U(\tau)\|_{Y^0_x}
\big)^2\dtau\\
&\quad +K\big(\|U\|_{L^\infty_TL^2_x}\big)
\int_0^T\big(1+\|U(\tau)\|_{Y^0_x}\big)\|U(\tau)\|_{Y^s_x}\dtau\\
&\leq \|U_0\|_{H^s}
+T^{1/2}K\big(\|U\|_{L^\infty_TL^2_x}\big)
\big(1+\|U\|_{L^4_TY^0_x}\big)^2\|U\|_{L^\infty_TH^s_x}\\
&\quad 
+T^{1/2}K\big(\|U\|_{L^\infty_TL^2_x}\big)
\big(1+\|U\|_{L^4_TY^0_x}\big)\|U\|_{L^4_TY^s_x}.
\end{aligned}
\end{equation}
This estimate is tame in the sense that the right-hand-side is linear in the high norm $\|U;L^\infty_TH^s_x\cap L^4_TY^s_x\|$.

Let us then recall that the solution $U=\mathcal{S}(U_0)$ of (\ref{Strong}) is obtained by Banach fixed point argument. Therefore, it is the $L^\infty_TH^{-1}_x\cap L^4_TY^{-1}_x$-limit of the iterative sequence $U_{n+1}=\mathscr{G}(U_0;U_n)$, $U_1=\mathscr{G}(U_0;0)$. Note that we have already established the local well-posedness in $L^2_x\cap Y^0_x$, which implies a uniform bound:
$$
\|U_n;L^\infty_TL^2_x\cap L^4_TY^0_x\|
\leq K\big(\|U_0\|_{L^2}\big).
$$
Substituting into the inequality (\ref{G_High}), we find
$$
\begin{aligned}
\|U_{n+1};L^\infty_TH^s_x\cap L^4_TY^s_x\|
&\leq \|U_0\|_{H^s}
+T^{1/2}K\big(\|U_0\|_{L^2}\big)
\|U_n;L^\infty_TH^s_x\cap L^4_TY^s_x\|.
\end{aligned}
$$
Therefore, if we require further $K\big(\|U_0\|_{L^2_x}\big)T^{1/2}\leq1/2$, then inductively we have 
$$
\|U_n;L^\infty_TH^s_x\cap L^4_TY^s_x\|\leq 2\|U_0\|_{H^s},
$$
so the limit $U$ of $U_n$ must still satisfy the same estimate in its $H^s$ norm. Note that this choice of $T$ depends only on the low norm $\|U_0\|_{H^2}$ of the initial data, and \emph{not} on any high norm; consequently, a continuous induction on the time interval justifies $U\in L^\infty_TH^s_x$. This continuous induction resembles the proof of a posteriori regularity for semilinear Schr\"{o}dinger equation, as in Chapter 5 of \cite{CW1990}. 
\end{proof}

Now suppose $s\geq2$, and $U\in L^\infty_TH^s_x$ solves the integral equation (\ref{Strong}). Standard results on one-parameter unitary groups imply that $U$ solves the Schr\"{o}dinger type equation
$$
\partial_t U + i p(D_x)U = \mathcal{F}(U)
$$
in the usual sense: the function $U$ is of class $L^\infty_TH^{s}_x\cap C^1_TH^{s-2}_x$, and $(\partial_t+ip(D_x))U$ in fact equals to $\mathcal{F}(U)$, which is of class $L^\infty_TH^s_x$ by the $H^s$ estimate we just used. This then implies that $(\eta,u)\in L^\infty_TH^{s+2}_x\times L^\infty H^{s}_x$ solves (\ref{n40}) in the classical sense:
\begin{equation}\label{Regularity}
\begin{aligned}
\partial_t \eta & = \mathcal{G}(\eta)(\Id+\G(\eta))^{-1}u \in L^\infty_TH^{s}_x, \\ 
\partial_t u & = - \eta - \partial_x^4 \eta- N(\eta,u) 
\in L^\infty_TH^{s-2}_x.
\end{aligned}
\end{equation}
Therefore, the second time derivative 
$$
\begin{aligned}
\partial_t^2\eta
&=(\Id+\G(\eta))^{-1}\partial_t(\G(\eta))(\Id+\G(\eta))^{-1}u
+\mathcal{G}(\eta)(\Id+\G(\eta))^{-1}\partial_tu\\
&\in L^\infty_TH^{s+1}_x+L^\infty_TH^{s-2}_x
\subset L^\infty_TH^{s-2}_x.
\end{aligned}
$$
Here we used the shape derivative formula (\ref{eq-reg-shpder:Main}). This shows that the functions $\eta$ and $\psi=(\Id+\G(\eta))^{-1}u$ does solve (\ref{n40}) in the classical sense, and in fact
$$
\eta\in L^\infty_TH^{s+2}_x\cap C^1_TH^{s}_x\cap C^2_TH^{s-2}_x,\quad
\psi\in L^\infty_TH^{s+1}_x\cap C^1_{T}H^{s-2}_x.
$$

\subsection{Global well-posedness}\label{subsect-cauchy:GWP}

In this final subsection, we complete the proof of Theorem \ref{thm:main}. This relies on the following \emph{a priori estimate} of the solution $(\eta,u)$, within the proof of which we verify Assertion \ref{55}:

\begin{proposition}\label{prop-cauchy-gwp:Main}
	Consider the unique solution $(\eta,u) \in L^\infty_{T^*}(H^2_x \times L^2_x)$ obtained in Theorem \ref{thm-intro-main:LWP}, where $T^*\in(0,+\infty]$ is the maximal existing time. Then there exists a constant $C>0$ and an increasing function $K$ such that, for all time $t\in[0,T^*)$,
	\begin{equation}\label{eq-cauchy-gwp:Main}
	   \|u(t)\|_{L^2_x} + \|\eta(t)\|_{H^2_x} \le C\|u_0\|_{L^2} + K(E_0)(1+t),
	\end{equation}
	where $E_0$ is the total energy of the initial data $(\eta_0,u_0)\in H^2\times L^2$, defined in \eqref{eq-intro-csz:Energy}.
\end{proposition}

\begin{proof}
Let us first establish that if $s\ge 2$ and if $(\eta_0,u_0)$ belongs to $H^{s+2}\times H^{s}$ i.e.\ $U_0\in H^s$, then the total energy $E(t)$  is conserved. In fact, we have just proved that if $U_0\in H^s$, then the solution $U=q(D_x)\eta+iu$ of (\ref{Strong}) satisfies (\ref{Regularity}). Since $s\geq2$, we can 
differentiate the energy $E(t)$, to obtain:
$$
\begin{aligned}
E'(t)&=\frac{1}{2}\frac{\diff}{\dt}\int_\R u\cdot\G(\eta)(\Id+\G(\eta))^{-1}u \dx
	+ \frac{1}{2}\frac{\diff}{\dt}\int_\R \big((\partial_x^2\eta)^2 + \eta^2\big)\dx\\
&=\int_\R \partial_tu\cdot\G(\eta)(\Id+\G(\eta))^{-1}u\dx
+\frac{1}{2}\int_{\R}u\cdot\partial_t\big(\G(\eta)(\Id+\G(\eta))^{-1}\big)u\dx\\
&\quad+\int_\R \partial_t\eta\cdot\big(\partial_x^4\eta+\eta\big)\dx\\
&=\int_\R \partial_tu\cdot\G(\eta)\psi\dx
+\frac{1}{2}\int_{\R}\psi\big(\partial_t\G(\eta)\big)\psi\dx
+\int_\R \partial_t\eta\cdot\big(\partial_x^4\eta+\eta\big)\dx,
\end{aligned}
$$
where we used $u=(\Id+\G(\eta))^{-1}\psi$ and the fact that $\G(\eta)$ is self-adjoint. 
This computation is justified since the time derivatives $\partial_t\eta$ and $\partial_tu$ possess sufficient regularity, as ensured by (\ref{Regularity}). Substituting in (\ref{Regularity}), we find that the first and third integrals exhibit several cancellations, leaving only
$$
E'(t)=\frac{1}{2}\int_{\R}\psi\big(\partial_t\G(\eta)\big)\psi\dx
-\int_\R N(\eta,u)\G(\eta)\psi\dx.
$$
However, applying the variational version of the shape derivative formula, namely (\ref{eq-reg-shpder:LinkShpDerNonlin}), we obtain
$$
\frac{1}{2}\int_{\R}\psi\big(\partial_t\G(\eta)\big)\psi\dx
=\int_\R N(\eta,u)\partial_t\eta \dx,
$$
which exactly cancels with the last integral. As before, this computation is justified since $\partial_t\eta$ belongs to $H^s\subset H^2$ in $x$. Thus, we conclude that $E'(t)\equiv0$.

Consequently, $\sqrt{E_0}$ provides a bound for the norm $\|\eta\|_{H^2_x}$, and also $\|\mathcal{G}(\eta)\psi\|_{L^2_x}$:
$$
\|\eta\|_{H^2_x}+\|\mathcal{G}(\eta)\psi\|_{L^2_x}\lesssim \sqrt{E_0}.
$$
Applying Proposition \ref{prop-reg-bas:Trace}, we further obtain
\begin{align*}
	\|\psi\|_{\dot{H}^{1/2}_x}^2
	\lesssim&
	\big( 1+\la\eta_x\ra_{L^\infty_x} \big) \int_{\mathbb{R}} \psi \mathcal{G}(\eta)\psi \dx \\
	\lesssim& \big( 1+\|\eta\|_{H^{2}_x} \big) \int_{\mathbb{R}} \psi \mathcal{G}(\eta)\psi \dx
	\lesssim E_0+E_0^{3/2}.
\end{align*}
Since $u= \psi + \mathcal{G}(\eta)\psi$, we find that in order to bound $\|u\|_{L^2_x}$, it remains to control the low-frequency component $\|\Delta_0u\|_{L^2_x}$, where $\Delta_0$ is the dyadic building block defined in \eqref{eq-besov-def:DyaBlk}. From the second equation of \eqref{n40}, we compute
\begin{align*}
	\frac{\diff}{\dt}\|\Delta_0 u\|_{L^2_x}^2 
	&= -2 \int_{\R} \Delta_0u \cdot \Delta_0 N(\eta,u) \dx 
	-2 \int_{\R} \Delta_0u 
	\cdot \Delta_0 ( \partial_x^2\eta+ \eta ) \dx \\
	&\lesssim |\Delta_0 u|_{L^\infty_x} \|\Delta_0 N(\eta,u)\|_{L^1_x} + \|\Delta_0 u\|_{L^2_x} \cdot\big\|\Delta_0 ( \partial_x^2\eta+ \eta )\big\|_{L^2_x} \\
	&\lesssim \|\Delta_0 u\|_{L^2_x} \left( \|N(\eta,u)\|_{L^1_x} + \|\eta\|_{L^2_x} \right).
\end{align*}
Recalling the definition \eqref{eq-intro-csz:N2} of $N(\eta,u)$, the $L^1_x$-norm of $N(\eta,u)$ can be bounded by
\begin{equation*}
\begin{aligned}
	\|N(\eta,u)\|_{L^1_x}
	&\le \|\partial_x\eta\cdot\partial_x\psi\|_{L^2_x} + \|\mathcal{G}(\eta)\psi\|_{L^2_x} \\
	&\le K\big(|\partial_x\eta|_{L^\infty_x}\big)\|\mathcal{G}(\eta)\psi\|_{L^2_x}
	\le K(E_0).
\end{aligned}
\end{equation*}
The bound does not depend on the length of the time interval. Therefore, applying Grönwall's inequality to $\|\Delta_0u\|_{L^2_x}$, we obtain, for arbitrary $t\in[0,T^*)$,
$$
\|\Delta_0u(t)\|_{L^2_x}
\le \|\Delta_0u_0\|_{L^2_x} + K(E_0) (1+t) \le C\|u_0\|_{L^2_x} + K(E_0) (1+t),
$$
which completes the proof of \eqref{eq-cauchy-gwp:Main} when the initial data $(\eta_0,u_0)\in H^{s+2}\times H^s$. 

Now, consider an initial data $(\eta_0,\psi_0)\in H^2\times H^1$ for (\ref{n42}), meaning that $U_0\in L^2$ 
for (\ref{Strong}). Let $U_0^n$ be a sequence of Schwartz functions converging to $U_0$ in $L^2$. Then for the time $T=T(\|U_0\|_{L^2})$, Proposition \ref{Contraction} shows that the corresponding solution sequence $U^n=\mathcal{S}(U_0^n)$ converges to $U=\mathcal{S}(U_0)$ in $L^\infty_TH^{-1}_x\cap L^4_TY^{-1}_x$, while Proposition \ref{Self} shows that the sequences $U^n$ is bounded in $L^\infty_TL^2_x\cap L^4_TY^0_x$. Consequently, $U^n$ converges in the weak-$\ast$ topology of $L^\infty_TL^2_x$ to $U$. Since each $U^n$ satisfies the a priori bound (\ref{eq-cauchy-gwp:Main}), its weak limit $U$ does as well, 
at least on the time interval $[0,T]$. By continuous induction, we extend this result to the entire interval 
$[0,T^*)$, concluding the proof of (\ref{eq-cauchy-gwp:Main}) for initial data $(\eta_0,\psi_0)$ of minimal regularity $H^2\times H^1$.
\end{proof}

The proof of Theorem \ref{thm:main} follows as a direct corollary of Proposition \ref{prop-cauchy-gwp:Main}. In fact, the local well-posedness result, Proposition \ref{Contraction}, relies only on the finiteness of the $L^2_x$ norm of $U_0$. On the other hand, the a priori bound for $\|U(t)\|_{L^2_x}$ in terms of the initial energy is guaranteed by Proposition \ref{prop-cauchy-gwp:Main}, which remains finite for any finite time $t$. Therefore, the solution map extends beyond any prescribed time. The regularity of the solution when $(\eta_0, u_0) \in H^2 \times L^2$ follows from Assertion \ref{52}. Eventually, the regularity of the solution when $(\eta_0, u_0) \in H^{s+2} \times H^s$ for $s > 0$ follows from Assertion~\ref{54}.

Finally, we have to justify the continuous dependence on the initial data. To do so, we use the automatic continuity result proved in~\cite{ABITZ} (see Theorem 1 therein). Specifically, using the terminology in~\cite{ABITZ}, 
we begin by remembering that the weak Lipschitz estimate for the solution map $(\eta_0, u_0) \mapsto (\eta, u)$, from $H^{1} \times H^{-1}$ to $L^\infty_{\mathrm{loc}}(\mathbb{R}; H^{1} \times H^{-1})$ is a direct consequence of \eqref{eq-cauchy-lwp:LipIniData}. On the other hand, the result from the previous paragraph regarding the propagation of regularity implies that the flow map 
satisfy tame estimates (in the sense of~\cite{ABITZ}). Therefore, Theorem 1 
in~\cite{ABITZ} implies that the solution map is also continuous from $H^{s+2} \times H^{s}$ to $C^0_{\mathrm{loc}}(\mathbb{R}; H^{s+2} \times H^{s})$ for all $s\ge 0$. This completes the proof of Theorem \ref{thm:main}.

\begin{remark}During the proof of Proposition \ref{prop-cauchy-gwp:Main}, we justified the conservation of energy $E(t)$ only for solutions with high regularity. The weak convergence argument used therein guarantees that $E(t)$ is non-increasing for initial data with minimal regularity. However, it is not clear whether the energy is truly conserved in that regime; the argument in this paper does not rule out the possibility of non-viscous dissipation.
\end{remark}

\appendix
\section{Paraproducts and Zygmund spaces}\label{app:Besov}

\subsection{Basic definitions}\label{subsect-besov:Def}

We fix a dyadic decomposition on the Euclidean space $\R^d$ with general dimension $d\ge 1$,
\begin{equation*}
	1 = \chi(\xi) + \sum_{j=1}^\infty \varphi(2^{-j}\xi),\ \ \forall \xi\in\R^d,
\end{equation*}
such that $\chi,\varphi\in C^\infty_0(\R)$ are radial and supported in a ball and an annulus centered at zero, respectively, and, for all $N\in\N$,
\begin{equation*}
	\chi(2^{-N}\xi) = \chi(\xi) + \sum_{j=1}^{N}\varphi(2^{-j}\xi).
\end{equation*}
Then we denote
\begin{equation}\label{eq-besov-def:DyaBlk}
	\Delta_j := \left\{
	\begin{array}{cl}
		\varphi(2^{-j}D_x), & j\ge 1, \\
		\chi(2^{-j}D_x), & j=0, \\
		0, & j\le -1,
	\end{array}
	\right.
	\ \ 
	S_j := \left\{
	\begin{array}{cl}
		\chi(2^{-j}D_x), & j\ge 1, \\
		\chi(2^{-j}D_x), & j=0, \\
		0, & j\le -1.
	\end{array}
	\right.
\end{equation}

\begin{definition}[Function space]\label{def-besov-def:Sobo}
    Let $s\in\R$. We denote by $H^s$ the space of those tempered distributions $u$ with
    \begin{equation}\label{eq-besov-def:Sobo}
        \|u\|_{H^s} := \| \langle D \rangle^{s}u \|_{L^2} < \infty.
    \end{equation}
    We denote by $C^\sigma_*$ the Zygmund space of those tempered distributions $u$ with
    \begin{equation}\label{eq-besov-def:Zyg}
        |u|_{C^\sigma_*} := \sup_{j\in\N}2^{j\sigma}\la\Delta_j u\ra_{L^\infty}< \infty.
    \end{equation}
\end{definition}

\subsection{Paraproduct decomposition and basic inequalities}\label{subsect-besov:Bony}

Let us recall the decomposition of products introduced by Coifman--Meyer and Bony, 
known as \textit{paraproduct decomposition}:
\begin{equation}\label{eq-besov-bnoy:Decomp}
    ab = T_a b + T_b a + R(a,b),
\end{equation}
where the paraproduct $T_a$ and the remainder 
$R(a,b)$ are defined as
\begin{equation}\label{eq-besov-bony:DefPara}
    T_a b := \sum_{j\ge N} S_{j-N}a \Delta_j b,
\quad
    R(a,b) = \sum_{|j-j'|<N} \Delta_j a \Delta_{j'}b.
\end{equation}
Here $N\in\N$ is a fixed integer depending only on the choice of dyadic blocks (see \eqref{eq-besov-def:DyaBlk}), such that, for all $j\ge N$ and all tempered distribution $a,b$, the Fourier transform of $S_{j-N}a \Delta_j b$ is supported in an annulus of size $2^{j}$. We refer to the original paper \cite{bony1981calcul} for more details of this theory. In this section, we shall introduce some basic inequalities involving only Sobolev and H{\"o}lder spaces, which are frequently used in previous sections. Most of the results below can be found in Section 2 of \cite{bahouri2011fourier}, where the inequalities are stated for general Besov spaces.

\begin{proposition}[Proposition 2.82 of \cite{bahouri2011fourier}]\label{prop-besov-bony:EstiPara}
    Let $u\in H^s$ with $s\in\R$. Then, if $a\in L^\infty$, we have
    \begin{equation}\label{eq-besov-bony:EstiParaLinfty}
        \|T_a u\|_{H^s} \lesssim |a|_{L^\infty} \|u\|_{H^s};
    \end{equation}
    and if $a\in C^{-\sigma}_*$ for some $\sigma>0$, the inequality becomes
    \begin{equation}\label{eq-besov-bony:EstiParaNegInd}
        \|T_a u\|_{H^{s-\sigma}} \lesssim \|a\|_{C^{-\sigma}_*} \|u\|_{H^s}.
    \end{equation}
\end{proposition}

\begin{proposition}[Proposition 2.85 of \cite{bahouri2011fourier}]\label{prop-besov-bony:EstiRem}
    Let $a\in H^s$and $u\in C^\sigma_*$ with $s,\sigma\in\R$. If $s+\sigma>0$, we have
    \begin{equation}\label{eq-besov-bony:EstiRem}
        \|R(a,u)\|_{H^{s+\sigma}} \lesssim \|a\|_{H^s} \|u\|_{C^\sigma_*}.
    \end{equation}
    Moreover, for all $v\in H^{s'}$ with $s+s'>0$, we also have
    \begin{equation}\label{eq-besov-bony:EstiRemSobo}
        \|R(a,v)\|_{H^{s+s'-d/2}} \lesssim \|a\|_{H^s} \|v\|_{H^{s'}}.
    \end{equation}    
    In the limit case $s+s'=0$, we have, for any $\varepsilon>0$,
    \begin{equation}\label{eq-besov-bony:EstiRemSoboLim}
    	\|R(a,v)\|_{H^{-d/2-\varepsilon}} 
    	\lesssim \|a\|_{H^s} \|v\|_{H^{s'}}.
    \end{equation}
\end{proposition}

By combing the results above, one obtains the following product law.
\begin{corollary}\label{cor-besov-bony:ProdLaw}
    Consider real numbers $s,s',r,\varepsilon\in\R$ and functions $a\in H^{s}$, $b\in H^{s'}$. The following estimates on the product $ab$ hold.
    \begin{enumerate}
        \item If $s+s'>0$ and $r\le \min(s,s',s+s'-d/2)$, then the product $ab$ belongs to $H^{r}$ with 
        \begin{equation}\label{eq-besov-bony:ProdLaw}
        \|ab\|_{H^r} \lesssim \|a\|_{H^s} \|b\|_{H^{s'}}.
        \end{equation}
        
        \item In the end point case $s+s'=0$, the product $ab$ belongs to $H^{-d/2-\varepsilon}$ for all $\varepsilon>0$ with
        \begin{equation}\label{eq-besov-bony:ProdLawEndPt}
            \|ab\|_{H^{-d/2-\varepsilon}} \lesssim \|a\|_{H^s} \|b\|_{H^{s'}}.
        \end{equation}
        
        \item When $s,s'\ge 0$, if additionally $a,b\in L^\infty$, then there holds
        \begin{equation}\label{eq-besov-bony:ProdLaw-tame}
            \|ab\|_{H^{\min(s,s')}} \lesssim |a|_{L^\infty} \|b\|_{H^{s'}} + \|a\|_{H^s} |b|_{L^\infty}. 
        \end{equation}
    \end{enumerate}
\end{corollary}

By combining the two previous points with the embedding~$H^\mu({\R}^d)\subset C^{\mu-d/2}_*({\R}^d)$ (for any~$\mu\in\R$) 
we immediately obtain the following results. 
\begin{proposition}\label{lemPa}
Let~$r,\mu\in \R$ be such that~$r+\mu>0$. If~$\gamma\in\R$ satisfies 
$$
\gamma\le r  \quad\text{and}\quad \gamma < r+\mu-\frac{d}{2},
$$
then there exists a constant~$K$ such that, for all 
$a\in H^{r}({\R}^d)$ and all~$u\in H^{\mu}({\R}^d)$, 
%we have
$$
\lA au - T_a u\rA_{H^{\gamma}}\le K \lA a\rA_{H^{r}}\lA u\rA_{H^\mu}.
$$

Furthermore, in the limit case $r+\mu=0$, $\mu\le 0$, for any $\varepsilon>0$, we also have
$$
\lA au - T_a u\rA_{H^{-d/2-\varepsilon}}\le K \lA a\rA_{H^{r}}\lA u\rA_{H^\mu}.
$$
\end{proposition}

Another important result is about the composition of smooth functions with Sobolev functions. 

\begin{proposition}[Theorem 2.89 of \cite{bahouri2011fourier}]\label{prop-besov-bony:Paralin}
    For all function $F\in C^\infty$ and for all $u\in H^s$ with $s>d/2$, the difference
    $F(u) - F(0)$ belongs to $H^{s}$ and the following tame estimate holds:
    \begin{equation}\label{eq-besov-bony:Compo}
        \left\| F(u) - F(0) \right\|_{H^{s}} \le K\big(|u|_{L^\infty}\big)\|u\|_{H^s}.
    \end{equation}
    Similarly, for all~$s\geq 0$ and all $F\in C^\infty$,
\begin{equation}\label{esti:F(u)bis}
\lA F(u)-F(0)\rA_{C^{s}_*}\le K(\la u\ra_{L^\infty})\lA u\rA_{C^{s}_*},
\end{equation}
for all~$u\in C^{s}_*$. 
\end{proposition}

\subsection{Further estimates on paraproduct}\label{subsect-besov:Para}

In the previous section, we have studied the Sobolev regularity of the paraproduct $T_a b$ where $a$ belongs to spaces based on $L^\infty$ and $b$ lies in Sobolev spaces. A natural question is whether the same Sobolev regularity remains true when one reverses the spaces to which $a,b$ belong. 

To begin with, we focus on the case where $a\in H^s$ with strictly negative index $s<0$.

\begin{proposition}\label{prop-besov-para:NegInd}
    Let $a\in H^s$ and $b\in C^\sigma_*$ with $s<0$ and $\sigma\in\R$, we have
    \begin{equation}\label{eq-besov-para:NegInd}
        \|T_a b\|_{H^{s+\sigma}} \lesssim \|a\|_{H^s}\|b\|_{C^\sigma_*}.
    \end{equation}
\end{proposition}
\begin{proof}
    Due to the construction \eqref{eq-besov-def:DyaBlk} of dyadic blocks, we have
    \begin{equation*}
        \Delta_j \left(T_a b\right) = \sum_{|j-j'|<N} S_{j'}a \Delta_{j'}b,
    \end{equation*}
    for some $N\in\N$ independent of $j,j'$. Then we have
    \begin{align*}
    	2^{j(s+\sigma)}\|\Delta_j\left(T_ab\right)\|_{L^2} &\le 2^{j(s+\sigma)}\sum_{|j-j'|<N} \|S_{j'} a \Delta_{j'} b\|_{L^2} \\
    	&\le 2^{j(s+\sigma)}\sum_{|j-j'|<N} \|S_{j'} a\|_{L^2} |\Delta_{j'} b|_{L^\infty} \\
    	&\le 2^{js}\sum_{|j-j'|<N} 2^{(j-j')\sigma} \|S_{j'} a\|_{L^2} \|b\|_{C_*^\sigma} \\
    	&= \sum_{|j-j'|<N} 2^{(j-j')(s+\sigma)} 2^{j's}\|S_{j'} a\|_{L^2} \|b\|_{C_*^\sigma} \\
    	&\lesssim \sum_{|j-j'|<N} 2^{j's}\|S_{j'} a\|_{L^2} \|b\|_{C_*^\sigma}.
    \end{align*}
    By taking $\ell^2$-norm in $j$ and apply the characterization
    $$
    \sum_{j\in\N}2^{js}\|S_j a\|_{L^2} \sim \|a\|_{H^s}, \quad \text{for }s<0,
    $$
    which is proved in Proposition 2.79 of \cite{bahouri2011fourier}, one can conclude the desired inequality \eqref{eq-besov-para:NegInd}.
\end{proof}

Eventually, we also need the 
following estimate which complements the previous one when $s=0$. 

\begin{corollary}\label{cor-besov-para:L2Hol}
    Let $a\in L^2$ and $b\in W^{k,\infty}$ with $k\in\N$. Then, we have
    \begin{equation}\label{eq-besov-para:L2Hol}
        \|T_a b\|_{H^k} \lesssim \|a\|_{L^2} \|b\|_{W^{k,\infty}}.
    \end{equation}
\end{corollary}
\begin{proof}
The proof follows easily from  Proposition~ \ref{Bilinear}; however, for lack of a reference we include a proof. 

We begin by noticing that Proposition~\ref{Bilinear} implies that $T_a b$ belongs to $L^2$ so that it suffices to show that all the $k$-order derivatives of $T_a b$ belong to $L^2$. Consider an arbitrary index $\alpha\in\N^d$ with $|\alpha|=k$ and write
    \begin{equation*}
        \partial^\alpha \left( T_ab \right) = T_a \partial^\alpha b + \sum_{\beta+\gamma=\alpha,\ |\beta|\ge 1} T_{\partial^\beta a}\partial^\gamma b.
    \end{equation*}
According to Proposition \ref{Bilinear}, the first term belongs to $L^2$, since the condition $b\in W^{k,\infty}$ guarantees that $\partial^\alpha b\in L^\infty \subset BMO$. As for the second term, observe that $\partial^\beta a \in H^{-|\beta|}$ with $|\beta|>0$. Therefore, it is possible to apply Proposition~\ref{prop-besov-para:NegInd} with $s=-|\beta|$ and $\sigma=k-|\gamma|$, to get
    \begin{equation*}
        \left\| T_{\partial^\beta a}\partial^\gamma b \right\|_{L^2} \lesssim \|\partial^\beta a\|_{H^{-|\beta|}} \|\partial^\gamma b\|_{C^{k-|\gamma|}_*} \lesssim \|a\|_{L^2} \|b\|_{C^k_*} \lesssim \|a\|_{L^2} \|b\|_{W^{k,\infty}},
    \end{equation*}
    which completes the proof by taking the summation over all indices such that $\beta+\gamma=\alpha$ and $ |\beta|\ge 1$.
\end{proof}

\section{Paradifferential calculus}\label{app:ParaDiff}

In Appendix \ref{app:Besov}, we have seen the definition \ref{eq-besov-bony:DefPara} of paraproduct, which can be regarded as a refinement of multiplication with rough functions. The \textit{paradifferential operator} to be introduced below is a generalization of this notion. For a symbol $a(x,\xi)$, we shall modify it such that its frequency in $x$ is controlled by the size of $\xi$, which allows us to re-prove the results of pseudo-differential operators (boundedness on Sobolev space, symbolic calculus, commutator estimate, etc) for symbols rough in $x$.

To begin with, we review some basic notions about pseudo-differential operators.

\begin{definition}[Pseudo-differential operator]\label{def-paradiff:PDO}
	Let $a$ be a tempered distribution on $\R^d \times \R^d$. For all Schwartz function $u\in \mathcal{S}(\R^d)$, we define
	\begin{equation}\label{eq-paradiff:PDO}
		\left\langle \Op{a}u, v\right\rangle_{\mathcal{S}'\times\mathcal{S}} := \frac{1}{(2\pi)^d} \iint e^{ix\cdot\xi}a(x,\xi)\hat{u}(\xi) \overline{v(x)}\dxi \dx.
	\end{equation}
	Then $\Op{a}$ is a continuous application from Schwartz functions $\mathcal{S}(\R^d)$ to tempered distributions $\mathcal{S}'(\R^d)$.  The mapping $\Op{a}$ is called the  \textit{pseudo-differential operator} with \textit{symbol} $a$.
\end{definition}

Now, we are in the position to define  \textit{paradifferential operators}. We use the Fourier multiplier $S_j$ 
and the cutoff function $\varphi$ defined in \eqref{eq-besov-def:DyaBlk}. 
We also used an index $N\in\N$ 
chosen large enough as in the definition \eqref{eq-besov-bony:DefPara} of paraproduct.

\begin{definition}[Paradifferential operator]\label{def-paradiff:Paradiff}
	Let $a = a(x,\xi)$ be a symbol smooth in $\xi\neq 0$ with Sobolev or H{\"o}lder regularity in $x$. Then the paradifferential operator of symbol $a$ is defined as
	\begin{equation}\label{eq-paradiff:Paradiff}
		T_a := \Op{\sigma_a},\ \ \sigma_a(x,\xi) := \sum_{j\ge N} (S_{j-N}a)(x,\xi) \varphi\left(\frac{\xi}{2^j}\right),
	\end{equation}
	where $S_{j-N}$ acts on $x$ variable only:
	$$
	(S_{j-N} a)(x,\xi)=(\chi(2^{-(j-N)}D_x)a(\cdot,\xi))(x) \quad\text{for}\quad j\ge N,
	$$
	and $S_{j-N}a=0$ for $j<N$.
\end{definition}
\begin{remark}
	In particular, when $a = a(x)$ is independent of $\xi$, this definition coincides with the one of the 
	paraproduct $T_a$ (see \eqref{eq-besov-bony:DefPara}). 
\end{remark}

\begin{remark}
Note that we consider symbols $a(x,\xi)$ that need not be smooth for $\xi=0$ (for instance $a(x,\xi)=|\xi|^{m}$ for any $m\in\R$). In fact, by construction, the low-frequency information $|\xi| \ll 1$ of $u$ is eliminated, meaning that paradifferential operators are never bijective. Nevertheless, for elliptic symbols, it is possible to construct left and right inverses, up to some reasonable remainders, which are known as \textit{parametrices}.
\end{remark}

\begin{definition}\label{def-paradiff:SymbClass}
	Let $\rho\ge0$ and $m\in\R$. The symbol class $\Gamma_\rho^m$ is defined as the collection of symbols $a=a(x,\xi)$ H{\"o}lder in $x$ and smooth in $\xi$, such that the following semi-norm is finite,
	\begin{equation}\label{eq-paradiff:SymbNorm}
		M^m_\rho(a) := \sup_{|\alpha|\le \rho+1+d/2} \sup_{|\xi|>1/2} \left\| \langle\xi\rangle^{-m+|\alpha|} \partial_\xi^\alpha a(\cdot,\xi) \right\|_{W^{\rho,\infty}_x} < +\infty.
	\end{equation}
	Recall that $W^{\rho,\infty}$ is the usual space of Lipschitz continuous functions when $\rho\in\N$, and H{\"o}lder space otherwise (see Definition \ref{def-intro-main:HolSpace}).
\end{definition}

Now, we are able to show boundedness and symbolic calculus for paradifferential operators. The proof of following results can be found in Chapter 5 of \cite{metivier2008para}.
\begin{proposition}\label{prop-paradiff:Bd}
	Given $m\in\R$, for all $a\in\Gamma^m_0$, the paradifferential operator $T_a$ is of order $m$, namely
	\begin{equation}\label{eq-paradiff:Bd}
		\|T_a\|_{\mathcal{L}\left(H^{s};H^{s-m}\right)} \lesssim_{s} M^m_0(a),\ \ \forall s\in\R.
	\end{equation}
\end{proposition}

\begin{proposition}\label{prop-paradiff:SymbCal}
	Let $a\in\Gamma^m_\rho$ and $b\in\Gamma^{m'}_\rho$ with $m,m'\in\R$ and $\rho> 0$. Then the composition $T_a T_b$ and adjoint $T_a^*$ are both paradifferential operators, such that, for all $s\in\R$,
	\begin{align}
		\big\| T_aT_b - T_{a\#_\rho b};{\mathcal{L}(H^s;H^{s-m-m'+\rho})}\|
		\lesssim_s& M^m_\rho(a)M^{m'}_0(b) + M^m_0(a)M^{m'}_\rho(b), \label{eq-paradiff:SymbCalComp} \\
		\| T_a - T_{a^*} \|_{\mathcal{L}(H^s;H^{s-m+\rho})} \lesssim_s& M^m_\rho(a), \label{eq-paradiff:SymbCalAdj}
	\end{align}
	where
	\begin{equation*}
		a \#_\rho b = \sum_{|\alpha|<\rho} \frac{1}{\alpha!} \partial_\xi^\alpha a D_x^\alpha b,
		\quad a^* = \sum_{|\alpha|<\rho} \frac{1}{\alpha!} \partial_\xi^\alpha D_x^\alpha \overline{a}. 
	\end{equation*}
\end{proposition}
As a corollary, we have the following commutator estimate:
\begin{corollary}\label{cor-paradiff:CommuEsti}
	If $a\in\Gamma^m_\rho$, $b\in\Gamma^{m'}_\rho$ with $m,m'\in\R$ and $\rho>0$, the commutator $[T_a,T_b]$ is of order $m+m'-\min(\rho,1)$ in the sense that
	\begin{equation}\label{eq-paradiff:CommuEsti}
		\big\| [T_a,T_b] ;{\mathcal{L}(H^s;H^{s-m-m'+\min(\rho,1)})}\big\| \lesssim_s M^m_\rho(a)M^{m'}_0(b) + M^m_0(a)M^{m'}_\rho(b).
	\end{equation}
\end{corollary}

To end this part, we introduce a result concerning paralinear parabolic problems,
\begin{proposition}\label{prop-paradiff:Parabolic}
	Let~$r\in \R$,~$\rho\in (0,1)$,~$J=[z_0,z_1]\subset\R$ and let 
	$p\in \Gamma^{1}_{\rho}({\R}^d\times J)$ 
	satisfying$$
	\RE p(z;x,\xi) \geq c \la\xi\ra,
	$$
	for some positive constant~$c$. Then for any ~$f\in L^2_z(J;H^{r-1/2}_x)$ and~$w_0\in H^{r}$, there exists $w\in C^0_z(J;H^r_x)\cap L^2_z(J;H^{r+1/2}_x)$ solution of the parabolic evolution equation
	\begin{equation}\label{eq-paradiff:Parabolic}
		\partial_z w + T_p w =f,\quad w\arrowvert_{z=z_0}=w_0,
	\end{equation}
	satisfying 
	\begin{equation}\label{eq-paradiff:ParabolicEstiBas}
		\lA w \rA_{C^0_z(J;H^r_x)\cap L^2_z(J;H^{r+1/2}_x)}\le C\cdot\left(\lA w_0\rA_{H^{r}_x}+ \lA f\rA_{L^2(J;H^{r-1/2}_x)}\right),
	\end{equation}
	for some positive constant~$C$ depending only on~$r,\rho,c$ and~$M^1_\rho(p)$. Moreover, this solution is unique in ~$C^0_z(J;H^s_x)\cap L^2_z(J;H^{s+1/2}_x)$ for any~$s\in \R$.
\end{proposition}

\section{Elliptic regularity}\label{app:Ellip}

The purpose of this appendix is to prove Proposition \ref{prop-reg-bas:EllipReg} about the elliptic regularity for the equation \eqref{eq-reg-bas:EllipEq}. For the sake of generality, we shall prove a result which holds in arbitrary dimension. Given $d\ge 1$, we consider the following boundary value problem in the lower half-space $\R^{d+1}_-=\{(x,z)\in\R^{d}\times\R:z<0\}$,
\begin{equation}\label{eq-ellip:MainEq}
	\left\{\begin{array}{ll}
		\alpha\partial_z^2 v + \beta\cdot\nabla_x\partial_zv + \Delta_x v - \gamma\partial_z v = 0 & \text{on }\R^{d+1}_-, \\ [0.5ex]
		v|_{z=0} = \psi, & \\ [0.5ex]
		\lim_{z\rightarrow -\infty} \nabla_{x,z}v = 0, & 
	\end{array}\right.
\end{equation}
where the coefficients $\alpha=\alpha(x),\beta=\beta(x)$ and $\gamma=\gamma(x)$ 
are given by:
\begin{equation}\label{eq-ellip:Coeff}
	\alpha := 1 + |\nabla_x\eta|^2,\ \ \beta := -2\nabla_x\eta,\ \ \gamma :=\Delta_x\eta.
\end{equation}
In what follows, we shall prove the following:
\begin{proposition}\label{prop-ellip:Main}
	Consider three real numbers $s,s_0,\sigma$ satisfying
	\begin{equation}\label{n200}
	s\ge s_0, \ \mez \le \sigma \le s+\frac{1}{2}, \quad \text{with} \quad  s_0>d/2,\ s_0\ge 1.
	\end{equation}
	Given $\eta\in H^{s+1}(\R^d)$, $\psi\in H^{\sigma}(\R^d)$, 
	the equation \eqref{eq-ellip:MainEq} admits a unique variational solution $v\in H^{1,0}$ (see Definition~\ref{def-ellip-var:VarSol} below). Moreover, this solution verifies, for any $z_0<0$,
	\begin{equation}\label{eq-ellip:Main}
	\begin{aligned}
	    &\|\nabla_{x,z}v\|_{L^2_z(z_0,0;H_x^{\sigma-1/2})} + \|\nabla_{x,z}^2v\|_{L^2_z(z_0,0;H_x^{\sigma-3/2})} \\
	    &\hspace{8em}\le K(\|\eta\|_{H^{s_0+1}}) \big( \|\eta\|_{H^{s+1}} \|\psi\|_{H^{1/2}} + \|\psi\|_{H^{\sigma}} \big).
	\end{aligned}
	\end{equation}
\end{proposition}

The rest of this section is devoted to the proof of Proposition~\ref{prop-ellip:Main}.

\subsection{Variational formulation}\label{subsect-ellip:Var}
The first part of the proof consists in reformulating System \eqref{eq-ellip:MainEq} into a  variational form. To begin with, we transform the system \eqref{eq-ellip:MainEq} to an elliptic problem with a homogeneous 
Dirichlet boundary condition at $z=0$ and a non-trivial source term. Consider the following simple extension $\Psi=\Psi(x,z)$ of $\psi=\psi(x)$ defined in the lower half-space $\R^{d+1}_-$ by:
\begin{equation}\label{eq-ellip-var:DefPsi}
	\Psi(x,z) = \Big(e^{z\Dx}\psi\Big)(x),
\end{equation}
where $e^{z\Dx}$ is the Fourier multiplier with symbol $e^{z\la\xi\ra}$. 
For any $\psi\in H^\sigma$ with $\sigma\in\R$, the Plancherel theorem ensures that, for all $k\in\N$,
\begin{equation}\label{eq-ellip-var:EstiPsi}
	\|\partial_z^k \Psi\|_{L^2(\R_-;H^{\sigma+1/2-k})} \lesssim \|\psi\|_{H^\sigma}.
\end{equation}
Moreover, $\Psi$ decays exponentially to zero as $z\to-\infty$. Then the difference $w:= v-\Psi$ satisfies
\begin{equation}\label{eq-ellip-var:EqW}
	\left\{\begin{array}{ll}
		\nabla_{x,z}\cdot \left( A\nabla_{x,z}w \right) = -\nabla_{x,z}\cdot \left( A\nabla_{x,z}\Psi \right) & \text{in }\R^{d+1}_-, \\ [0.5ex]
		w|_{z=0} = 0, & \\ [0.5ex]
		\lim_{z\rightarrow -\infty} \nabla_{x,z}w = 0, & 
	\end{array}\right.
\end{equation}
where the matrix $A$ is given by
\begin{equation}\label{eq-ellip-var:DefA}
	A = \left(\begin{array}{cc}
		1 & \nabla_x\eta \\
		\nabla_x\eta^T & 1+|\nabla_x\eta|^2
	\end{array}\right).
\end{equation}

Before stating the variational formulation, we clarify the functional space to be used, following \cite{ABZ2014}. For arbitrary $f,g\in C^\infty_0(\R^{d+1}_-)$, we define a bilinear form,
\begin{equation}\label{eq-ellip-var:SpaceScalProd}
	\langle f,g \rangle_{H^{1,0}} :=  \iint_{\R^{d+1}_-} \nabla_{x,z} f \cdot \nabla_{x,z} g \dx\dz.
\end{equation}
One can verify that $\langle \cdot,\cdot \rangle_{H^{1,0}}$ defines an inner product on $C^\infty_0(\R^{d+1}_-)$. Then, we denote by $H^{1,0}$ the completion of $C^\infty_0(\R^{d+1}_-)$ with respect to the inner product $\langle \cdot,\cdot \rangle_{H^{1,0}}$. 

We claim that $H^{1,0}$ is a Hilbert space contained in the space of distributions $\mathcal{D}'(\R^{d+1}_-)$.
\begin{lemma}\label{lem-ellip-var:Space}
	Let $f\in H^{1,0}$. Then $f$ is a locally integrable function satisfying 
	\begin{equation}\label{eq-ellip-var:Space}
		\iint |f(x,z)|^2 e^{-|z|} \dx\dz \le C \|f\|_{H^{1,0}},
	\end{equation}
	where $C>0$ is a universal constant.
\end{lemma}
\begin{proof}
    Consider first a smooth function $f\in C^\infty_0(\R^{d+1}_-)$ and observe that, for all $(x,z)\in\R^{d+1}_-$, we have
	$$ 
	|f(x,z)| = \left| \int_0^z f_z(x,z') \dz' \right| \le \sqrt{|z|} \sqrt{\int_{-\infty}^0 |f_z(x,z')|^2 \dz'},
	$$
	to deduce that
	$$
	\iint_{\R^{d+1}_-} |f(x,z)|^2 e^{-|z|}\dx\dz \le \|f_z\|_{L^2(\dx\dz)}^2 \int_{-\infty}^0 e^{-|z|}|z|\dz  \le C \|f\|_{H^{1,0}}^2.
	$$
	To extend this result to $f\in H^{1,0}$, it is sufficient 
	to consider a sequence $f_k\in C^\infty_0(\R^{d+1}_-)$ converging to $f$ in $H^{1,0}$.
\end{proof}

Now, we are in a position to introduce the variational problem associated to our main system \eqref{eq-ellip:MainEq}. For later use, the variational formulation will be presented in the following slightly general form: given $F\in L^2(\R^{d+1}_-)$, we seek $w\in H^{1,0}$ such that
\begin{equation}\label{eq-ellip-var:VarForm}
	\iint_{\R^{d+1}_-} A\nabla_{x,z}w \cdot \nabla_{x,z}\varphi \dx\dz = - \iint_{\R^{d+1}_-} F\cdot\nabla_{x,z} \varphi \dx\dz,\ \ \forall \varphi\in H^{1,0}.
\end{equation}

\begin{remark}\label{rmk-ellip-var:DistriForm}
    From the definition of $H^{1,0}$ above, the class of smooth compactly supported functions $C^\infty_0(\R^{d+1}_-)$ is dense in $H^{1,0}$. Thus, for $w\in H^{1,0}$ and $F\in L^2(\R^{d+1}_-)$, the variational problem \eqref{eq-ellip-var:VarForm} holds if and only if 
    \begin{equation}\label{eq-ellip-var:DistriForm}
		\nabla_{x,z}\cdot\left(A\nabla_{x,z}w\right) = -\nabla_{x,z}\cdot F, \quad\text{in }\mathcal{D}'(\R^{d+1}_-),
	\end{equation}
	where $\mathcal{D}'(\R^{d+1}_-)$ represents the class of distributions on $\R^{d+1}_-$. For the sake of simplicity, in Section~\ref{subsect-ellip:reg} below, we shall use the distributional formulation \eqref{eq-ellip-var:DistriForm} instead of the original one \eqref{eq-ellip-var:VarForm}.
\end{remark}

The following proposition asserts that this variational problem has a solution, and also includes a statement about higher elliptic regularity. We will use the following notations: given $r\in\R$ and $z_0<0$,
\begin{align}
	\| w \|_{\mathcal{X}^{r}_{z_0}} :=& \|\nabla_{x,z}w\|_{L^2(\R^{d+1}_-)} + \|\nabla_{x,z}w\|_{L^2(z_0,0;H^{r+1/2})}  \label{eq-ellip-var:AuxSpX} \\
	&\hspace{12em}+ \|\nabla_{x,z}^2w\|_{L^2(z_0,0;H^{r-1/2})}, \nonumber \\
	\| F \|_{\mathcal{Y}^{r}_{z_0}} :=& \|F\|_{L^2(\R^{d+1}_-)} + \|\nabla_{x,z}\cdot F\|_{L^2(z_0,0;H^{r-1/2})}. \label{eq-ellip-var:AuxSpY}
\end{align}

\begin{proposition}\label{prop-ellip-var:Main}
	Let $\eta\in W^{1,\infty}(\R^d)$ and $F\in L^2(\R^{d+1}_-)$.
	\begin{itemize}
		\item(Existence and uniqueness) The variational problem \eqref{eq-ellip-var:VarForm} admits a unique solution $w\in H^{1,0}$, denoted as $\mathscr{R}(\eta)F$. Moreover, this solution satisfies
		\begin{equation}\label{eq-ellip-var:MainBasicReg}
			\|\nabla_{x,z}\mathscr{R}(\eta)F\|_{L^2(\R^{d+1}_-)} \le K(\|\eta\|_{W^{1,\infty}}) \|F\|_{L^2(\R^{d+1}_-)}.
		\end{equation}
		\item(Regularity) Consider real numbers $s,s_0,\sigma$ satisfying
		\begin{equation}\label{eq-ellip-var:MainCond}
			s\ge s_0,\ \mez \le \sigma \le s+\frac{1}{2}, \quad{with}\quad s_0>d/2,\ s_0\ge 1.
		\end{equation} 
		If we further assume that $\eta\in H^{s+1}(\R^d)$ and $\nabla_{x,z}F \in L^2(z_0,0;H^{\sigma-3/2})$ for some $z_0<0$, then the unique solution $w = \mathscr{R}(\eta)F$ satisfies, for all $z_0<z_0'<0$,
		\begin{equation}\label{eq-ellip-var:Main}
			\| \mathscr{R}(\eta)F \|_{\mathcal{X}^{\sigma-1}_{z_0'}} \le K(\|\eta\|_{H^{s_0+1}})\big( \|\eta\|_{H^{s+1}} \| F \|_{\mathcal{Y}^{-1/2}_{z_0}} +  \| F \|_{\mathcal{Y}^{\sigma-1}_{z_0}} \big) .
		\end{equation}
	\end{itemize}
\end{proposition}

\begin{definition}\label{def-ellip-var:VarSol}	
	With the notation $\mathscr{R}(\eta)$ as defined in the previous statement, 
	we shall say that $v$ is a variational solution to \eqref{eq-ellip:MainEq} if
	\begin{equation}\label{eq-ellip-var:VarSol}
		v = \Psi + \mathscr{R}(\eta)(A\nabla_{x,z}\Psi),
	\end{equation}
	where $\Psi$ is defined by \eqref{eq-ellip-var:DefPsi} and $A$ is defined in \eqref{eq-ellip-var:DefA}.
\end{definition}
Now observe that, since the extension $\Psi = e^{z|D_x|}\psi$ of $\psi$ satisfies \eqref{eq-ellip-var:EstiPsi}, implying \eqref{eq-ellip:Main} with $v$ replaced by $\Psi$, we see that Proposition~\ref{prop-ellip:Main} follows from Proposition~\ref{prop-ellip-var:Main}.

The rest of this section is dedicated to the proof of Proposition~\ref{prop-ellip-var:Main}. We divide the proof into two steps. We begin by noting that the first statement about the existence and uniqueness of a variational solution is straightforward. In the second step, we will examine the higher regularity of the solutions.

\subsection{Existence and uniqueness of a variational solution}
	Since the elements of the coefficient matrix $A$ (see \eqref{eq-ellip-var:DefA}) are all smooth functions of $\nabla\eta$, there exists a constant $C_0>0$ depending only on $\|\eta\|_{W^{1,\infty}}$ such that, for all $x\in\R^d$ and $\Xi\in\R^{d+1}$,
	\begin{equation*}
		C_0^{-1}|\Xi|^2 \le \Xi \cdot A(x)\Xi \le C_0|\Xi|^2.
	\end{equation*}
	On the other hand, for all $F\in L^2(\R^{d+1}_-)$, the right-hand side of the variational form \eqref{eq-ellip-var:VarForm} is a bounded linear form on $H^{1,0}$, whose operator norm equals $\|\nabla_{x,z}F\|_{L^2(\R^{d+1}_-)}$. As a result, the Riesz Theorem implies that there exists a unique solution to the variational problem \eqref{eq-ellip-var:VarForm} and that the estimate \eqref{eq-ellip-var:MainBasicReg} holds.

\subsection{Regularity of the solution}\label{subsect-ellip:reg}

In this part, we focus on the regularity part of Proposition~\ref{prop-ellip-var:Main}. The strategy is to prove \eqref{eq-ellip-var:Main} for $\sigma=1/2$ and then use an iteration in $\sigma$ to conclude the remaining case $1/2<\sigma\le s+1/2$. 

\begin{proof}[Proof of \eqref{eq-ellip-var:Main} when $\sigma=1/2$]
	From \eqref{eq-ellip-var:MainBasicReg}, 
	we immediately get
	$$
	\|\nabla_{x,z}w\|_{L^2(\R^{d+1}_-)} \le K(\|\eta\|_{W^{1,\infty}}) \|F\|_{L^2(\R^{d+1}_-)}.
	$$
	This also yields a bound for $\nabla_x\nabla_{x,z}v$ in $L^2_z(\R_-;H_x^{-1})$ and hence, 
	to prove~\eqref{eq-ellip-var:Main}, it remains only to estimate $\partial_z^2 w$. 
	To do so, we notice that the equivalent formulation \eqref{eq-ellip-var:DistriForm} provides an expression of $\partial_z^2w$ in the sense of distribution. Specifically, in terms of $\alpha$, $\beta$, and $\gamma$ defined in \eqref{eq-ellip:Coeff}, $\partial_z^2 w$ can be expressed as
	\begin{equation}\label{eq-ellip-wp:FormulaWzz}
	\begin{aligned}
		\partial_z^2 w &= -\frac{\beta}{\alpha}\cdot \nabla_x\partial_zw - \left(\frac{1}{\alpha}-1\right)\Delta_x w - \Delta_x w + \frac{\gamma}{\alpha} \partial_zw 
		\\
		&\quad - \left(\frac{1}{\alpha}-1\right)\nabla_{x,z}\cdot F - \nabla_{x,z}\cdot F,
	\end{aligned}
	\end{equation}
	where the coefficients $\beta/\alpha$ and 
	$\alpha^{-1} -1$ belong to $H^s$ (in light of Proposition \ref{prop-besov-bony:Paralin}) and $\gamma/\alpha$ belongs to $L^2$ due to Corollary~\ref{cor-besov-bony:ProdLaw}. Then the desired estimate of $\partial_z^2w$ follows from the product law in Sobolev spaces (see Corollary~\ref{cor-besov-bony:ProdLaw}). For example, $\gamma/\alpha\in L^\infty_zL^2_x$ and $\partial_zw\in L^2_zL^2_x$, so there product is in $L^2_zL^1_x\subset L^2_zH^{-1}_x$. Note that, when $d=1$, the condition $s\ge1$ is necessary in this step to ensure that the assumptions of Corollary~\ref{cor-besov-bony:ProdLaw} are verified.
\end{proof}

As for the remaining case $1/2<\sigma\le s+1/2$, we will apply the following iteration in $\sigma$.
\begin{proposition}\label{prop-ellip-reg:Iteration}
	Consider $s, s_0,\sigma$ and $\varepsilon$ such that 
	\begin{equation}\label{n201}
	s\ge s_0,\ \frac{1}{2} \le \sigma \le s+\frac{1}{2}-\varepsilon, \quad\text{with}\ s_0\ge 1,\ s_0>\frac{d}{2}, \ \text{and}\  0<\varepsilon<\min\left(s_0-\frac{d}{2},1\right).
	\end{equation}
	Let $\eta\in H^{s+1}_x$, $F\in \mathcal{Y}_{z_0}^{\sigma-1+\varepsilon}$, for some $z_0<0$, and let $w$ be the unique solution to the variational problem \eqref{eq-ellip-var:VarForm}. If there exists $z_0<z_0'<0$ such that
	\begin{equation}\label{eq-ellip-reg:IterAssump}
		\|w\|_{\mathcal{X}^{\sigma-1}_{z_0'}} \le K(\|\eta\|_{H^{s_0+1}})\big( \|\eta\|_{H^{s+1}} \| F \|_{\mathcal{Y}^{-1/2}_{z_0}} +  \| F \|_{\mathcal{Y}^{\sigma-1}_{z_0}} \big),
	\end{equation}
	then, for arbitrary $z_0'<z_0''<0$, there holds		
	\begin{equation}\label{eq-ellip-reg:IterConcl}
		\|w\|_{\mathcal{X}^{\sigma-1+\varepsilon}_{z_0''}} \le K(\|\eta\|_{H^{s_0+1}})\big( \|\eta\|_{H^{s+1}} \| F \|_{\mathcal{Y}^{-1/2}_{z_0}} +  \| F \|_{\mathcal{Y}^{\sigma-1+\varepsilon}_{z_0}} \big).
	\end{equation}
	Recall that the notations $\mathcal{X}^{r}_{z_0}$ and $\mathcal{Y}^{r}_{z_0}$ were introduced in \eqref{eq-ellip-var:AuxSpX} and \eqref{eq-ellip-var:AuxSpY}.
\end{proposition}

The argument follows that of \cite{ABZ2011,ABZ2014} (see also~\cite{MR3649370,MR3730012,MR4400907,shao2023toolboxparadifferentialcalculuscompact}). We divide the proof of this iteration into three steps: $(i)$ paralinearization of the equation \eqref{eq-ellip-var:DistriForm}, $(ii)$ factorization of the paralinearized 
elliptic operator, and $(iii)$ reduction to a 
parabolic problem. The first rigorous statement is as follows.

\begin{lemma}[Paralinearization of \eqref{eq-ellip-var:DistriForm}]\label{lem-ellip-reg:Paralin}
    	Under the assumptions of Proposition \ref{prop-ellip-reg:Iteration}, the solution $w$ to the variational problem \eqref{eq-ellip-var:VarForm} satisfies an equation of the form
	\begin{equation}\label{eq-ellip-reg:Paralin}
		\left( T_\alpha \partial_z^2 + T_\beta \cdot \nabla_x\partial_z -\Delta_x \right) w = f,
	\end{equation}
	where the right-hand side $f$ satisfies,
	\begin{equation}\label{eq-ellip-reg:EstiF}
	    \|f\|_{L^2(z_0',0;H^{\sigma-3/2+\varepsilon})}\le K(\|\eta\|_{H^{s_0+1}})\big( \|\eta\|_{H^{s+1}} \| F \|_{\mathcal{Y}^{-1/2}_{z_0}} +  \| F \|_{\mathcal{Y}^{\sigma-1+\varepsilon}_{z_0}} \big).
	\end{equation}
\end{lemma}
\begin{proof}
	Clearly, the right-hand side $f$ equals
	\begin{equation*}
	    f = f_1 + f_2 + \nabla_{x,z}\cdot F,
	\end{equation*}
	with
	\begin{align*}
	    f_1 =& -\left( \alpha\partial_z^2w - T_{\alpha}\partial_z^2w \right) - \left( \beta\cdot\nabla_x\partial_zw - T_{\beta}\cdot\nabla_x\partial_zw \right), \\
	    f_2 =& - \gamma \partial_zw = -\left( \gamma \partial_zw - T_{\gamma}\partial_zw \right) - T_{\gamma}\partial_zw
	\end{align*}
	From $\eta\in H^{s+1}$ ($s>d/2$) and Proposition \ref{prop-besov-bony:Paralin}, we deduce that
	\begin{equation*}
		\| \alpha-1 \|_{H^s} + \|\beta\|_{H^s} + \|\gamma\|_{H^{s-1}} \le K(\|\eta\|_{W^{1,\infty}})\|\eta\|_{H^{s+1}}.
	\end{equation*}
	Then the desired estimate \eqref{eq-ellip-reg:EstiF} follows from Proposition~\ref{lemPa} and~\ref{prop-besov-bony:EstiPara}, along with the assumption~\eqref{eq-ellip-reg:IterAssump}. Here we also used the fact that 
	\begin{equation*}
	    \|\nabla_{x,z}w\|_{L^2(z_0'',0;L^2)} + \|\nabla_{x,z}^2w\|_{L^2(z_0'',0;H^{-1})} \le K(\|\eta\|_{H^{s_0+1}}) \| F \|_{\mathcal{Y}^{-1/2}_{z_0}},
	\end{equation*}
	which is no more than the inequality \eqref{eq-ellip-var:Main} with $\sigma=1/2$ and $s=s_0$ (already proved above).
\end{proof}

\begin{lemma}[Factorization of the paralinearized elliptic operator]\label{lem-ellip-reg:Facto}
	Consider two real numbers $s_0$ and $\varepsilon$ such that
	$s_0>d/2$ and $0<\varepsilon<\min(s-d/2,1)$. Then, for any $\eta\in H^{s_0+1}$, 
	the following factorization holds:
	\begin{equation}\label{eq-ellip-reg:Facto}
		T_\alpha \partial_z^2 + T_\beta \cdot \nabla_x\partial_z + T_{-|\xi|^2} = T_\alpha (\partial_z + T_a)(\partial_z - T_A) + R_1|D_x| + R_2\partial_z,
	\end{equation}
	where $a=a(x,\xi)$ and $A=A(x,\xi)$ are two symbols in $\Gamma^1_{\varepsilon}$ (see Definition \ref{def-paradiff:SymbClass}) defined by
	\begin{equation}\label{eq-ellip-reg:FactoSymb}
		\begin{aligned}
			a(x, \xi) =& \frac{1}{2\alpha}\left( \sqrt{ 4\alpha \la \xi \ra^2 -  (\beta \cdot \xi)^2} + i \beta\cdot \xi \right),\\ 
			A(x, \xi) =& \frac{1}{2\alpha}\left( \sqrt{ 4\alpha \la \xi \ra^2 -  (\beta \cdot \xi)^2} - i \beta\cdot \xi \right),
		\end{aligned}
	\end{equation}
	and the reminders $R_1$, $R_2$ are of order $(1-\varepsilon)$, which means that, 
	for all $r\in\R$, there holds
	\begin{equation}\label{eq-ellip-reg:FactoRem}
		\|R_1\|_{\mathcal{L}(H^r;H^{r-1+\varepsilon})} + \|R_2\|_{\mathcal{L}(H^r;H^{r-1+\varepsilon})} \le K_r(\|\eta\|_{H^{s_0+1}}).
	\end{equation}
\end{lemma}
\begin{proof}
    By Sobolev embedding, we have $\alpha-1,\beta\in H^{s_0}\subset W^{\varepsilon,\infty}$, which guarantees that the symbols $a,A$ defined by \eqref{eq-ellip-reg:FactoSymb} belong to $\Gamma^{1}_{\varepsilon}$. To obtain the factorization \eqref{eq-ellip-reg:Facto}, we observe that they verify
	\begin{equation*}
		aA=\frac{\la \xi\ra^2}{\alpha},\ \ a - A = \frac{i\beta\cdot\xi}{\alpha},
	\end{equation*}
	Therefore, according to Proposition~\ref{prop-paradiff:SymbCal} applied with $\rho=\varepsilon$, we have that
	\begin{align*}
		& R_1 := T_{\alpha} T_a T_{A} |D_x|^{-1} + \Delta_x |D_x|^{-1}  \quad \text{is of order } 1-\varepsilon,\\
		& R_2 := -T_{\alpha} T_{a-A} + T_{\beta}\cdot \nabla_x \quad \text{is of order } 1-\varepsilon.
	\end{align*}
	Now, since
	\begin{align*}
		T_\alpha (\partial_z + T_a)(\partial_z - T_A) =& T_\alpha\partial_z^2 + T_\alpha T_{a-A} \partial_z -T_{\alpha} T_a T_{A|\xi|^{-1}} \\
		=& T_\alpha \partial_z^2 + T_\beta \cdot \nabla_x\partial_z + T_{-|\xi|^2} - R_1 |D_x| - R_2 \partial_z,
	\end{align*}
	as can be verified by an elementary calculation, 
	we get the desired equality \eqref{eq-ellip-reg:Facto}.
\end{proof}

The next step is to combine the previous two lemmas.
\begin{lemma}\label{lem-ellip-reg:Parabolic}
	Under the hypothesis of Proposition \ref{prop-ellip-reg:Iteration}, the solution $v$ to the variational problem \eqref{eq-ellip-var:VarForm} satisfies
	\begin{equation}\label{eq-ellip-reg:Parabolic}
		(\partial_z + T_a)(\partial_z - T_A) w = g,
	\end{equation}
	where the remainder $g$ equals
	\begin{equation*}
		g := (\Id-T_{\alpha^{-1}}T_\alpha) (\partial_z + T_a)(\partial_z - T_A) w + T_{\alpha^{-1}}(R_1|D_x|w + R_2\partial_zw) - T_{\alpha^{-1}}f,
	\end{equation*}
	and $f$ is defined in \eqref{eq-ellip-reg:Paralin}. Moreover, $g$ satisfies,
	\begin{equation}\label{eq-ellip-reg:EstiG}
		\|g\|_{L^2(z_0',0;H^{\sigma-3/2+\varepsilon})} \le K(\|\eta\|_{H^{s_0+1}})\big( \|\eta\|_{H^{s+1}} \| F \|_{\mathcal{Y}^{-1/2}_{z_0}} +  \| F \|_{\mathcal{Y}^{\sigma-1+\varepsilon}_{z_0}} \big).
	\end{equation}
\end{lemma}
\begin{proof}[Proof of Lemma \ref{lem-ellip-reg:Parabolic}]
	The formula \eqref{eq-ellip-reg:Parabolic} can be directly deduced from \eqref{eq-ellip-reg:Paralin} and \eqref{eq-ellip-reg:Facto}. To prove the estimate \eqref{eq-ellip-reg:EstiG}, we first recall that the condition $\eta\in H^{s_0+1} \subset W^{\epsilon,\infty}$ guarantees that $\alpha, \alpha^{-1} \in \Gamma^0_\varepsilon$, $a,A\in \Gamma^1_\varepsilon$. As a consequence, an application of Proposition \ref{prop-paradiff:Bd} gives
	\begin{align*}
		&\hspace{-2em}\|T_{\alpha^{-1}}(R_1|D_x|w + R_2\partial_zw) - T_{\alpha^{-1}}f\|_{L^2(z_0',0;H^{\sigma-3/2+\varepsilon})}  \\
		&\le K(\|\eta\|_{H^{s_0+1}}) \|R_1|D_x|w + R_2\partial_zw - f\|_{L^2(z_0',0;H^{\sigma-3/2+\varepsilon})} \\
		&\le K(\|\eta\|_{H^{s_0+1}}) \left( \|\partial_z w + |D_x|w\|_{L^2(z_0',0;H^{\sigma-1/2})} + \|f\|_{L^2(z_0',0;H^{\sigma-3/2+\varepsilon})} \right) \\
		&\le  K(\|\eta\|_{H^{s_0+1}})\big( \|\eta\|_{H^{s+1}} \| F \|_{\mathcal{Y}^{-1/2}_{z_0}} +  \| F \|_{\mathcal{Y}^{\sigma-1+\varepsilon}_{z_0}} \big),
	\end{align*}
	where the last inequality follows from the assumption \eqref{eq-ellip-reg:IterAssump} and Lemma~\ref{lem-ellip-reg:Paralin}.
	Moreover, the symbolic calculus \eqref{eq-paradiff:SymbCalComp} implies that $(\Id-T_{\alpha^{-1}}T_\alpha)$ is of order $(-\varepsilon)$, i.e.
	\begin{align*}
		&\hspace{-2em}\|(\Id-T_{\alpha^{-1}}T_\alpha) (\partial_z + T_a)(\partial_z - T_A) w\|_{L^2(z_0',0;H^{\sigma-3/2+\varepsilon})}  \\
		&\le K(\|\eta\|_{H^{s_0+1}}) \|(\partial_z + T_a)(\partial_z - T_A) w\|_{L^2(z_0',0;H^{\sigma-3/2})} \\
		&\le K(\|\eta\|_{H^{s_0+1}})\big( \|\eta\|_{H^{s+1}} \| F \|_{\mathcal{Y}^{-1/2}_{z_0}} +  \| F \|_{\mathcal{Y}^{\sigma-1+\varepsilon}_{z_0}} \big).
	\end{align*}
	Note that, in the last inequality, we applied the bound \eqref{eq-ellip-reg:IterAssump} of $\nabla_{x,z}^2w$ and the boundedness of $T_a,T_A$ (see Proposition \ref{prop-paradiff:Bd}).
\end{proof}

We are now in a position to conclude Proposition~\ref{prop-ellip-reg:Iteration}. Our goal is to deduce \eqref{eq-ellip-reg:IterConcl} from Lemma~\ref{lem-ellip-reg:Parabolic}. To do so, let us fix a bump function $\chi\in C^\infty_0(\R)$ supported in $(z_0',+\infty)$ and equals to $1$ on $[z_0'',0]$ (remembering  that $z_0'<z_0''<0$). Then it is clear that
\begin{equation*}
	(\partial_z + T_a)\left( \chi(z)(\partial_z-T_A)w \right) = \chi(z)g + \chi'(z)(\partial_z-T_A)w,
\end{equation*}
where the right-hand side belongs to $L^2(z_0',0;H^{\sigma-3/2+\varepsilon})$ due to \eqref{eq-ellip-reg:EstiG} and the iteration assumption \eqref{eq-ellip-reg:IterAssump}. 
The proof then relies on two observations. Firstly, note that $a\in \Gamma^1_\varepsilon$ (recall that $\eta\in H^{s_0+1}\subset W^{\varepsilon,\infty}$) satisfies $\Real a \ge c|\xi|$ for some constant $c>0$ 
depending only on $\|\nabla\eta\|_{L^\infty_x}$, and secondly, observe that $\chi(z)(\partial_z-T_A)w$ vanishes at $z=z_0$. Thus, by applying Proposition \ref{prop-paradiff:Parabolic}, we have 
$$
\chi(z)(\partial_z-T_A)w \in L^2(z_0'',0;H^{\sigma-1/2+\varepsilon}),
$$
and, due to our choice of $\chi$, $(\partial_z-T_A)w$ lies in $L^2(z_0'',0;H^{\sigma-1/2+\varepsilon})$. Now, since $\Real A \ge c|\xi|$ and $w|_{z=0}=0$, we are able to apply Proposition \ref{prop-paradiff:Parabolic} again for $(\partial_z-T_A)w$ and conclude that
$$
w\in L^2(z_0'',0;H^{\sigma+1/2+\varepsilon}),\quad (\partial_z-T_A)w \in L^2(z_0'',0;H_x^{\sigma-1/2+\varepsilon}),
$$
which yields the bounds for $\nabla_{x,z}w$ in \eqref{eq-ellip-reg:IterConcl}, so as the bounds for $\nabla_{x,z}\nabla w$. As for the term $\partial_z^2 w$, we proceed as above, 
by using formula \eqref{eq-ellip-wp:FormulaWzz} together with Proposition~\ref{lemPa} and~\ref{prop-besov-bony:EstiPara}.

\section{Shape derivative}\label{app:ShpDer}

The aim of this section is to prove Theorem~\ref{thm-reg-shpder:Main}. As this result is of independent interest, we will prove a more general version that holds in any dimension.

\begin{theorem}\label{thm-shp:ShpDer}
    Consider an integer $d\ge1$ and two real numbers $s,\sigma$ satisfying
    \begin{equation*}
    	s>\frac{d}{2},\quad s\ge 1,\quad \text{and} \quad \frac{1}{2} \le \sigma \le s+\frac{1}{2}.
    \end{equation*}
    For all $\eta,\delta\eta\in H^{s+1}(\R^d)$ and $\psi\in H^\sigma(\R^d)$, the following limit exists in $H^{\sigma-1}(\R^d)$:
    \begin{equation}\label{eq-shp:ShpDer}
        \lim_{\varepsilon\rightarrow 0}\frac{ \mathcal{G}(\eta+\varepsilon\delta\eta)\psi - \mathcal{G}(\eta)\psi }{\varepsilon} 
        = -\mathcal{G}(\eta)\left(\delta\eta \mathcal{B}(\eta)\psi \right) - \nabla_x\cdot\left(\delta\eta\mathcal{V}(\eta)\psi \right).
    \end{equation}
    Note that the right-hand side is defined in the sense of distribution $\mathcal{D}'(\R^d)$ as follows: for all $\varphi\in C^\infty_0(\R^d)$,
    \begin{equation}\label{eq-shp:ShpDerWeak}
    	\begin{aligned}
    		&\langle -\mathcal{G}(\eta)\left(\delta\eta \mathcal{B}(\eta)\psi \right) - \nabla_x\cdot\left(\delta\eta\mathcal{V}(\eta)\psi \right), \varphi \rangle_{\mathcal{D}'(\R^d)\times\mathcal{D}(\R^d)} \\
    		&\hspace{6em}= \int_{\R^d}\left[ -(\delta\eta\G(\eta)\varphi) \mathcal{B}(\eta)\psi + (\delta\eta\nabla_x\varphi)\cdot\mathcal{V}(\eta)\psi \right] \dx,
    	\end{aligned}
    \end{equation}
    where the integral in the right-hand side makes sense due to Corollary~\ref{cor-besov-bony:ProdLaw}. Moreover, it satisfies
    \begin{equation}\label{eq-shp:ShpDerEsti}
        \left\| -\mathcal{G}(\eta)\left(\delta\eta \mathcal{B}(\eta)\psi \right) - \nabla_x\cdot\left(\delta\eta\mathcal{V}(\eta)\psi \right) \right\|_{H^{\sigma-1}} \le K(\|\eta\|_{H^{s+1}}) \|\delta\eta\|_{H^{s+1}} \|\psi\|_{H^{\sigma}}.
    \end{equation}
\end{theorem}
\begin{remark}
	In order to establish the differentiability of~\( \mathcal{G}(\eta) \psi \) with respect to~\( \eta \) (see Theorem~\ref{thm-shp:ShpDer} and Proposition~\ref{prop-shp:Main} below), the condition \( s \geq 1 \) can be replaced by a weaker condition $s+\sigma\ge \frac{3}{2}$. To do so, in the flattening of \( y = \eta \), one must apply a regularizing diffeomorphism instead of the trivial one \( z = y - \eta \) used in this article (see \cite{lannes2013water} for the definition of such transformations). Then, by repeating exactly the same arguments, Theorem~\ref{thm-shp:ShpDer} and Proposition~\ref{prop-shp:Main} below can be proved without assuming \( s \geq 1 \). Meanwhile, the condition $s+\sigma\ge \frac{3}{2}$ is required so that the right-hand side of formula \eqref{eq-shp:ShpDer} is well-defined, and is automatically satisfied in the setting \( s \geq 1 \) and \( \sigma \geq \frac{1}{2} \).
\end{remark}

In Section 3.7.3 of \cite{lannes2013water}, Lannes proved this result for \( \psi \) satisfying \( |D_x|^{1/2} \psi \in H^1(\mathbb{R}^d) \). We will consider a more general setting that allows for the case where \( \psi \in H^{1/2} \). Our motivation for doing so is to cover the case \( (\eta, \psi) \in H^2 \times H^1 \), which corresponds to the regularity in Theorem~\ref{thm:main}. The strategy is to approximate \( (\eta, \delta\eta, \psi) \) by a sequence of smooth functions \( (\eta_k, \delta\eta_k, \psi_k)_{k\in\mathbb{N}} \) and then pass to the limit as \( k \to +\infty \). In this process, uniform convergence is required on the left-hand side of \eqref{eq-shp:ShpDer} to ensure that the order of limits, \( \varepsilon \to 0 \) and \( k \to +\infty \), can be interchanged. To achieve this, we will apply the regularity results established in the previous section to obtain a precise estimate for the difference between \( \mathcal{G}(\eta+\varepsilon\delta\eta)\psi \) and \( \mathcal{G}(\eta)\psi \).

\begin{proposition}\label{prop-shp:Esti}
	Assume that $\varepsilon$, $s$, and $\sigma$ are 
	three real numbers satisfying
	\begin{equation}\label{eq-shp:EstiCond}
		0<\varepsilon\le 1,\quad s>d/2,\quad s\ge 1\quad\text{and}\quad \mez \le \sigma \le s+\frac{1}{2}.
	\end{equation}
	Given $\eta,\delta\eta\in H^{s+1}$ and $\psi\in H^{\sigma}$, we denote by $v^\varepsilon$ and $v$ the variational solution to \eqref{eq-ellip:MainEq} (see Definition~\ref{def-ellip-var:VarSol}) associated to $(\eta+\varepsilon\delta\eta,\psi)$ and $(\eta,\psi)$, respectively. Namely,
	\begin{equation}\label{eq-shp:DefVVeps}
		v^\varepsilon := \Psi + \mathscr{R}(\eta+\varepsilon\delta\eta)(A^\varepsilon\nabla_{x,z}\Psi), \quad v := \Psi + \mathscr{R}(\eta)(A\nabla_{x,z}\Psi),
	\end{equation}
	where $\Psi = e^{z|D_x|}\psi$ and $A^{\varepsilon}$ is given by
	\begin{equation}\label{eq-shp:DefAeps}
		A^\varepsilon = \left(\begin{array}{cc}
			1 & \nabla_x(\eta + \varepsilon\delta\eta) \\
			\nabla_x(\eta + \varepsilon\delta\eta)^T & 1+|\nabla_x(\eta + \varepsilon\delta\eta)|^2
		\end{array}\right).
	\end{equation}
	Then the following estimate holds true,
	\begin{equation}\label{eq-shp:EstiDiff}
		\|v^\varepsilon-v\|_{\mathcal{X}^{\sigma-1}_{-1}} \le \varepsilon K(\|\eta\|_{H^{s+1}}+\|\delta\eta\|_{H^{s+1}}) \|\delta\eta\|_{H^{s+1}} \|\psi\|_{H^\sigma}.
	\end{equation}
	Recall that the notation $\mathcal{X}^{\sigma-1}_{-1}$ is defined in \eqref{eq-ellip-var:AuxSpX}.
	
	Furthermore, let $\delta v$ be defined as
	\begin{equation}\label{eq-shp:DefDeltav}
		\delta v := \mathscr{R}(\eta)(\delta A\nabla_{x,z}v)
	\end{equation}
	with
	\begin{equation}\label{eq-shp:DefDeltaA}
		\delta A = \left(
		\begin{array}{cc}
			0 & \nabla_x\delta\eta \\
			\nabla_x\delta\eta^T & 2\nabla_x\eta\cdot\nabla_x\delta\eta
		\end{array}
		\right).
	\end{equation}
	The following estimate also holds,
	\begin{equation}\label{eq-shp:EstiDeri}
		\left\|\frac{v^\varepsilon-v}{\varepsilon} - \delta v\right\|_{\mathcal{X}^{\sigma-1}_{-1}} \le \varepsilon K(\|\eta\|_{H^{s+1}}+\|\delta\eta\|_{H^{s+1}}) \|\delta\eta\|_{H^{s+1}} \|\psi\|_{H^\sigma}.
	\end{equation}
\end{proposition}
\begin{proof}[Proof of Proposition~\ref{prop-shp:Esti}]
The desired estimates \eqref{eq-shp:EstiDiff} and \eqref{eq-shp:EstiDeri} can be obtained by combining Proposition~\ref{prop-ellip-var:Main} and the following identity, which can be proved directly from Definition~\ref{def-ellip-var:VarSol}.
\begin{equation}\label{eq-shp:DiffFormu}
	\mathscr{R}(\eta+\varepsilon\delta\eta)(A^\varepsilon\nabla_{x,z}\Psi) - \mathscr{R}(\eta)(A\nabla_{x,z}\Psi) = \mathscr{R}(\eta)\left( (A^\varepsilon-A)\nabla_{x,z}v^\varepsilon \right).
\end{equation}
	Through the identity \eqref{eq-shp:DiffFormu} above and the inequality \eqref{eq-ellip-var:Main}, we have
	\begin{equation*}
	\begin{aligned}
	\|v^\varepsilon-v\|_{\mathcal{X}^{\sigma-1}_{-1}} 
	&= \|\mathscr{R}(\eta)\left( (A^\varepsilon-A)\nabla_{x,z}v^\varepsilon \right)\|_{\mathcal{X}^{\sigma-1}_{-1}} \\
	&\le K(\|\eta\|_{H^{s+1}}) \| (A^\varepsilon-A)\nabla_{x,z}v^\varepsilon \|_{\mathcal{Y}^{\sigma-1}_{-2}},
	\end{aligned}
	\end{equation*}
	where $\mathcal{Y}^{\sigma-1}_{-2}$ is defined in \eqref{eq-ellip-var:AuxSpY}. Then we apply the product law in Sobolev spaces (Corollary~\ref{cor-besov-bony:ProdLaw}) and the estimate \eqref{eq-ellip:Main} to obtain
	\begin{align*}
		\| (A^\varepsilon-A)\nabla_{x,z}v^\varepsilon \|_{\mathcal{Y}^{\sigma-1}_{-2}} \lesssim& \|A^\varepsilon-A\|_{H^s} \| \nabla_{x,z}v^\varepsilon \|_{\mathcal{Y}^{\sigma-1}_{-2}} \\
		\lesssim& \|A^\varepsilon-A\|_{H^s} \| v^\varepsilon \|_{\mathcal{X}^{\sigma-1}_{-2}} \\
		\le& K(\|\eta + \varepsilon\delta\eta\|_{H^{s+1}}) \|A^\varepsilon-A\|_{H^s} \|\psi\|_{H^\sigma}.
	\end{align*}
	Then the estimate \eqref{eq-shp:EstiDiff} follows from the following inequality
	\begin{equation*}
		\|A^\varepsilon-A\|_{H^s} \lesssim \varepsilon K(\|\eta\|_{H^{s+1}}+\varepsilon\|\delta\eta\|_{H^{s+1}}) \|\delta\eta\|_{H^{s+1}},
	\end{equation*}
	which in turn easily follows from the definitions \eqref{eq-shp:DefAeps} and \eqref{eq-ellip-var:DefA} of $A^\varepsilon$ and $A$.
	
	The proof of \eqref{eq-shp:EstiDeri} is similar. We apply again the identity \eqref{eq-shp:DiffFormu} to compute
	\begin{align*}
		\frac{v^\varepsilon-v}{\varepsilon} - \delta v =& \mathscr{R}(\eta)\left( \frac{A^\varepsilon-A}{\varepsilon}\nabla_{x,z}v^\varepsilon \right) - \mathscr{R}(\eta)(\delta A\nabla_{x,z}v) \\
		=& \mathscr{R}(\eta)\left[\left(\frac{A^\varepsilon-A}{\varepsilon}-\delta A\right) \nabla_{x,z}v^\varepsilon + \delta A \nabla_{x,z}(v^\varepsilon-v)\right] \\
		=& \mathscr{R}(\eta)\left[\varepsilon|\nabla_x\delta\eta|^2 \partial_z v^\varepsilon + \delta A \nabla_{x,z}(v^\varepsilon-v)\right].
	\end{align*}
	An application of \eqref{eq-ellip-var:Main}, together with Corollary~\ref{cor-besov-bony:ProdLaw}, gives 
	\begin{align*}
		\left\|\frac{v^\varepsilon-v}{\varepsilon} - \delta v\right\|_{\mathcal{X}^{\sigma-1}_{-1}} &\le K(\|\eta\|_{H^{s+1}}) \left\| \varepsilon|\nabla_x\delta\eta|^2 \partial_z v^\varepsilon + \delta A \nabla_{x,z}(v^\varepsilon-v) \right\|_{\mathcal{Y}^{\sigma-1}_{-2}}  \\
		&\hspace{-4em}\le K(\|\eta\|_{H^{s+1}}+\|\delta\eta\|_{H^{s+1}}) \left( \varepsilon \|\delta\eta\|_{H^{s+1}} \left\|\partial_z v^\varepsilon\right\|_{\mathcal{Y}^{\sigma-1}_{-2}} + \left\| \nabla_{x,z}(v^\varepsilon-v) \right\|_{\mathcal{Y}^{\sigma-1}_{-2}}\right) \\
		&\hspace{-4em}\le K(\|\eta\|_{H^{s+1}}+\|\delta\eta\|_{H^{s+1}}) \left( \varepsilon \|\delta\eta\|_{H^{s+1}} \left\|v^\varepsilon\right\|_{\mathcal{X}^{\sigma-1}_{-2}} + \left\| v^\varepsilon-v \right\|_{\mathcal{X}^{\sigma-1}_{-2}}\right).
	\end{align*}
	Thanks to estimates \eqref{eq-ellip:Main} and \eqref{eq-shp:EstiDiff}, the right-hand side is bounded by 
	\begin{equation*}
		\varepsilon K(\|\eta\|_{H^{s+1}}+\|\delta\eta\|_{H^{s+1}}) \|\delta\eta\|_{H^{s+1}} \|\psi\|_{H^\sigma},
	\end{equation*}
	which completes the proof of \eqref{eq-shp:EstiDeri}.
\end{proof}

Through a classical interpolation result (see Theorem~$3.1$ in~\cite{Lions-Magenes-Vol1}), the estimates \eqref{eq-shp:EstiDiff} and \eqref{eq-shp:EstiDeri} imply
\begin{align*}
	\|v^\varepsilon|_{z=0}-v|_{z=0}\|_{H^{\sigma-1}} \le& \varepsilon K(\|\eta\|_{H^{s+1}}+\|\delta\eta\|_{H^{s+1}}) \|\delta\eta\|_{H^{s+1}}\|\psi\|_{H^\sigma}, \\
	\left\|\left.\frac{v^\varepsilon-v}{\varepsilon}\right|_{z=0} - \delta v|_{z=0}\right\|_{\mathcal{X}^{\sigma-1}_{-1}} \le& \varepsilon K(\|\eta\|_{H^{s+1}}+\|\delta\eta\|_{H^{s+1}}) \|\delta\eta\|_{H^{s+1}}\|\psi\|_{H^\sigma}.
\end{align*}
Consequently, the linear operators $\G(\eta)\psi$, $\B(\eta)\psi$, and $\V(\eta)\psi$, defined by
\begin{align*}
	&\G(\eta)\psi = (1+|\nabla_x\eta|^2)\partial_zv|_{z=0} - \nabla_x\eta\cdot\nabla v|_{z=0}, \quad \B(\eta)\psi = \partial_z|_{z=0}, \\
	&\V(\eta)\psi = \nabla_xv|_{z=0} - \nabla_x\eta \partial_zv|_{z=0}.
\end{align*}
satisfies the properties listed in the following proposition.
\begin{proposition}\label{prop-shp:Main}
	Consider two real numbers $s,\sigma$ satisfying
	\begin{equation*}
		s>d/2,\quad s\ge 1\quad\text{and}\quad \mez \le \sigma \le s+\frac{1}{2}.
	\end{equation*}
	For all $\eta,\delta\eta\in H^{s+1}(\R^d)$, $\psi\in H^\sigma(\R^d)$, and $0<\varepsilon\le 1$, the following estimates hold true:
	\begin{equation}\label{eq-shp:MainEstiDiff}
		\| \mathcal{A}(\eta+\varepsilon\delta\eta)\psi - \mathcal{A}(\eta)\psi\|_{H^{\sigma-1}} \le \varepsilon K(\|\eta\|_{H^{s+1}}+\|\delta\eta\|_{H^{s+1}}) \|\delta\eta\|_{H^{s+1}} \|\psi\|_{H^\sigma},
	\end{equation}
	where $\mathcal{A}$ is taken among $\G$, $\B$, and $\V$. Moreover, there exists $\mathcal{A}_\eta(\eta;\psi,\delta\eta) \in H^{\sigma-1}(\R^d)$ bilinear in $(\psi,\delta\eta)$ such that, for all $0<\varepsilon\le 1$,
	\begin{equation}\label{eq-shp:MainEstiDeri}
		\begin{aligned}
			&\left\| \frac{\mathcal{A}(\eta+\varepsilon\delta\eta)\psi - \mathcal{A}(\eta)\psi}{\varepsilon} - \mathcal{A}_\eta(\eta;\psi,\delta\eta) \right\|_{H^{\sigma-1}} \\
			&\hspace{8em}\le \varepsilon K(\|\eta\|_{H^{s+1}}+\|\delta\eta\|_{H^{s+1}}) \|\delta\eta\|_{H^{s+1}} \|\psi\|_{H^\sigma}.
		\end{aligned}
	\end{equation}
	Consequently, $\mathcal{A}_\eta(\eta;\psi,\delta\eta)$ satisfies the following estimate:
	\begin{equation}\label{eq-shp:MainEstiLim}
	    \|\mathcal{A}_\eta(\eta;\psi,\delta\eta)\|_{H^{\sigma-1}} \le K(\|\eta\|_{H^{s+1}}) \|\delta\eta\|_{H^{s+1}} \|\psi\|_{H^{\sigma}}.
	\end{equation}
\end{proposition}

Now we are in a position to finish the proof of Theorem~\ref{thm-shp:ShpDer}.
\begin{proof}[Proof of Theorem~\ref{thm-shp:ShpDer}]
	Consider sequences $\eta_k,\delta\eta_k,\psi_k\in H^\infty := \cap_{s\in\R}H^s$ converging to $\eta,\delta\eta,\psi$ in $H^{s+1},H^{s+1},H^{\sigma}$, respectively. From \eqref{eq-shp:MainEstiDeri} and Section 3.7.3 of \cite{lannes2013water}, one can deduce that the limit
	\begin{align*}
		\frac{ \mathcal{G}(\eta_k+\varepsilon\delta\eta_k)\psi_k - \mathcal{G}(\eta_k)\psi_k }{\varepsilon} 
		\overset{\varepsilon\to0}{\rightarrow}& \mathcal{A}_\eta(\eta_k;\delta\eta_k,\psi_k) \\
		=&  -\mathcal{G}(\eta)\left(\delta\eta_k \mathcal{B}(\eta_k)\psi_k \right) - \nabla_x\cdot\left(\delta\eta_k\mathcal{V}(\eta_k)\psi_k \right)
	\end{align*}
	exists in $H^{\sigma-1}$ uniformly for $k\in\N$, since $(\eta_k,\delta\eta_k,\psi_k)$'s are uniformly bounded. Besides, the estimate \eqref{eq-shp:MainEstiDiff}, together with the boundedness of Dirichlet-to-Neumann operator $\G(\eta)$ (see Proposition~\ref{prop-reg-bas:Bd} and also Theorem 3.12 of \cite{ABZ2014}), implies that, for all $0<\varepsilon\le 1$, the limit 
	\begin{equation*}
		\lim_{k\to+\infty}\frac{ \mathcal{G}(\eta_k+\varepsilon\delta\eta_k)\psi_k - \mathcal{G}(\eta_k)\psi_k }{\varepsilon} = \frac{ \mathcal{G}(\eta+\varepsilon\delta\eta)\psi - \mathcal{G}(\eta)\psi }{\varepsilon} 
	\end{equation*}
	exists in $H^{\sigma-1}$. Then, by uniform limit theorem, 
	\begin{equation}\label{eq-shp:RegLim}
		-\mathcal{G}(\eta_k)\left(\delta\eta_k \mathcal{B}(\eta_k)\psi_k \right) - \nabla_x\cdot\left(\delta\eta_k\mathcal{V}(\eta_k)\psi_k \right)
	\end{equation}
	converges in $H^{\sigma-1}$ and its limit coincides with the limit on the left-hand side of \eqref{eq-shp:ShpDer}, whose existence is guaranteed by \eqref{eq-shp:MainEstiDeri}. Therefore, it suffices to show that \eqref{eq-shp:RegLim} tends to the right-hand side of \eqref{eq-shp:ShpDer} in the sense of distributions $\mathcal{D}'(\R^d)$.
	
	Consider a test function $\varphi\in C^\infty_0(\R^d)$. From \eqref{eq-shp:ShpDerWeak}, we have
	\begin{align*}
	\big\langle\!-\mathcal{G}(\eta_k)&\big(\delta\eta_k \mathcal{B}(\eta_k)\psi_k \big) - \nabla_x\cdot\big(\delta\eta_k\mathcal{V}(\eta_k)\psi_k \big), \varphi \big\rangle_{\mathcal{D}'(\R^d)\times\mathcal{D}(\R^d)} \\
		&= \int_{\R^d}\left[ -(\delta\eta_k\G(\eta_k)\varphi) \mathcal{B}(\eta_k)\psi_k + (\delta\eta_k\nabla_x\varphi)\cdot\mathcal{V}(\eta_k)\psi_k \right] \dx \\
		&\overset{k\to+\infty}{\longrightarrow} \int_{\R^d}\left[ -(\delta\eta\G(\eta)\varphi) \mathcal{B}(\eta)\psi + (\delta\eta\nabla_x\varphi)\cdot\mathcal{V}(\eta)\psi \right] \dx \\
		&=\langle -\mathcal{G}(\eta)\left(\delta\eta \mathcal{B}(\eta)\psi \right) - \nabla_x\cdot\left(\delta\eta\mathcal{V}(\eta)\psi \right), \varphi \rangle_{\mathcal{D}'(\R^d)\times\mathcal{D}(\R^d)}.
	\end{align*}
	To prove the limit, we first observe that
	$$
	\lim_{k\to+\infty} \left(\|\B(\eta_k)\psi_k - \B(\eta)\psi\|_{H^{\sigma-1}} + \|\V(\eta_k)\psi_k - \V(\eta)\psi\|_{H^{\sigma-1}}\right) = 0,
	$$
	thanks to the estimate \eqref{eq-shp:MainEstiDiff} and the boundedness of $\B(\eta)$ and $\V(\eta)$, which can be deduced directly from Proposition~\ref{prop-ellip:Main} and Corollary~\ref{cor-besov-bony:ProdLaw}. Thus, it remains to check $\delta\eta_k\G(\eta_k)\varphi$ converges to $\delta\eta\G(\eta)\varphi$ in $H^{1-\sigma}$. Through the product law in Sobolev space (see Corollary~\ref{cor-besov-bony:ProdLaw}), it suffices to show that 
	$$
	\lim_{k\to+\infty} \left(\| \delta\eta_k - \delta\eta \|_{H^{s+1}} + \|\G(\eta_k)\varphi - \G(\eta)\varphi\|_{H^{1-\sigma}}\right) = 0,
	$$
	which is no more than a consequence of \eqref{eq-shp:MainEstiDiff} (applied with $\sigma=s+1/2$) and the fact $s+\sigma\ge3/2$ (recall that $s\ge 1$ and $\sigma\ge 1/2$). The estimate \eqref{eq-shp:ShpDerEsti} has been proved by \eqref{eq-shp:MainEstiLim}.
\end{proof}

\medskip

\noindent\textbf{Acknowledgements.} The authors would like to warmly thank Jean-Marc Delort, Carlos Kenig and Hui Zhu, for numerous conversations concerning the problems studied here, as well as the 
Institut des Hautes \'Etudes Scientifiques (IHES) for its hospitality at the beginning of this research, which benefited from its stimulating scientific environment.

T.A. is partially
supported by the BOURGEONS project, grant ANR-23-CE40-0014-01 of the French National
Research Agency (ANR). 

\printbibliography[heading=bibliography,title=References]
\end{document}